\documentclass[11pt]{article} 

\usepackage{geometry}
\usepackage{microtype}

\usepackage{amsmath,amsfonts,amssymb,amsthm}
\usepackage{mathtools}
\usepackage{mathrsfs}
\usepackage{physics}
\usepackage{extarrows}

\usepackage{graphicx}
\usepackage{float}
\usepackage{subcaption}
\usepackage{pdfpages}

\usepackage{tikz}
\usepackage{tikz-cd}

\usepackage[ruled,vlined]{algorithm2e}

\usepackage{longtable}
\usepackage{booktabs}
\usepackage{colortbl}
\usepackage{adjustbox}

\usepackage{enumerate}
\usepackage{enumitem}
\usepackage{comment}

\usepackage{xcolor}
\usepackage[colorinlistoftodos]{todonotes}
\presetkeys{todonotes}{color=blue!20,inline}{}

\usepackage{authblk}

\usepackage[
  colorlinks=true,
  linkcolor=blue,
  citecolor=blue,
  urlcolor=blue
]{hyperref}
\usepackage{cleveref}

\crefname{figure}{\textbf{Figure}}{\textbf{Figures}}
\crefname{section}{\textbf{Section}}{\textbf{Sections}}
\crefname{subsection}{\textbf{Section}}{\textbf{Sections}}
\crefname{subsubsection}{\textbf{Section}}{\textbf{Sections}}
\crefname{appendix}{\textbf{Appendix}}{\textbf{Appendices}}
\crefname{equation}{\textbf{Equation}}{\textbf{Equations}}
\crefname{remark}{\textbf{Remark}}{\textbf{Remarks}}
\crefname{example}{\textbf{Example}}{\textbf{Examples}}
\Crefname{example}{\textbf{Example}}{\textbf{Examples}}
\crefname{theorem}{\textbf{Theorem}}{\textbf{Theorems}}
\crefname{definition}{\textbf{Definition}}{\textbf{Definitions}}
\crefname{lemma}{\textbf{Lemma}}{\textbf{Lemmas}}
\Crefname{lemma}{\textbf{Lemma}}{\textbf{Lemmas}}
\crefname{proposition}{\textbf{Proposition}}{\textbf{Propositions}}
\crefname{corollary}{\textbf{Corollary}}{\textbf{Corollaries}}
\crefname{table}{\textbf{Table}}{\textbf{Tables}}

\theoremstyle{definition}
\newtheorem{theorem}{Theorem}[section]

\newtheorem{lemma}[theorem]{Lemma}
\newtheorem{proposition}[theorem]{Proposition}
\newtheorem{corollary}[theorem]{Corollary}
\newtheorem{definition}[theorem]{Definition}
\newtheorem{example}[theorem]{Example}

\theoremstyle{remark}
\newtheorem{remark}[theorem]{Remark}

\newcommand{\TODO}[3]{%
  \hbox to 0pt{\textcolor{#1}{$^\bullet$}}%
  \marginpar{\footnotesize \textcolor{#1}{\begin{flushleft}#2: #3\end{flushleft}}}}

\newcommand {\mm}[1] {\ifmmode{#1}\else{\mbox{\(#1\)}}\fi}

\newcommand{\Rspace}        {\mm{{\mathbb R}}}

\newcommand{\Zbb}        {\mm{\mathbb{Z}}}
\newcommand{\Lcal}        {\mm{{\mathcal L}}}
\newcommand{\Bcal}        {\mm{{\mathcal B}}}

\newcommand{\Hom}           {\mm{\mathrm{Hom}}}
\newcommand{\Hgroup}           {H}

\newcommand{\diag}        {\mm{\mathrm{diag}}}

\newcommand{\kernel}        {\mm{\mathrm{ker}}\,}
\newcommand{\image}        {\mm{\mathrm{im}}\,}

\newcommand{\adjoint}{\ast}
\newcommand{\dual}{\vee}

\newcommand{\Rcal}{\mathcal{R}}
\newcommand{\ResistanceDistance}[1]{\hat D_{#1}}
\usepackage{tikz}

\newcommand{\caB}{\mathcal{B}}
\newcommand{\chainOne}{\alpha}
\newcommand{\chainTwo}{\beta}
\newcommand{\ERoperator}{\mathcal{T}}
\newcommand{\nullity}{\mathrm{null}}

\newcommand{\cochainLup}[1]{\mathcal{L}_{\mathrm{up}}^{#1}}
\newcommand{\cochainLdown}[1]{\mathcal{L}_{\mathrm{down}}^{#1}}
\newcommand{\chainLup}[1]{\mathcal{L}^{\mathrm{up}}_{#1}}
\newcommand{\chainLdown}[1]{\mathcal{L}^{\mathrm{down}}_{#1}}

\newcommand{\Lupmatrix}[2]{L_{#1,\mathrm{up}}^{\mathrm{#2}}}

\newcommand{\TopSubset}{\mathfrak{t}} 
\newcommand{\FaceSubset}{\mathfrak{s}} 
\newcommand{\TopTree}{K_{\TopSubset}} 
\newcommand{\FaceTree}{K_{\overline{\FaceSubset}}} 

\usepackage{longtable}
\usepackage{booktabs}
\usepackage{xcolor}
\usepackage{colortbl}
\usepackage{authblk}
\usepackage{adjustbox}
\usepackage[normalem]{ulem}

\title{Effective Resistance in Simplicial Complexes:\\Generalizations and Properties}

\author[1]{Inés García-Redondo\thanks{Corresponding author: \href{mailto:ines.garciaredondo@unifr.ch}{ines.garciaredondo@unifr.ch}}}
\author[2]{Claudia Landi}
\author[3]{Sarah Percival}
\author[4]{Anda Skeja}
\author[5]{Bei Wang}
\author[6]{\\Ling Zhou}

\affil[1]{Artificial Intelligence for Data-Oriented Science Lab, University of Fribourg, Switzerland}
\affil[2]{DISMI, University of Modena and Reggio Emilia, Italy}
\affil[3]{Department of Mathematics and Statistics, University of New Mexico, USA}
\affil[4]{Department of Mathematics, Uppsala University, Sweden}
\affil[5]{Kahlert School of Computing and SCI Institute, University of Utah, USA}
\affil[6]{Department of Mathematics, University of Oklahoma, USA}

\begin{document}
\date{}
\maketitle

\begin{abstract}
Originally introduced in circuit theory, the notion of effective resistance has been fruitfully exploited in graph theory, where each edge in a graph is understood as a resistor in an electric network. In this context, the effective resistance between two vertices quantifies the total opposition to current flow when a unit current is injected at one vertex and extracted at the other. As such, it encodes rich structural and geometric information about the underlying graph: it satisfies a minimum-energy principle termed Thomson's principle, defines a metric on the vertex set, relates to the topology of the graph through Foster's theorem, and determines the probability of appearance of an edge in a random spanning tree. Generalizations of effective resistance to simplicial complexes have been proposed in several forms, often formulated as matrix products of standard operators associated with the complex.

We introduce a novel, basis-independent bilinear form, derived from an algebraic reinterpretation of circuit theory, that extends the classical effective resistance from graphs to simplicial complexes. This bilinear form acts on chains and cochains, and in particular, simplices and cosimplices in a simplicial complex. Our framework subsumes and unifies all existing matrix-based formulations of effective resistance. Moreover, we establish higher-order analogues of several fundamental properties known in the graph case: (i) we extend Thomson's principle to our basis-free framework; (ii) we establish metric properties for the effective resistance and its square root on chains and cycles, recovering the classical graph case; (iii) we provide a generalization of Foster's theorem to simplicial complexes; and (iv) we prove a matrix-tree theorem for higher-dimensional spanning trees, linking effective resistance to the probability of appearance of a simplex in random simplicial spanning trees.

\end{abstract}

\newpage
\tableofcontents

\section{Introduction}
\label{sec:introduction}


Notions from circuit theory have been successfully implemented in graph theory by modeling edges in a graph as resistors in an electrical network, and using edge weights to encode their resistances.
Under this paradigm, the effective resistance between two vertices reflects the total resistance to the flow of electrical current when one unit of current is injected at one vertex and extracted at the other, reflecting both the direct and indirect paths between the vertices and incorporating the overall structure of the graph.
As such, it enjoys a number of important geometric and algebraic properties; see for instance a recent review by Devriendt~\cite{Devriendt2022b}. 
To name a few examples relevant to this work, the effective resistance can be characterized through the energy dissipated in the circuit, and satisfies a minimal energy principle in graph theory~\cite{ThomsonTait1875}. In addition, it determines a metric between the vertices of a graph~\cite{GvishianiGurvich1987,KleinRandic1993} and has deep connections to the topology of the graph through Foster's theorem~\cite{Foster1949}. Lastly, the effective resistance between the endpoints of an edge is proportional to the probability that such edge appears in a random spanning tree~\cite{Kirchhoff1847}. 
These properties have led to a plethora of applications of effective resistance in graph theory, network science, and machine learning. 
For instance, it has been used to measure centrality~\cite{BrandesFleischer2005,MavroforakisEtAl2015,Newman2005,TeixeiraEtAl2013,TeixeiraEtAl2016}, characterize chemical networks~\cite{BabicKleinLukovits2002}, study graph sparsification~\cite{SpielmanSrivastava2011},  perform node clustering (community detection)~\cite{AlevAnariLau2018} and label propagation~\cite{OstingPalandeWang2020}.   

Several distinct formulations have been proposed to extend effective resistance to simplicial complexes. Kook and Lee~\cite{KookLee2018} introduced a definition for simplicial complexes with positive resistances assigned to their top-dimensional faces. Osting et al.~\cite{OstingPalandeWang2020} developed a different formulation to generalize graph sparsification to simplicial complexes. Their approach applies to simplices of arbitrary dimension, including top-dimensional faces. Black and Maxwell~\cite{BlackMaxwell2021} considered the effective resistance of cycles in a simplicial complex, using yet another formulation. Despite their differences, all three approaches define effective resistance through products of standard matrix operators associated with the simplicial complex.

In this paper, we build on these efforts to develop a basis-independent definition of effective resistance that recovers the matrix formulations above. We establish its theoretical properties and connections to classical results for graphs.
We begin with linear algebra preliminaries and their applications to chains and cochains (\cref{sec:background}), then review Laplacians of simplicial complexes (\cref{sec:laplacian}). An algebraic reinterpretation of circuit theory (\cref{sec:circuit-theory}) provides the foundation for the subsequent discussion of effective resistance in graphs and its matrix formulations (\cref{sec:resistance-graph}).
Our contributions are as follows:

\begin{itemize}[noitemsep,leftmargin=*]
\item We extend \emph{effective resistance} from edges in graphs to simplices, chains, and cochains in simplicial complexes through a bilinear form motivated by the physical principles of electrical circuits (see \cref{sec:ER simplicial bilinear-form}).

\item We establish a higher-dimensional analogue of Thomson's principle within our framework (see \cref{sec:thomson-principle}).

\item We show that the square root of effective resistance defines a pseudometric on chains and a metric on each class of chains modulo boundaries. On suitable finite sets of chains, we also construct metrics defined directly by the effective resistance, without taking its square root. These metrics recover the classical resistance metric on graph vertices (see \cref{sec:metric-properties}).

\item We generalize Foster's theorem from graphs to simplicial complexes, connecting the structure of a complex to its effective resistance (see \cref{sec:fosters}).

\item We prove a weighted simplicial matrix-tree theorem within our basis-independent framework, relating the product of the eigenvalues of the up Laplacian to the count of simplicial spanning trees. Using this theorem, we relate effective resistance to the probability that a simplex belongs to a random simplicial spanning tree (see \cref{sec:spanning-trees-and-er}).

\item We recover all previously mentioned matrix formulations of effective resistance~\cite{BlackMaxwell2021,KookLee2018,OstingPalandeWang2020} from our bilinear form through appropriate choices of bases for the chain or cochain spaces (see \cref{sec:ER-simplicial-matrix-form}).
\end{itemize}

\section{Preliminaries on Chain and Cochain Complexes}
\label{sec:background}

In this section, we review foundational definitions from linear algebra and illustrate their application to chain and cochain spaces. This establishes the essential framework and notations used throughout this work. We assume a basic understanding of simplicial homology, see~\cite{Munkres1984} for an introductory text and \cref{sec:notations} for a summary of notations.
 
\subsection{Linear Algebra Background and Notations}
\label{sec:linear-algebra}

In our work, all vector spaces considered are finite dimensional.
\paragraph{Duals.}
Let \(U\) be a vector space over a field \(\Bbbk\). The \emph{dual vector space} \(U^\dual\) is the space of linear forms on \(U\), that is, \(U^\dual \coloneqq \Hom_{\Bbbk}(U, \Bbbk)\). Given a linear map \(f\colon U \to V\), the \emph{dual map} \(f^\dual: V^\dual \to U^\dual\) is defined as \[f^\dual (\varphi) (u) \coloneqq \varphi (f(u)), \quad \forall \varphi \in V^\dual, \,\forall u \in U.\]
Let \(\mathcal{B}_U = \{u_1, \dots, u_n\}\) be a basis of \(U\). The \emph{dual basis} \(\mathcal{B}_U^\dual = \{u_1^\dual, \dots, u_n^\dual\}\) of $U^\dual$ is given by the linear forms \(u_i^\dual \in U^\dual\) acting on the vectors of $\Bcal_U$ according to the rule:
\[u_i^\dual(u_j) = \delta_{ij},\]
with \(\delta_{ij}\) the Kronecker delta. This defines an isomorphism \(U\cong U^\dual\) which is not canonical, that is, it depends on the choice of the basis \(\Bcal_U\).

\paragraph{Matrix representations.}
Given a linear map \(f\colon U \to V\) and bases \(\Bcal_U\) of \(U\) and \(\Bcal_V \) of \(V\), we denote by \([f]^{\Bcal_V}_{\Bcal_U}\) the matrix representation of \(f\) with respect to these bases, which will be a matrix of dimension \(\dim(V) \times \dim(U)\). 
For a bilinear form \(\varphi: U \times V \to \Rspace\), we denote its matrix representation by \([\varphi]_{\Bcal_U, \Bcal_V}\), omitting the basis of the target space \(\Rspace\), which is always taken to be \(\{1\}\). 

\paragraph{Adjoints.}
Assume that the vector spaces \(U\) and \(V\) have inner products \(\langle \cdot, \cdot \rangle_{U}\) and \(\langle \cdot, \cdot \rangle_{V}\), respectively. Given an operator \(f\colon U \to V\) between them, we define the \emph{adjoint operator}
$f^\adjoint\colon V \to U$
as the map satisfying
$\langle u, f^\adjoint(v)\rangle_U = \langle f(u), v\rangle_V$
for all \(u \in U\) and \(v \in V\).
The adjoint operator uniquely exists (see~\cite[page 249]{Taylor1958}) and its image and kernel satisfy the following property (see \cite[Theorem 5.1]{Lim2020}):  
\begin{equation}
\label{eq:adjoint-image}
    \image(f^\adjoint)=\image(f^\adjoint f) \quad \text{and} \quad \kernel(f^*) = \kernel(f f^*).
\end{equation}

\paragraph{Matrix representation of the adjoint map.}
Let \(f\colon U \to V\) be a linear map between the vector spaces \(U\) and \(V\) with inner products \(\langle \cdot, \cdot \rangle_{U}\) and \(\langle \cdot, \cdot \rangle_{V}\), respectively. Choose bases \(\Bcal_U\) and \(\Bcal_V\) in \(U\) and \(V\) respectively, and denote \([f]_{\Bcal_U}^{\Bcal_V}=A\), \([\langle \cdot, \cdot\rangle_U]_{\Bcal_U, \Bcal_U} = M\) and \([\langle \cdot, \cdot\rangle_V]_{\Bcal_V, \Bcal_V} = N\). It is easy to check that the matrix representation of the adjoint map \(f^\adjoint: V \to U\) with respect to these bases is then given by
\begin{equation}
    \label{eq:adjoint-matrix-expression}
    [f^\adjoint]^{\Bcal_U}_{\Bcal_V} = M^{-1} A^\top N.
\end{equation}

\paragraph{Pseudoinverse operator.}
The \emph{pseudoinverse} generalizes the notion of the inverse of a linear operator to operators that may not be invertible. The concept was originally introduced for matrices by Moore~\cite{Moore1920} and Penrose~\cite{Penrose1955, Penrose1956} as the \emph{Moore--Penrose pseudoinverse matrix}, and later extended to bounded linear operators between Hilbert spaces~\cite{DesoerWhalen1963}.

Let \(\varphi : X \to Y\) be a linear operator between finite-dimensional Hilbert spaces \((X,\langle\cdot,\cdot\rangle_X)\) and \((Y,\langle\cdot,\cdot\rangle_Y)\).
The \emph{pseudoinverse} of \(\varphi\) is the unique linear operator \(\varphi^+:Y\to X\) satisfying the conditions
\[
\varphi\varphi^+\varphi=\varphi,\qquad
\varphi^+\varphi\varphi^+=\varphi^+,\qquad
(\varphi\varphi^+)^*=\varphi\varphi^+,\qquad
(\varphi^+\varphi)^*=\varphi^+\varphi .
\]

The pseudoinverse always exists for bounded linear operators on finite-dimensional Hilbert spaces and is uniquely determined \cite[Theorem 1]{DesoerWhalen1963}. It satisfies several useful properties; in particular,
\(
(\varphi^+)^+ = \varphi,
\)
and its kernel and image are characterized by
\begin{equation}
\label{eq:image-kernel-pseudoinverse}
\image(\varphi^+)=\image(\varphi^*)=\kernel(\varphi)^\perp,
\qquad
\kernel(\varphi^+)=\kernel(\varphi^*)=\image(\varphi)^\perp .
\end{equation}
Moreover,
\begin{equation}
\label{eq:identity-pseudoinvers-and-operator}
\varphi^+ = \varphi^+ \varphi \varphi^+,
\qquad \text{and} \qquad\varphi = \varphi \varphi^+ \varphi .
\end{equation}
The operators \(\varphi^+\varphi\) and \(\varphi\varphi^+\) are orthogonal projections onto \(\kernel(\varphi)^\perp\) and \(\image(\varphi)\), respectively (see Corollaries~1–3 in \cite{DesoerWhalen1963}). 
Finally, if \(\psi\) is an isomorphism preserving the inner product, then the pseudoinverse is invariant under conjugation:
\begin{equation}
\label{eq:pseudoinverse-invariant-under-conjugation}
    (\psi\varphi\psi^{-1})^+ = \psi\,\varphi^+\,\psi^{-1}.
\end{equation}

\paragraph{Matrix representation of the pseudoinverse operator.}
The pseudoinverse of a linear operator is consistent with the Moore--Penrose pseudoinverse of its matrix representation. In particular, let \(\mathcal{B}_X\) and \(\mathcal{B}_Y\) be orthogonal bases of \(X\) and \(Y\) with Gram matrices \(G_X\) and \(G_Y\), and let \(A\) be the matrix representation of \(\varphi\) in these bases. Then the matrix representation of \(\varphi^+\) is
\begin{equation}
\label{eq:pseudoinverse-matrix-expresion}
    [\varphi^+]_{\mathcal{B}_Y}^{\mathcal{B}_X}
= G_X^{-1/2}\!\left(G_Y^{1/2} A G_X^{-1/2}\right)^{+}\! G_Y^{1/2}.
\end{equation}
In particular, when the bases are orthonormal (\(G_X=G_Y=I\)), this reduces to the usual matrix pseudoinverse \(A^+\).\\

\subsection{Chain and Cochain Inner Product Spaces}
\label{sec:linear-algebra-chains-cochains}

Let $K$ be a $d$-dimensional finite simplicial complex; $K_p$ denote its set of $p$-simplices, $0 \leq p \leq d$; and $n_p = |K_p|$. We pick an \emph{orientation} for $K$ by choosing an ordering of its vertices.

\paragraph{Chain inner product space.} 
Let $C_p(K)\coloneqq C_p(K; \Rspace)$ denote the \emph{chain group} with real coefficients, defined as the vector space spanned by the oriented simplices in $K_p$. 
Chain groups are endowed with \emph{boundary maps} $\partial_p : C_p(K) \to C_{p-1}(K)$.
We assume that \(K\) is \emph{weighted} with a \emph{weight function} \(w_p: K_p \to \Rspace_{>0}\), $0 \leq p \leq d$, that we use to endow the chain spaces with inner products \(\langle \cdot, \cdot\rangle_{C_p} \colon C_p(K) \times C_p(K) \to \Rspace\) defined as
\[
\langle \sigma, \tau\rangle_{C_p} \coloneqq
    \begin{cases}
        w_p(\sigma)^{-1}, & \text{if }\sigma = \tau, \\
        0, & \text{otherwise}; 
    \end{cases}
\]
for \(p\)-simplices $\sigma, \, \tau \in K_p$, and extended linearly to general chains. 

We denote by $\partial_p^\adjoint$ the adjoint of the boundary operator $\partial_p$ with respect to the chosen inner product, which define the chain complex:
\begin{align}
\label{eq:chain_complex}
\cdots \ \ 
C_{p+1}(K) 
\underset{\partial_{p+1}}{\overset{\partial^\adjoint_{p+1}}{\leftrightarrows}}
C_{p}(K) 
\underset{\partial_p}{\overset{\partial^\adjoint_p}{\leftrightarrows}}
C_{p-1}(K) 
\ \ \cdots
\end{align}
The kernel $\ker(\partial_p)$ 
is called the $p$-th \emph{cycle group}, and its elements are called the $p$-\emph{cycles}.
The image $\image(\partial_{p+1})$ 
is called the $p$-th \emph{boundary group}, and its elements are called the $p$-\emph{boundaries}.
The $p$-th homology group of $K$, denoted $\Hgroup_p(K)$, is defined as the quotient of the $p$-th cycle group by the $p$-th boundary group.
The dimension of $\Hgroup_p(K)$ is called the $p$-th \emph{Betti number}, denoted $\beta_p(K)$. 

\paragraph{Cochain inner product space.} 
Similarly, we can define \emph{cochain spaces} \(C^p(K)\coloneqq C_p(K; \Rspace)^\dual\) as the dual vector spaces of the chain spaces. The \emph{coboundary maps} \(\delta_p: C^p(K)\to C^{p+1}(K)\) are then defined as the duals of the boundary maps: \(\delta_p \coloneqq \partial_{p+1}^\dual\). 
We also have an inner product in the space of cochains
\(\langle \cdot, \cdot\rangle_{C^p} \colon C^p(K) \times C^p(K) \to \Rspace\)
defined as
\begin{equation}\label{eq:inner product cochain}
\langle \sigma^\dual, \tau^\dual\rangle_{C^p} \coloneqq
    \begin{cases}
        w_p(\sigma), & \text{if }\sigma = \tau, \\
        0, & \text{otherwise}; 
    \end{cases}
\end{equation}
for all \(\sigma, \tau \in K_p\) and extended linearly.

We denote by $\delta^\adjoint_{p}$ the adjoint of the coboundary operator $\delta_p$ with respect to the chosen inner product, which defines the cochain complex:
\begin{align}
\label{eq:cochain_complex}
\cdots \ \ 
C^{p+1}(K) 
\underset{\delta^\adjoint_p}{\overset{\delta_p}{\leftrightarrows}}
C^{p}(K) 
\underset{\delta^\adjoint_{p-1}}{\overset{\delta_{p-1}}{\leftrightarrows}}
C^{p-1}(K) 
\ \ \cdots
\end{align}
The kernel of $\delta_p$ is called the $p$-th \emph{cocycle group}, and its elements are called the $p$-\emph{cocycles}.
The image of $\delta_{p-1}$ is called the $p$-th \emph{coboundary group}, and its elements are called the $p$-\emph{coboundaries}.
The $p$-th cohomology group of $K$, denoted $H^p(K)$, is defined as the quotient of the $p$-th cocycle group by the $p$-th coboundary group.
With real coefficients, the cohomology groups are simply the dual vector spaces of the homology groups \(\Hgroup^p (K) = \Hom_{\Rspace} \left(\Hgroup_p(K), \Rspace\right)\). Consequently, the dimension of $H^p(K)$ coincides with the $p$-th Betti number.

For notational uniformity, in degree \(0\) we set
\(
    C_{-1}(K)=C^{-1}(K)=\{0\}
\)
and take \(\partial_0\) and \(\delta_{-1}\) to be the zero maps. 

\paragraph{Musical isomorphisms.} As noted in \cref{sec:linear-algebra}, the correspondence between dual bases yields a non-canonical isomorphism between dual spaces. In contrast, the inner product on \(C_p(K)\) defined above induces the musical (flat and sharp) maps, providing canonical isomorphisms \(C_p(K) \cong C^p(K)\) (cf.~\cite{VazDaRocha2016}).

\begin{definition}[Musical isomorphisms]
\label[definition]{def:flat-sharp iso}
    The \emph{flat isomorphism} \( \flat_p \colon C_p(K) \to C^p(K) \) sends each chain $\chainOne \in C_p(K)$ to the linear form 
    \[\flat_p(\chainOne) \coloneqq \langle \chainOne, \cdot \rangle_{C_p}.\] 
    Conversely, the \emph{sharp isomorphism} \( \sharp_p \colon C^p(K) \to C_p(K) \) sends each cochain \( f \in C^p(K) \) to the unique chain \( \chainOne \in C_p(K) \) such that
    \[
    f(\chainTwo) = \langle \chainOne, \chainTwo \rangle_{C_p}, \quad \forall \chainTwo \in C_p(K).
    \]
\end{definition}

By construction, $\sharp_p = \flat_p^{-1}$. Indeed, since $\flat_p (\chainOne)(\chainTwo) = \langle\chainOne, \chainTwo \rangle_{C_p} \in \Rspace$ for all $\chainTwo\in C_p(K)$, and since $\sharp_p( \flat_p(\chainOne))$ is a chain $\chainOne'$ such that $\flat_p(\chainOne)(\chainTwo) = \langle \chainOne', \chainTwo\rangle_{C_p}$ for all $\chainTwo \in C_p(K)$, one sees that $\chainOne=\chainOne'$. 
For notational simplicity, we will often write $\flat_p^{-1}$ instead of $\sharp_p$.

\begin{lemma}\label[lemma]{lem:music}
The flat isomorphism
\(\flat_p \colon C_p(K) \to C^p(K)\) 
is an \emph{orthogonal map}, that is, it preserves the inner product:
\[
\langle \chainOne, \chainTwo\rangle_{C_p}
= \langle \flat_p(\chainOne), \flat_p(\chainTwo)\rangle_{C^p}, \quad \forall \chainOne, \chainTwo \in C_p(K).
\]
Similarly, the sharp isomorphism
\(\sharp_p \colon C^p(K) \to C_p(K)\) 
is orthogonal, and the two are adjoint to each other:
\(
\flat_p^* = \sharp_p 
\)
and \(
\sharp_p^* = \flat_p.
\)
\end{lemma}

\begin{proof}
For any simplex $\sigma \in K_p$, we have 
\(
\flat_p(\sigma) = \langle \sigma, \cdot \rangle_{C_p}
= w_p(\sigma)^{-1}\sigma^\dual
\)
by definition. Thus, for any two simplices $\sigma, \tau \in K_p$,
\[
\langle \flat_p(\sigma), \flat_p(\tau)\rangle_{C^p}
= w_p(\sigma)^{-1} w_p(\tau)^{-1} 
\langle \sigma^\dual, \tau^\dual \rangle_{C^p}
= 
\begin{cases}
w_p(\sigma)^{-1}, & \sigma = \tau,\\
0, & \text{otherwise}.
\end{cases}
\]
This coincides with $\langle \sigma, \tau\rangle_{C_p}$.
By bilinearity of the inner product and linearity of $\flat_p$,
the equality extends to all $\chainOne, \chainTwo \in C_p(K)$,
so $\flat_p$ preserves inner products. A similar argument holds for $\sharp_p$.

Since $\sharp_p = \flat_p^{-1}$ and both maps are isometries, we have $\sharp_p^* = \flat_p$ and, equivalently, $\flat_p^* = \sharp_p$. 
Indeed, for any $f \in C^p(K)$ and $\alpha \in C_p(K)$,
\[
\langle f, \sharp_p^* \alpha\rangle_{C^p}
= \langle \sharp_p f, \alpha\rangle_{C_p}
= \langle \flat_p \sharp_p f, \flat_p \alpha\rangle_{C^p}
= \langle f, \flat_p \alpha\rangle_{C^p},
\]
where the first equality is the definition of the adjoint, the second uses that $\flat_p$ preserves inner products, and the last follows from $\flat_p = \sharp_p^{-1}$.
\end{proof}

\begin{lemma}
\label[lemma]{lem:commutativity of operators}
    The following diagrams commute:
    \begin{equation}\label{eq:combined diagram}
    \begin{tikzcd}
        \,
        & C_{p}(K) 
        \arrow[l, phantom, "\cdots"] 
        \arrow[d, "\flat_{p}"', shift right=2] \arrow[r, "\partial_{p}"] 
        & C_{p-1}(K)  
        \arrow[d, "\flat_{p-1}"', shift right=2] 
        & \, \arrow[l, phantom, "\cdots"] \\
        \, 
        & C^{p}(K) 
        \arrow[l, phantom, "\cdots"] \arrow[u, "\sharp_{p}"', shift right] \arrow[r, "\delta_{p-1}^*"', shift right]  
        & C^{p-1}(K) \arrow[u, "\sharp_{p-1}"']           
        & \, \arrow[l, phantom, "\cdots"] 
    \end{tikzcd}
 \text{and}
    \begin{tikzcd}
        \,
        & C_{p}(K) \arrow[l, phantom, "\cdots"] \arrow[d, "\flat_{p}"', shift right=2] 
        & C_{p-1}(K) \arrow[l, "\partial_p^*"'] 
        \arrow[d, "\flat_{p-1}"', shift right=2] 
        & \, \arrow[l, phantom, "\cdots"] \\
        \, \arrow[r, phantom, "\cdots"] 
        & C^{p}(K) \arrow[u, "\sharp_p"', shift right] 
        & C^{p-1}(K) \arrow[l, "\delta_{p-1}"] \arrow[u, "\sharp_{p-1}"']           
        & \, \arrow[l, phantom, "\cdots"] 
    \end{tikzcd}
    \end{equation}
    In particular, \(
    \flat_{p-1} \circ \partial_p = \delta^{\adjoint}_{p-1} \circ \flat_p\) and 
    \(
    \delta_{p-1} \circ \flat_{p-1} =  \flat_p \circ \partial_p^\adjoint \).
\end{lemma} 

\begin{proof}
It suffices to show that \(\delta_{p-1} = \flat_p \circ \partial_p^* \circ \flat_{p-1}^{-1}\). Indeed, since the coboundary map is the \emph{dual} of the boundary map, for \(f \in C^{p-1}(K)\) and \(\alpha \in C_p(K)\),
\begin{align*}
(\flat_p \partial_p^* \flat_{p-1}^{-1} f)(\alpha)
&= \langle \partial_p^*(\flat_{p-1}^{-1} f),\, \alpha \rangle_{C_p} & \text{(def.~of }\flat_p\text{)} \\
&= \langle \flat_{p-1}^{-1} f,\, \partial_p \alpha \rangle_{C_{p-1}} & \text{(def.~of adjoint map)} \\
& = f(\partial_p \alpha) & \text{(def.~of }\flat_{p-1}^{-1}\text{)}\\
&= (\partial_p^\dual f)(\alpha) & \text{(def.~of dual map)} \\
&= (\delta_{p-1} f)(\alpha). & \text{(def.~of }\delta_{p-1}\text{)}
\end{align*}
showing that \(\delta_{p-1} = \flat_p \circ \partial_p^* \circ \flat_{p-1}^{-1}\).
\end{proof} 

\begin{remark}
In our exposition, we first equip both chain and cochain spaces with inner products and then show that the musical isomorphisms preserve them. This is equivalent to endowing only the chain space with an inner product, constructing the flat isomorphism from it, and inducing the corresponding inner product on the cochain space via the push-forward through this isomorphism. 
\end{remark}

\subsection{Bases in the Spaces of Chains and Cochains}
\label{sec:bases-chains-cochains}

In a \(d\)-dimensional simplicial complex \(K\), the set of simplices provides a standard choice of basis for the chain groups, which in turn induces a dual basis for the cochain groups.

\begin{definition}[Standard bases]\label[definition]{def:standard basis}
    For \(0 \leq p \leq d\), the standard basis in the space of \emph{chains} \(C_p(K)\), is defined by the ordered set of \(p\)-simplices in \(K\), that is,
    \[\Bcal_p \coloneqq \left\{ \sigma_1, \dots, \sigma_{n_p} \right\}.\]
    The standard basis in the space of \emph{cochains} is defined as the dual basis, that is,
    \[\Bcal^p \coloneqq \left\{ \sigma_1^\dual, \dots, \sigma_{n_p}^\dual \right\}.\]
\end{definition}

\begin{remark}
We already implicitly consider this choice of basis when we define the inner products $\langle\cdot, \cdot\rangle_{C_p}$ and $\langle\cdot, \cdot\rangle_{C^p}$ in \cref{sec:linear-algebra-chains-cochains}. These are defined so that the standard bases are \emph{orthogonal bases}.
\end{remark}

By normalizing the bases above, we obtain the following orthonormal bases. 
\begin{definition}[Orthonormal bases]
\label[definition]{def:orthonormal basis}
    For \(0 \leq p \leq d\), the orthonormal basis in the space of \emph{chains} \(C_p(K)\) with respect to \(\langle \cdot,\cdot \rangle_{C_p}\) is given by
    \[    \widetilde{\Bcal}_p \coloneqq \left\{\sqrt{w_p(\sigma_1)}\,  \sigma_1,\, \dots, \, \sqrt{w_p(\sigma_{n_p})}\, \sigma_{n_p} \right\}.\]
    Similarly, the orthonormal basis in the space of \emph{cochains} with respect to \(\langle \cdot,\cdot \rangle_{C^p}\) is given by
    \[    \widetilde{\Bcal}^p \coloneqq \left\{\dfrac{\sigma_1^\dual}{\sqrt{w_p(\sigma_1)}}  , \dots, \dfrac{\sigma_{n_p}^\dual}{\sqrt{w_p(\sigma_{n_p})}} \right\}.\]
\end{definition}

\begin{remark} 
The ``standard'' and ``orthonormal'' bases serve different roles. 
The standard bases depend only on the combinatorial structure of \(K\), while the orthonormal bases also depend on the choice of inner products, or equivalently, on the weights \(w_p\). Consequently, neither choice is strictly canonical once a weighting scheme is introduced; each is convenient in a different context. For this reason, we avoid referring to either as canonical.
\end{remark}

\subsection{Matrix Representations}
\label{sec:matrix-representations}
We can now provide matrix representations for all the operators defined in \cref{sec:linear-algebra-chains-cochains} with respect to the bases in \cref{sec:bases-chains-cochains}.

\paragraph{Inner product matrices and changes of basis.}
In the standard bases the inner product matrices in the spaces of chains and cochains are respectively
\[[\langle\cdot, \cdot\rangle_{C_p}]_{\Bcal_p, \Bcal_p} = W_p^{-1}, \qquad \text{and} \qquad [\langle\cdot, \cdot\rangle_{C^p}]_{\Bcal^p, \Bcal^p} = W_p ,\]
where $W_p \coloneqq \diag(w_p(\sigma_1), \dots , w_p(\sigma_{n_p}))$ denotes the diagonal matrix defined by the weights of the $p$-simplices. In the orthonormal bases, by definition, these matrices will be the identity.
The change of basis matrix between the standard and orthonormal basis is given by 
\[[\mathrm{Id}]^{\Bcal_p}_{\widetilde{\Bcal}_p} = W_p^{1/2}, \qquad \text{and} \qquad [\mathrm{Id}]^{\Bcal^p}_{\widetilde{\Bcal}^p} = W_p^{-1/2},\]
in the space of chains and cochains, respectively.

\paragraph{Boundary and coboundary operators.}
Let $B_p\coloneqq[\partial_p]^{\Bcal_{p-1}}_{\Bcal_p}$ be the matrix representation of the boundary operator with respect to the standard bases, which, in some sources is called the \emph{incidence matrix}. With respect to the orthonormal bases, we have  
\[[\partial_p]^{\widetilde{\Bcal}_{p-1}}_{\widetilde{\Bcal}_p} = W_{p-1}^{-1/2} B_p W_{p}^{1/2}.\]
\begin{lemma}\label[lemma]{lem:coboundary matrix representation}
    The matrix representations for the coboundary operator are given by \[[\delta_p]^{\Bcal^{p+1}}_{\Bcal^p} = B_{p+1}^\top, \quad \text{and}\quad[\delta_p]^{\widetilde{\Bcal}^{p+1}}_{\widetilde{\Bcal}^p} = W_{p+1}^{1/2} B_{p+1}^\top W_p^{-1/2}.\]
\end{lemma}
\begin{proof}
Let $\sigma_i^\dual \in \Bcal^{p}$, with $1 \leq i \leq n_{p}$, and $\sigma_j' \in \Bcal_{p+1}$, with $1 \leq j \leq n_{p+1}$. Then, the value at position $(j, i)$ in $[\delta_p]^{\Bcal^{p+1}}_{\Bcal^p}$ is precisely given by $\delta_p(\sigma_i^\dual) (\sigma_j')$. By definition, we know that 
\[\delta_p(\sigma_i^\dual) (\sigma_j') = \sigma_i^\dual\left(\partial_{p+1}(\sigma'_j)\right),\] 
the left-hand side being precisely the value of $[\partial_{p+1}]^{\Bcal_{p}}_{\Bcal_{p+1}}$ at position $(i,j)$, proving the first representation in the lemma. The second follows after change of basis.
\end{proof}

Applying \cref{eq:adjoint-matrix-expression}, we get the matrix representations of the adjoint of the boundary map: 
\begin{equation}\label{eq:adjoint of boundary matrix representation}
    [\partial_p^\adjoint]_{\Bcal_{p-1}}^{\Bcal_{p}} = W_p B_{p}^\top W_{p-1}^{-1}, \quad \text{and} \quad [\partial_p^\adjoint]_{\widetilde{\Bcal}_{p-1}}^{\widetilde{\Bcal}_{p}} = W_p^{1/2} B_{p}^\top W_{p-1}^{-1/2} = \left([\partial_p]^{\widetilde{\Bcal}_{p-1}}_{\widetilde{\Bcal}_p}\right)^\top,
\end{equation}
and the matrix representations of the adjoint of the coboundary map: 
\begin{equation}\label{eq:adjoint of coboundary matrix representation}
    [\delta_p^\adjoint]_{\Bcal^{p+1}}^{\Bcal^p} = W_p^{-1} B_{p+1} W_{p+1}, \quad \text{and} \quad [\delta_p^\adjoint]_{\widetilde{\Bcal}^{p+1}}^{\widetilde{\Bcal}^p} = W_p^{-1/2} B_{p+1} W_{p+1}^{1/2} = \left([\delta_p]^{\widetilde{\Bcal}^{p+1}}_{\widetilde{\Bcal}^p}\right)^\top, 
\end{equation}
where the second representation is to be expected, as in general the matrix representation of the adjoint of a linear map with respect to orthonormal basis equals the transpose of the matrix representation of the original map. 

\paragraph{Musical isomorphisms.}
For each simplex $\sigma \in K_p$, the flat map sends $\sigma$ to $w_p(\sigma)^{-1}\sigma^\dual$, and the sharp map sends $\sigma^\dual$ to $w_p(\sigma)\,\sigma$.
Thus, with respect to the standard bases we have that the musical isomorphisms have the following matrix representations
\begin{equation}
    \label{eq:matrix-musical-standard}
    [\flat_p]_{\Bcal_p}^{\Bcal^p} = W_p^{-1}, \qquad \text{and} \qquad [\sharp_p]_{\mathcal{B}^p}^{\Bcal_p} = W_p.
\end{equation}

We can also compute the adjoint of the flat isomorphism, using \cref{eq:adjoint-matrix-expression,eq:matrix-musical-standard}:
\[[\flat_p^\adjoint]_{\Bcal^p}^{\Bcal_p} = W_p W_p^{-1} W_p = W_p = [\flat_p^{-1}]_{\mathcal{B}^p}^{\Bcal_p}.\]

\section{Laplacians of Graphs and Simplicial Complexes}
\label{sec:laplacian} 

We review Laplacians of simplicial complexes and obtain the graph Laplacian as a particular case of this formulation.

\subsection{Laplacians as Operators}
\label{sec:laplacian-operator}

In this paper, following~\cite{HorakJost2013}, we work with the combinatorial \emph{up Laplace  operator}, \emph{down Laplace operator}, and \emph{Hodge Laplace  operator}. As operators on the \(p\)-cochain space, they are defined as 
\begin{align}
    \cochainLup{p}(K) & \coloneqq  \delta_p^\adjoint \delta_p;\label{eq:up-laplacian}\\
    \Lcal_{\textrm{down}}^p(K) & \coloneqq \delta_{p-1} \delta_{p-1}^\adjoint; 
    \notag\\
    \Lcal_{\textrm{Hodge}}^p(K) & \coloneqq \cochainLup{p}(K) +  \cochainLdown{p}(K).  
    \notag
\end{align}
When the underlying simplicial complex $K$ is clear from the context, we will omit it from the notations. These Laplace operators are self-adjoint, non-negative (therefore positive semi-definite) and compact~\cite{HorakJost2013}. As a consequence, their  eigenvalues are non-negative reals and their eigenspaces are orthogonal and sum (directly) to the whole space.

Analogously, as operators on \(p\)-chain spaces, the combinatorial \emph{up Laplace  operator}, \emph{down Laplace operator}, and \emph{Hodge Laplace  operator} are defined as
\begin{align}
    \chainLup{p}(K) & \coloneqq \partial_{p+1}\partial^\adjoint_{p+1}
    ;\label{eq:up-laplacian-chain}\\
    \chainLdown{p}(K) & \coloneqq \partial^\adjoint_{p}\partial_{p}; 
    \notag\\
    \Lcal^{\textrm{Hodge}}_p(K) & \coloneqq \chainLup{p}(K) + \chainLdown{p}(K). 
    \notag
\end{align}
Again, we omit $K$ from these notations when the context is clear. 

\begin{lemma}
\label[lemma]{lem:L pseudoinverse chains cochains}
The chain and cochain up Laplacians, as well as their pseudoinverses, are related via the musical isomorphisms: 
\begin{enumerate}
    \item \(\chainLup{p} = \flat_{p}^{-1} \, \cochainLup{p} \, \flat_{p}\);
    \item  \( (\cochainLup{p})^+ 
    = \flat_{p} \, (\chainLup{p})^+ \, \flat_{p}^{-1}.\)
\end{enumerate}
\end{lemma}

\begin{proof}
    The first item follows directly from \cref{lem:commutativity of operators} and the definition of the up Laplacians.
    
    The second item is a consequence of \cref{eq:pseudoinverse-invariant-under-conjugation}, which asserts the compatibility of pseudoinverses under inner-product-preserving isomorphisms.
\end{proof}

\begin{remark}
  Analogous statements for the down Laplacians hold but are omitted here, since they are not needed in this paper.  
\end{remark}

\paragraph{Hodge theory.}  
By standard results from linear algebra in finite-dimensional vector spaces, one obtains the canonical identifications:
\begin{equation}
    H^p(K; \Rspace) \;\cong\; H_p(K; \Rspace) \;\cong\; \ker \bigl(\Lcal^p_{\textrm{Hodge}}(K)\bigr).
\end{equation}
In particular, the $p$-th (co)homology group can be realized concretely as the space of \emph{harmonic $p$-(co)chains}, i.e., those annihilated by the Hodge Laplacian.  

Moreover, the (co)chain space admits the \emph{Hodge decomposition}:
\begin{equation}
\label{eq:Hodge_decomposition}
    C_p(K; \Rspace) 
    \;=\; \image(\partial_{p}^\adjoint) \;\oplus\; \ker\bigl(\Lcal_p^{\textrm{Hodge}}(K)\bigr) \;\oplus\; \image(\partial_{p+1}), 
\end{equation}
where the three summands are mutually orthogonal with respect to the inner product. In particular,
\[
(\image \partial_p^\adjoint)^\perp \;=\; \ker \bigl(\Lcal_p^{\textrm{Hodge}}(K)\bigr) \oplus \image (\partial_{p+1}),
\quad\text{and}\quad
(\image \partial_{p+1})^\perp \;=\; \image (\partial_p^\adjoint) \oplus \ker \bigl(\Lcal^p_{\textrm{Hodge}}(K)\bigr).
\]

\subsection{Laplacians as Matrices}
\label{sec:laplacian-matrix}

We focus on up Laplace operators from now on. In the standard basis (\cref{def:standard basis}), it follows from \cref{lem:coboundary matrix representation} and \cref{eq:adjoint of boundary matrix representation,eq:adjoint of coboundary matrix representation} that the matrix representations of the cochain up Laplacians $\cochainLup{p}$ and chain up Laplacian $\chainLup{p}$ are respectively given by
\begin{align}
    [\cochainLup{p}]_{\caB^p}^{\caB^p} &
    =W_p^{-1} B_{p+1} W_{p+1} B_{p+1}^\top \eqqcolon \Lupmatrix{p}{},  
    \label{eq:L_up}\\
    [\chainLup{p}]_{\caB_p}^{\caB_p} &
    =B_{p+1} W_{p+1} B_{p+1}^\top W_p^{-1} = W_p \Lupmatrix{p}{}W_p^{-1}.  \label{eq:L_up_chain}
\end{align} 
The ones for $\cochainLdown{p}$ and $\chainLdown{p}$ are similarly defined. These expressions recover the \emph{combinatorial Hodge Laplacian operators} defined by Horak and Jost~\cite{HorakJost2013}, which encompass several previous approaches to define Laplace operators in simplicial complexes (e.g.~\cite{duval2002shifted,friedman1996computing,muhammad2006control}).

In the orthonormal basis (\cref{def:orthonormal basis}), it follows similarly that the matrix representations of the cochain/chain up Laplacians $\cochainLup{p}$ and $\chainLup{p}$ are respectively given by
\begin{align}
    [\cochainLup{p}]_{\tilde\Bcal^p}^{\tilde\Bcal^p} 
        & 
        = W_{p}^{-1/2} B_{p+1} W_{p+1} B_{p+1}^\top W_{p}^{-1/2} 
        = W_p^{-1/2} \Lupmatrix{p}{} W_p^{-1/2},  
        \label{eq:L_up_orthonormal}\\
    [\chainLup{p}]_{\tilde\caB_p}^{\tilde\caB_p} &
        = [\cochainLup{p}]_{\tilde\Bcal^p}^{\tilde\Bcal^p}. 
        \label{eq:L_up_chain_orthonormal} 
\end{align} 
These are now symmetric matrices, as expected, and together with their corresponding down versions, they recover the normalized Hodge Laplacian operators; see \cite[Definition 4.1]{ponoi2024normalized}. 
Interestingly, normalized and unnormalized Hodge Laplacian matrices are often treated as distinct objects in the literature. Our framework shows that they are simply different manifestations of the same operator under different basis choices. This highlights the unifying power of the approach we promote here, which reveals seemingly distinct operators as alternative representations of a single, basis-independent one. 

\paragraph{Graph Laplacians as matrices.}
\label{sec:laplacian-graph}

Let \(G = (V, E)\) be a finite, undirected, simple graph (no self-loops or multiple edges), where \(V\) is the set of vertices with \(|V| = n_0\), and \(E\) the set of edges with \(|E| = n_1\). 
We adopt this convention throughout and subsequently refer to graphs simply as $G$. 
The classical \emph{graph Laplacian} \(L \in \Rspace^{n_0 \times n_0}\) is $L \coloneqq B_1 W_1 B_1^\top$, which coincides with the up Laplacian \(\Lupmatrix{0}{}\) in \cref{eq:L_up} for \(p = 0\) and \(W_0 = \mathrm{Id}\). 

Equivalently, if \(w_{uv}=w_{vu}=w(e)\) for the edge \(e=\{u,v\}\), with \(w_{uv}=0\)
for nonadjacent vertices, then the entries of the graph Laplacian are given by
\begin{equation}\label{eq:graph-laplacian-entries}
    L_{uv}=\sum_{e\in E}w(e)(B_1)_{ue}(B_1)_{ve}=
    \begin{cases}
        \displaystyle\sum_{x\in V\setminus\{u\}}w_{ux}, & u=v,\\
        -w_{uv}, & u\ne v.
    \end{cases}
\end{equation}

\section{Algebraic Formulation of Circuit Theory}
\label{sec:circuit-theory}

We now review an algebraic reinterpretation of circuit theory. We refer to \cite{BaezFong2018Compositional,BambergSternberg1988Course} for full details on the exposition in \cref{sec:circuit-theory-graphs}. The contents in \cref{sec:circuit-theory-higher-dims} are original, although similar ideas can be found in \cite{Catanzaro2015}.

\subsection{Algebraic Formulation of Circuit Theory in Graphs}
\label{sec:circuit-theory-graphs}

Let \(G\) be a finite oriented graph, representing an electrical circuit. We select subsets \( V_{\mathrm{in}} \subseteq V \) and \( V_{\mathrm{out}} \subseteq V \) as the sets of \emph{input} and \emph{output} vertices, where current is injected and extracted respectively. Their union \(\partial G \coloneqq V_{\mathrm{in}} \cup V_{\mathrm{out}}\) is called the set of \emph{terminals}. The remaining vertices are called \emph{interior} or \emph{non-terminal}.

To incorporate physical intuition into our model, we define weights based on the physical properties of the circuit. We see each edge \( e \in E \) as a resistor with \emph{resistance} \( \rho_e > 0 \) (not to be confused with the \emph{effective resistance} \( r_e \)). Then, we model the circuit \(G\) as a weighted 1-dimensional simplicial complex in which vertices have unit weight, and edges are weighted by the \emph{inverse} of their resistances, also called \emph{conductances}, \( w(e) = \rho_e^{-1} \).
This weighting scheme allows us to define inner products on \( C_1(G) \) and \(C_0(G)\), as described in \cref{sec:bases-chains-cochains}. Note that, with these definitions, \( \langle e, e \rangle_{C_1} = w(e)^{-1} = \rho_e.\)



The governing physical quantities of an electrical circuit are the \emph{voltage}, which can be described by a 1-cochain $f \in C^1(G)$, and the \textit{current}, described by a 1-chain $\chainOne \in C_1(G)$. These are subject to several physical laws (see \cite[Ch. 12]{BambergSternberg1988Course} or \cite[Sec. 2.2]{BaezFong2018Compositional}) that we summarize below. Let \(\iota\) denote the inclusion \(\iota \colon C_0(\partial G) \to C_0(G)\) of chain spaces.
\begin{enumerate}[label=(\arabic*)]
    \item \textbf{Kirchhoff's voltage law} (KVL) states that the sum of voltages along any closed loop in the circuit is zero. 
    Algebraically, this means that the voltage \( f \in C^1(G) \) is a coboundary, i.e.,
    \[
    f = \delta_0 \phi
    \]
    for some (not necessarily unique) \( \phi \in C^0(G) \), called the \emph{potential}. 
    Equivalently, \( f \) evaluates to zero on every closed path (that is, every \(1\)-cycle).

    \item \textbf{Kirchhoff's current law} (KCL) expresses conservation of current: the net current vanishes at interior vertices and may be non-zero only at terminals.
    Algebraically, this implies that a current \( \chainOne \in C_1(G) \) satisfies
    \[
    \partial_1 \chainOne = \iota \chainTwo
    \]
    for some \( \chainTwo \in C_0(\partial G) \), called the \emph{boundary current}. 
    When \( \partial G = \emptyset \), the circuit is closed and this condition reduces to \( \partial_1 \chainOne = 0 \).

    \item \textbf{Ohm's law} (OL) states that the voltage across a resistor is proportional to the current through it, with proportionality constant given by the resistance. 
    Since resistances define the inner product on \( C_1(G) \), Ohm’s law can be written using the flat isomorphism (see \cref{def:flat-sharp iso}) as
    \[
    f = \flat_1 \chainOne,
    \]
    where \( f \in C^1(G) \) is the voltage and \( \chainOne \in C_1(G) \) is the current.
\end{enumerate}
In what follows, we assume that all voltages and currents are {\em physically valid}, that is, they are cochains and chains, respectively, that satisfy (KVL), (KCL) and (OL). 

\begin{remark}[Interpretation of Kirchhoff's current law]\label[remark]{rem:kcl-weak-form}
We explain why Kirchhoff's current law, i.e.~conservation of current, can be expressed as \( \partial_1 \chainOne  = \iota\chainTwo \) for some \( \chainTwo \in C_0(\partial G) \), where \( \chainOne  \in C_1(G) \). 

Expressing the current as a linear combination of the standard basis of edges \(\alpha = \sum_{e \in E} \alpha_e \, e\), we can compute its boundary and evaluate at some vertex $x \in V$ to obtain its component over that vertex. Since we have unit weights in the vertices, this amounts to computing
\begin{equation}
    \label{eq:remark-kirchhof-current-law}
    \langle \partial_1 \alpha, x\rangle_{C_0} = \sum_{(u, x) \in E} \alpha_e - \sum_{(x, v) \in E} \alpha_e. 
\end{equation}

Kirchhoff's current law states that the input current must equal the output current in all internal vertices, i.e., \cref{eq:remark-kirchhof-current-law} vanishes if $x \in V$ is internal, while there is a positive excess, denoted $\beta_x>0$ at $x \in V_{\mathrm{in}}$ corresponding to the injected current, and a negative deficiency, $\beta_x<0$ at $x \in V_{\mathrm{out}}$ corresponding to the extracted current. This information can be collected in a 0-chain $\beta=\sum_{x \in \partial G} \beta_x \, x$, defined only on the terminals. We can extend it to all vertices via $\iota \beta = \sum_{x\in V} \beta_x\, x$ setting $\beta_x = 0$ whenever $x$ is an internal vertex. With this,
\begin{equation}
    \label{eq:remark-kirchhof-current-law-2}
    \langle \iota \beta, x \rangle_{C_0} =  \beta_x.
\end{equation}
Putting together \cref{eq:remark-kirchhof-current-law} and \cref{eq:remark-kirchhof-current-law-2}, we obtain that for all $x \in V$,
\[\langle \partial_1 \alpha, x\rangle_{C_0} = \langle \iota\beta, x \rangle_{C_0},\]
which implies \( \partial_1 \chainOne  = \iota\chainTwo \). We emphasize that for a different choice of the inner product on the vertices, Kirchhoff's current law would not hold in its usual formulation. Thus, the choice of unit weights on the vertices is necessary to retrieve Kirchhoff's current law in its classical form.
\end{remark}

\subsection{Extending Circuit Theory to Higher Dimensions}
\label{sec:circuit-theory-higher-dims}

The algebraic reinterpretation of circuit laws facilitates their generalization to higher dimensional simplicial complexes.
With the following proposition, we formulate this extension and characterize when chains and cochains in arbitrary dimension can be interpreted as currents, voltages, or potentials. Previous work \cite{Catanzaro2015} developed a similar result to obtain insights related to Reidemeister torsion. However, our setting differs in that voltages are cochains rather than chains, and thus we use musical isomorphisms to relate them. 

\begin{proposition}
\label[proposition]{prop:potential_and_voltage_law}
    Given  \( \chainTwo \in C_{p-1}(K) \), the system of equations
    \begin{equation}\label{eq:circuit-system}
   \left\{\begin{array}{lcc}
   \delta_{p-1}\phi=f&\quad & \text{(KVL)}\\
   \partial_p\alpha= \beta&\quad & \text{(KCL)}\\
   \flat_p\alpha=f &\quad & \text{(OL)}
   \end{array}\right. 
       \end{equation}
    in the unknowns \(\chainOne \in C_{p}(K)\),  \(f\in C^{p}(K)\), and \(\phi \in C^{p-1}(K)\) admits solutions if and only if the equation 
    \begin{equation}\label{eq:laplacian_potential_boundary_current}
      \cochainLup{p-1} \phi = \flat_{p-1}\beta 
    \end{equation} in the unknown $\phi$ does. When solutions exist for \cref{eq:laplacian_potential_boundary_current}, the solution of minimal norm is given by $\phi = (\cochainLup{p-1})^+  \flat_{p-1} \beta.$
\end{proposition}

\begin{proof}
 Let us first assume that  \(\chainOne \in C_{p}(K)\),  \(f\in C^{p}(K)\), and \(\phi \in C^{p-1}(K)\) exist that satisfy the system for a given \(\beta\). Then, (OL) and the fact that \(\flat_p\) is an isomorphism imply that \(\alpha= \flat_p^{-1}f\). Hence, (KVL) implies  \(\alpha=\flat_p^{-1}\delta_{p-1}\phi\). By (KCL), we have \(\beta=\partial_p\flat_p^{-1}\delta_{p-1}\phi\) that, by the commutativity of the diagram 
 \begin{equation}\label{eq:commutative}
     \begin{tikzcd}[column sep=large]
    C^{p-1}(K) 
    & C^{p}(K) \arrow[l, "\delta^{\adjoint}_{p-1} " above, shift left=0.6ex] 
    \\
    C_{p-1}(K) \arrow[u, "\flat_{p-1}" left]   
    & C_{p}(K), \arrow[l, "\partial_{p}"] 
    \arrow[u, "\flat_{p}" right] 
    \end{tikzcd}
 \end{equation}
 is equivalent to \(\beta=\flat_{p-1}^{-1}\delta^*_{p-1}\delta_{p-1}\phi\), that is, $\cochainLup{p-1} \phi = \flat_{p-1}\beta$.

Conversely, let us assume that  \(\phi \in C^{p-1}(K)\) exists that satisfies the equation \(\cochainLup{p-1} \phi = \flat_{p-1}\beta\) with $\beta$ given. Then, reversing the argument above, \(\beta=\partial_p\flat_p^{-1}\delta_{p-1}\phi\). So, taking \(\alpha=\flat_p^{-1}\delta_{p-1}\phi\) we get a solution for equation (KCL); taking $f=\flat_p\alpha$ we get a solution for equation (OL); moreover, this choice of \(f\) and $\alpha$ also satisfies (KVL) because \(f=\flat_p\alpha=\flat_p\flat_p^{-1}\delta_{p-1}\phi=\delta_{p-1}\phi\). 

Suppose that solutions exist in either case. If $\phi$ solves $\cochainLup{p-1}\phi = \flat_{p-1}\beta$, it can be decomposed as $\phi = \phi_0+ \phi_1$ with $\phi_0 \in \ker\cochainLup{p-1}$ and $\phi_1 \in (\ker \cochainLup{p-1})^\perp = \image \cochainLup{p-1}$. From basic linear algebra, $\phi_1$ is also a solution of the equation
\[ \flat_{p-1} \beta = \cochainLup{p-1}\phi = \cochainLup{p-1}(\phi_0+\phi_1) = \cochainLup{p-1} \phi_1  \]
with lower norm
\[\|\phi\|^2 = \|\phi_0 + \phi_1\|^2 = \|\phi_0\|^2 + \|\phi_1\|^2 \geq \|\phi_1\|^2\]
where the second equation uses $\langle \phi_0, \phi_1\rangle =0$. Since \(\varphi^+\varphi\) is an orthogonal projection onto \(\kernel(\varphi)^\perp=\image(\varphi^\adjoint)\) for any operator $\varphi$, it follows that 
\[\phi_1 = (\cochainLup{p-1})^+\,\cochainLup{p-1}\phi_1 = (\cochainLup{p-1})^+  \flat_{p-1} \beta\]
which concludes the proof.
\end{proof}

\paragraph{Valid boundary currents.}
We now show that the name \textit{boundary current} carries meaning: a $(p-1)$-chain \(\beta \in C_{p-1}(K)\) is a \emph{physically valid boundary current}, i.e.~the system defined by the circuit laws in \cref{eq:circuit-system} in the unknowns $\alpha$, $\phi$, and $f$, admits a solution, precisely when $\beta$ is a boundary, $\beta \in \image \partial_p$. In particular, since boundaries always exist, this proves that the system \cref{eq:circuit-system} always has solutions for a fixed boundary current.

\begin{corollary}\label[corollary]{coro:voltage}
    For $p\ge 1$, let  $\chainTwo\in C_{p-1}( K)$ be a boundary,  that is, Kirchhoff's current law holds: there is some $\chainOne\in C_p(K)$ such that $\partial_p \chainOne  = \chainTwo$. Let $f\in C^p(K)$ be obtained by applying Ohm's law: $\flat_p\alpha=f$. Then $f$ is a voltage, i.e.,~it satisfies Kirchhoff's voltage law: there exists a potential $\phi\in C^{p-1}(K)$ such that $\delta_{p-1}(\phi)=f$. 
\end{corollary}

\begin{proof}
    By the assumptions and the commutativity of Diagram \eqref{eq:commutative}, we have $\flat_{p-1}\beta=\flat_{p-1}\partial_p\alpha=\delta_{p-1}^*\flat_p\alpha=\delta_{p-1}^*f$. 
    Hence, $\flat_{p-1}\beta\in \mathrm{ im}(\delta_{p-1}^*)$. By \cref{eq:adjoint-image}, $\mathrm{ im}(\delta_{p-1}^*)=\mathrm{ im}(\cochainLup{p-1})$. So, there exists $\phi\in C^{p-1}(K)$ such that $\cochainLup{p-1}\phi=\flat_{p-1}\beta$, and this $\phi$ must be a potential of $f$.
\end{proof}

In graph theory, boundary currents are more commonly described as \textit{balanced 0-chains}, that is $\beta = \sum_{v \in V} \beta_v \, v$ with $\sum_{v \in V} \beta_v = 0$. This characterization agrees with the previous result because graphs are typically assumed to be connected. In that case, $H_0(G) = 0$ which means $\image \partial_1 = \kernel \tilde \partial_0$. Since \(\tilde\partial_0\bigl(\sum_{v\in V}\beta_v\,v\bigr)=\sum_{v\in V}\beta_v\), the kernel of \(\tilde\partial_0\) is precisely the space of balanced \(0\)-chains.

This perspective allows us to generalize the notion of the \emph{set of terminals} to a subcomplex \(\partial K \subseteq K\). We require \(\partial K\) to share the same \((p-2)\)-skeleton as \(K\),
\(
(\partial K)_{\leq p-2}=K_{\leq p-2},
\)
while its \((p-1)\)-simplices form a subset of those of \(K\),
\(
(\partial K)_{p-1}\subseteq K_{p-1}.
\)

The \((p-1)\)-simplices of \(\partial K\) determine the support of the boundary current. Therefore, by the characterization above, we additionally require this support to consist of boundaries, namely
\(
C_{p-1}(\partial K)\subseteq \operatorname{im}\partial_p.
\)

\paragraph{Measuring differences of potential.}
The following lemma provides a necessary condition for a potential solving \cref{eq:laplacian_potential_boundary_current}, which will be useful to derive a formula to compute difference of potentials in the algebraic framework (\cref{coro:difference_potential}).

\begin{lemma}
\label[lemma]{lemma:inner-product-potential-boundary-current}
If $\phi \in C^{p-1}(K)$ is a solution of $\cochainLup{p-1}\phi=\flat_{p-1}\beta$ for some $\beta\in C_{p-1}(K)$, then for any $\gamma\in \image {\partial}_{p}(K)$, we have
\[\langle\phi, \flat_{p-1} \gamma\rangle = \langle (\cochainLup{p-1})^+\flat_{p-1}\beta,\flat_{p-1}\gamma\rangle_{C^{p-1}} .\]
\end{lemma}

\begin{proof}
Since $\cochainLup{p-1}$ is self-adjoint, using the orthogonal decomposition
\[
C^{p-1}(K)=\ker\cochainLup{p-1}\oplus \mathrm{im}\,\cochainLup{p-1},
\]
we write $\phi=\phi_0+\phi_1$ with $\phi_0\in\ker\cochainLup{p-1}$ and $\phi_1\in\mathrm{im}\,\cochainLup{p-1}$.
Since $\cochainLup{p-1}\phi=\flat_{p-1}\beta$, we have
\[
\cochainLup{p-1}\phi_1=\flat_{p-1}\beta .
\]
Because $\phi_1\in\mathrm{im}\,\cochainLup{p-1}$, and $(\cochainLup{p-1})^+\cochainLup{p-1}$ is an orthogonal projection onto this space, we get 
\[
\phi_1=(\cochainLup{p-1})^+\cochainLup{p-1}\phi_1
      =(\cochainLup{p-1})^+\flat_{p-1}\beta .
\]

Substituting in the inner product expression of our interest, we obtain
\[
\langle\phi,\flat_{p-1}\gamma\rangle
=
\langle\phi_0,\flat_{p-1}\gamma\rangle
+
\langle\phi_1,\flat_{p-1}\gamma\rangle .
\]
Injecting the latest expression for $\phi_1$ in the second summand of the right hand side of the equality above yields
\[
\langle\phi,\flat_{p-1}\gamma\rangle
=
\langle\phi_0,\flat_{p-1}\gamma\rangle
+
\langle (\cochainLup{p-1})^+\flat_{p-1}\beta,\flat_{p-1}\gamma\rangle .
\]
We thus only need to prove that the first summand vanishes to conclude the proof. From \cref{eq:combined diagram}, we observe that $\gamma \in \image\partial_{p} = \image (\flat_{p-1}^{-1} \delta_{p-1}^\adjoint \flat_p)$ means that 
\[ \flat_{p-1} \gamma \in \image \delta_{p-1}^\adjoint \flat_p \subset \image \delta_{p-1}^\adjoint.\]
On the other hand, from \cref{eq:adjoint-image}, we have $\phi_0 \in \ker \cochainLup{p-1}=(\image \cochainLup{p-1})^\perp = (\image \delta_{p-1}^\adjoint )^\perp$, and hence 
\[\langle\phi_0,\flat_{p-1}\gamma\rangle = 0\]
as desired.
\end{proof}

\begin{corollary}[Measuring differences of potentials in the algebraic framework.]
\label[corollary]{coro:difference_potential}
    Let $G$ be a connected graph, and $\phi \in C^0(G)$ a solution of $\cochainLup{0} \phi = \flat_0 \beta$ for $\beta= u-v \in C_0(\partial G)$. Then, for any two vertices $x,y \in V$, we have
    \[\phi(x) - \phi(y) = \langle (\chainLup{0})^+ \beta, \gamma \rangle_{C_0} \]
    with $\gamma = x-y \in C_0(G)$.
\end{corollary}

\begin{proof}
Notice that, since $G$ is connected, \[H_0(G; \mathbb{R}) = \dfrac{\ker \tilde\partial_0}{\image \partial_1} = 0\]
and thus $\ker \tilde\partial_0 = \image \partial_1$.
The claim follows from the following sequence of equalities:
    \begin{align*}
    \phi(x) - \phi(y) & = \phi(\gamma)   & (\gamma = x-y \in \ker \tilde{\partial}_0 = \image \partial_1)\\
                      & = \langle \flat_{0}^{-1} \phi, \gamma \rangle_{C_0} & (\text{definition of }\flat_0^{-1} = \sharp_0) \\
                      & = \langle \phi , \flat_0 \gamma \rangle_{C^0} & (\text{definition of adjoint, }\flat_0^* = \flat_0^{-1}) \\
                      & = \langle (\cochainLup{0})^+ \flat_0\beta, \flat_0 \gamma\rangle_{C^0} & (\text{\cref{lemma:inner-product-potential-boundary-current}}) \\
                      & = \langle (\chainLup{0})^+ \beta, \gamma \rangle_{C_0} & ((\chainLup{0})^+= \flat_0^{-1}\, (\cochainLup{0})^+ \, \flat_0,\ \cref{lem:L pseudoinverse chains cochains}).
\end{align*}  
\end{proof}

\section{Effective Resistance Matrices in Graphs}
\label{sec:resistance-graph}
 
In this section, we introduce the physical definition of effective resistance and then review its graph-theoretic formulation in terms of effective resistance matrices.

\paragraph{Effective resistance in circuit theory.} 
The \emph{effective resistance} between two vertices $u$ and $v$ in an electrical network is defined as the total opposition to current flow when a current $I_{uv}$ is injected at $u$ and extracted at $v$. Equivalently, if the entire network were replaced by a single resistor connecting $u$ and $v$, its resistance would equal the effective resistance between these vertices.

By Ohm's law, the effective resistance can be expressed as the potential difference between $u$ and $v$ divided by the injected current. It is customary to consider a \emph{unit current} injection (i.e., $I_{uv}=1$), so that the effective resistance is simply the resulting potential difference between $u$ and $v$.

An alternative characterization is given by \emph{Thomson's principle}. Among all flows that route one unit of current from $u$ to $v$, the physical electrical flow, that is, the one satisfying Kirchhoff's laws, minimizes the total energy dissipation in the network. From the expression for this energy, it follows that the minimal dissipated energy associated with a unit current from $u$ to $v$ is equal to the effective resistance between these vertices.

In \cref{sec:ER simplicial bilinear-form}, we reformulate these ideas within our algebraic framework. For the moment, we recall their matrix interpretation, as it is commonly presented in the graph theory literature.

\paragraph{Effective resistance of edges.}
Let $G$ be a connected graph modeling an electric circuit as described in \cref{sec:circuit-theory-graphs}. Recall that we have unit weights in the vertices and positive weights in the edges (the conductances), collected in the diagonal matrix $W_1 \in \Rspace^{n_1 \times n_1}$. 
Spielman and Srivastava~\cite{SpielmanSrivastava2011} defined the effective resistance of edges via the matrix $R \in \Rspace^{n_1 \times n_1}$ given by
\begin{align}\label{eq:graph_resistance}
R \coloneqq B_1^\top L^{+} B_1 = B_1^\top (B_1 W_1 B_1^\top)^+ B_1,
\end{align} 
where $L^+$ denotes the Moore-Penrose inverse of $L$.
We refer to $R$ as the \emph{edge-based effective resistance matrix}.
The diagonal entries of $R$ give the \emph{effective resistance} of edges $e \in E$:
\begin{equation}\label{eq:ER of edges}
r_e \coloneqq R(e,e).
\end{equation}

\paragraph{Effective resistance between vertices.}
Let $\mathbf{1}_v \in \Rspace^{n_0}$ denote the coordinate vector of vertex $v$ with respect to the standard basis given by the vertices of the graph, i.e. the vector with a 1 in the $v$-th position and zeros elsewhere. 
The \emph{effective resistance (distance)} between vertices \(u\) and \(v\) (see~\cite{KleinRandic1993}) is given by
\begin{equation}
\label{eq:devriendtER}
r_{uv} \coloneqq (\mathbf{1}_u - \mathbf{1}_v )^{\top} L^{+} (\mathbf{1}_u - \mathbf{1}_v).
\end{equation}

The \emph{vertex-based effective resistance matrix} is defined as the matrix of order $n_0 \times n_0$ with elements $\left(r_{uv}\right)_{u,v\in V}$. 
It is a classical fact that if $e$ is the edge connecting vertices $u$ and $v$, then $r_e = r_{uv}$.

\section{Effective Resistance in Simplicial Complexes}  \label{sec:ER simplicial bilinear-form}

We now present one of our main contributions: a reinterpretation of the effective resistance matrices in \cref{sec:resistance-graph} as bilinear forms, by rephrasing the physical definition of effective resistance in our algebraic framework for circuit theory. 

\subsection{Effective Resistance Bilinear Form} 
\label{sec:simplicial-bilinear}

Let $G$ be a connected graph modeling a resistor network as described in \cref{sec:circuit-theory-graphs}, where vertices have unit weights and edges are weighted by the conductances of the corresponding resistors.

\paragraph{Basis-free reinterpretation of the effective resistance between nodes.}
Recall that the effective resistance between vertices $u$ and $v$ is the potential difference $\phi(u)-\phi(v)$ produced by injecting a unit current at $u$ and extracting it at $v$. In the algebraic framework of \cref{sec:circuit-theory}, this is prescribed by the boundary current $\chainTwo = u-v \in C_0(\partial G) \subset C_0(G)$ with $u \in V_{\mathrm{in}}$ and $v \in V_{\mathrm{out}}$. By \cref{coro:difference_potential}, the resulting potential difference between \(u\) and \(v\) is given by
\begin{equation}
\label{eq:effective_resistance_nodes}
\tilde{r}_{uv} \coloneqq
\langle (\chainLup{0})^+ \beta, \beta \rangle_{C_0}.
\end{equation}

We now verify that $\tilde r_{uv}$ coincides with the standard effective resistance $r_{uv}$ from graph theory in~\cref{eq:devriendtER}.

\begin{proposition}\label[proposition]{prop:vertex-bases matrix of caR}
    Consider a connected graph \( G \) and vertices \( u, v \in V \). Then
    \[
        \tilde{r}_{uv} = r_{uv},
    \]
    where \( r_{uv} \) is the standard effective resistance between vertices $u$ and $v$ from \cref{eq:devriendtER}.
\end{proposition}

\begin{proof}
    From \cref{eq:devriendtER}, the effective resistance between \( u \) and \( v \) is given by
    \begin{align*}
        r_{uv} 
        & = (\mathbf{1}_u - \mathbf{1}_v )^{\top} L^{+} (\mathbf{1}_u - \mathbf{1}_v) & (\cref{eq:devriendtER}) \\
        & = \big\langle (\cochainLup{0})^+ (\flat_0^{ G}\iota u - \flat_0^{ G} \iota v),\, (\flat_0^{ G}\iota u - \flat_0^{ G}\iota v ) \big\rangle_{C^0} & (\mathbf{1}_u \text{ is the vector representing } \flat_0^{ G}\iota u)\\
        & = \big\langle (\cochainLup{0})^+ \flat_0^{G} \beta ,\, \flat_0^{ G}\beta \big\rangle_{C^0} &  (\text{linearity of } \flat_0^{ G})\\ 
        & = \big\langle (\chainLup{0})^+\beta ,\, \beta \big\rangle_{C_0} & \text{(\cref{lem:L pseudoinverse chains cochains})}\\ 
        & = \tilde{r}_{uv} & (\cref{eq:effective_resistance_nodes}).
    \end{align*}
    Hence, \( \tilde{r}_{uv} = r_{uv} \), as claimed.
\end{proof}

So far, we have been measuring differences of potential between the same vertices where we are injecting and extracting current. A different generalization is obtained if we want to measure the opposition to the current between vertices $x$ and $y$ when a unit current is injected at vertex $u$ and extracted from vertex $v$, with $\{u,v\}\ne \{x,y\}$. Still calling $\phi$ to the potential function at the vertices induced by the boundary current, we can evaluate the difference of potential between $x$ and $y$ using the same argument as before:
\[\phi(x)-\phi(y)=\big\langle (\chainLup{0})^+\beta ,\, \gamma \big\rangle_{C_0}\]
where $\beta=u-v$ and $\gamma=x-y$. This motivates our proposed use of a bilinear form, which in some cases captures information that a quadratic form alone does not.

\paragraph{Basis-free reinterpretation of the effective resistance for edges.}
Since $\beta\in\image\partial_1$ is a boundary current, we may choose a 1-chain $\alpha\in C_1(G)$ satisfying $\partial_1\alpha=\beta=u-v$, \cref{eq:effective_resistance_nodes} can be rewritten as
\begin{equation}
\label{eq:effective-resistance-bilinear-edges}
\tilde r_{uv}
= \langle (\chainLup{0})^+ \partial_1 \alpha,\,
    \partial_1 \alpha \rangle_{C_0} 
= \langle \partial_1^\adjoint (\chainLup{0})^+ \partial_1 \alpha,\,
    \alpha \rangle_{C_1}.
\end{equation}

As in the vertex case, a direct computation shows that 
\[
R(e, e') = \langle \partial_1^\adjoint (\chainLup{0})^+ \partial_1 e,\,
            e' \rangle_{C_1}, \quad e, e' \in E,
\]
with $R(e, e')$ the value in position $(e,e')$ in the edge-based effective resistance matrix, \cref{eq:graph_resistance}.
This will follow as a corollary of \cref{prop:matrix of caR_p}, which is stated in greater generality.

\paragraph{Effective resistance bilinear form in simplicial complexes.}
These connections with effective resistance matrices in graphs lead to the following definitions for higher-order simplices. \cref{eq:effective-resistance-bilinear-edges} inspires \cref{def:ER bilinear form in SC}, which extends the edge-based effective resistance to chains in arbitrary dimensions, while \cref{eq:effective_resistance_nodes} motivates \cref{def:ER-bilinear-form-boundaries}, generalizing the vertex-based effective resistance to boundaries in arbitrary dimensions.

\begin{definition}\label[definition]{def:ER bilinear form in SC}
    We define the \emph{effective resistance bilinear form} on the space of \emph{$p$-chains}, $C_p(K)$, as
    \begin{equation*}
        \begin{aligned}
                \Rcal_p\colon C_p(K) \times C_p(K) & \to \Rspace\\
              (\chainOne, \chainOne') &\mapsto  \Rcal_p(\chainOne, \chainOne') \coloneqq \big\langle \partial^\adjoint_p (\chainLup{p-1})^+ \partial_p \chainOne,   \chainOne' \big\rangle_{C_p}.
        \end{aligned}
    \end{equation*}
    The definition for \emph{$p$-cochains} is analogous: \(\Rcal^p(f, f') \coloneqq \big\langle \delta_{p-1}(\cochainLup{p-1})^+ \delta^{\adjoint}_{p-1} f, f' \big\rangle_{C^p}\).
\end{definition}

\begin{definition}
\label[definition]{def:ER-bilinear-form-boundaries}
    We define the \emph{effective resistance bilinear form} in the space of \emph{$p$-boundaries} $\image(\partial_{p+1})$ 
    as
        \begin{equation*}
        \begin{aligned}
                \hat{\mathcal{R}}_p\colon \image(\partial_{p+1}) \times \image(\partial_{p+1}) & \to \Rspace\\ 
              (\chainTwo, \chainTwo') &\mapsto  \hat{\Rcal}_p(\chainTwo, \chainTwo') \coloneqq \big\langle  (\chainLup{p})^+ \chainTwo,   \chainTwo' \big\rangle_{C_{p}}.
        \end{aligned}
    \end{equation*}
    The definition for cochains in $\image(\delta_p^\adjoint)$ is analogous: $\hat{\Rcal}^p(\phi,\phi') \coloneqq \langle (\cochainLup{p})^+ \phi, \phi'\rangle_{C^p}$.
\end{definition}

The effective resistance bilinear forms on chains and cochains are related via the musical isomorphism: applying \cref{lem:commutativity of operators} and \cref{lem:L pseudoinverse chains cochains} we obtain
\[\Rcal^p(f, f') = \big\langle \flat_{p} \partial^\adjoint_p (\chainLup{p-1})^+ \partial_p \flat_{p}^{-1} f, f' \big\rangle_{C^p}
= \Rcal_p\big(\flat_{p}^{-1}f,\, \flat_{p}^{-1}f'\big).\]
Similar arguments apply for the effective resistance for boundaries and their images under $\flat_p$.
Because of this correspondence, we restrict our exposition to the chain formulation.

In addition, we now prove that the effective resistance bilinear form in chains uniquely determines the bilinear form for boundaries and vice versa via pull back and push-forward relations.

\begin{lemma}[Equivalence of chain and boundary formulations]
\label[lemma]{lem:ER_chain_boundary_equivalence}
The effective resistance bilinear form on $(p+1)$-chains factors through the boundary operator. More precisely, for all $\alpha,\alpha' \in C_{p+1}(K)$,
\[
\mathcal R_{p+1}(\alpha,\alpha')
=
\hat{\mathcal R}_p(\partial_{p+1}\alpha,\partial_{p+1}\alpha').
\]
Consequently, $\mathcal R_{p+1}$ vanishes on $\ker(\partial_{p+1})$ 
and induces a unique bilinear form on the space of $p$-boundaries $\image(\partial_{p+1})$, 
namely $\hat{\mathcal R}_p$. In particular, the two formulations contain the same information.
\end{lemma}

\begin{proof}

By definition
\begin{align*}
    \mathcal R_{p+1}(\alpha,\alpha') & = \langle
\partial_{p+1}^{\adjoint}(\chainLup{p})^+\partial_{p+1}\alpha,
\alpha'
\rangle_{C_{p+1}} & (\cref{def:ER bilinear form in SC}) \\
& = \langle
(\chainLup{p})^+\partial_{p+1}\alpha,
 \partial_{p+1}\alpha' 
\rangle_{C_{p}} & (\text{definition of adjoint}) \\
 & = \hat{\Rcal}_p (\partial_{p+1} \alpha, \partial_{p+1} \alpha') & (\cref{def:ER-bilinear-form-boundaries})
\end{align*}
for all $\alpha, \alpha' \in C_{p+1}(K)$.

In particular, if $c\in\ker\partial_{p+1}$, then $\partial_{p+1}c=0$ and therefore
\[
\mathcal R_{p+1}(\alpha,c)=0
\qquad
\text{for all } \alpha\in C_{p+1}(K).
\]
Thus $\mathcal R_{p+1}$ depends only on the boundaries of its arguments and factors uniquely through the quotient
\[
C_{p+1}(K)/\ker(\partial_{p+1})\cong \image(\partial_{p+1}). 
\]
The induced bilinear form on $\image(\partial_{p+1})$ 
is precisely $\hat{\mathcal R}_p$, proving that the two formulations determine each other uniquely.
\end{proof}

\begin{example}
\label[example]{ex:injective-boundary} Suppose that \(\kernel\partial_{p+1}=\{0\}\), i.e., every \(p\)-boundary has a unique \((p+1)\)-chain filling. Since 
\(\partial_{p+1}^{\adjoint}(\chainLup{p})^+\partial_{p+1}\)
is the orthogonal projection onto
\((\kernel\partial_{p+1})^\perp=C_{p+1}(K)\) (see \cref{prop:properties of R operator} for more details), in this case, it is the identity. Thus, by \cref{lem:ER_chain_boundary_equivalence} and \cref{def:ER bilinear form in SC},
\[
\hat{\Rcal}_p(\partial_{p+1}\alpha,\partial_{p+1}\alpha')
=
\Rcal_{p+1}(\alpha,\alpha')
=\big\langle\partial_{p+1}^\adjoint(\chainLup{p})^+ \partial_{p+1}\alpha,\alpha'\big\rangle_{C_{p+1}}
=
\langle\alpha,\alpha'\rangle_{C_{p+1}}
\]
for all \(\alpha,\alpha'\in C_{p+1}(K)\).

In particular, if \(\{\sigma_1,\ldots,\sigma_m\}\) is the
standard simplex basis of \(C_{p+1}(K)\), then
\(\{\partial_{p+1}\sigma_1,\ldots,\partial_{p+1}\sigma_m\}\)
is an orthogonal basis of \(\image(\partial_{p+1})\) with
respect to \(\hat{\Rcal}_p\), with
\[
\hat{\Rcal}_p(\partial_{p+1}\sigma_i,\partial_{p+1}\sigma_j)
=
\frac{\delta_{ij}}{w(\sigma_i)}.
\]
\end{example}

In general, due to the lack of a standard choice of basis for the space of $p$-boundaries $\image(\partial_{p+1})$, 
we favor the effective resistance bilinear form on chains over the one defined over boundaries. Other authors have chosen otherwise, see \cite{BlackMaxwell2021}.

\paragraph{Why this formulation uses the up Laplacian only.} 

A natural question, in light of the effective resistance matrices introduced in \cref{eq:graph_resistance,eq:devriendtER} and the bilinear formulations of \cref{def:ER bilinear form in SC,def:ER-bilinear-form-boundaries}, is whether the effective resistance admits an alternative expression in terms of the down Laplacian or the full Hodge Laplacian.

Our derivation makes the answer clear: the appearance of the up Laplacian is not incidental but follows directly from the characterization of potentials and voltages in \cref{prop:potential_and_voltage_law}. In particular, this result allows differences of potential to be expressed in terms of the up Laplacian via \cref{coro:difference_potential}. By contrast, the algebraic formulation of the circuit laws provides no analogous characterization involving either the down Laplacian or the full Hodge Laplacian, and therefore does not directly lead to corresponding expressions for the effective resistance\footnote{A formal down-Laplacian counterpart of \cref{prop:potential_and_voltage_law} can be obtained by reversing the boundary and coboundary arrows and shifting degree: the system \(\delta_{p-1}^{\adjoint}\phi=f\), \(\partial_p^{\adjoint}\alpha=\beta\), and \(\flat_{p-1}\alpha=f\) is solvable if and only if \(\cochainLdown{p}\phi=\flat_p\beta\), with solution with minimal norm \(\phi=(\cochainLdown{p})^+\flat_p\beta\). This system involving adjoint operators, however, is purely algebraic and does not represent the circuit laws.}.

\paragraph{Effective resistance bilinear form via potentials.}
\label{para: ER form as potentials}
For any \(\chainTwo\in\image(\partial_{p+1})\), define
\[
    \phi_{\chainTwo}
    \coloneqq
    (\cochainLup{p})^+\flat_p\chainTwo
    \in C^p(K).
\]
By \cref{prop:potential_and_voltage_law}, with \(p\) replaced by \(p+1\),
\(\phi_{\chainTwo}\) is the potential with minimal norm generated by the boundary current \(\chainTwo\). Moreover, by
\cref{lem:L pseudoinverse chains cochains,lem:music}, for all
\(\chainTwo,\gamma\in\image(\partial_{p+1})\),
\begin{equation}\label{eq:general-boundary-potential-pairing-rem}
    \hat{\Rcal}_p(\chainTwo,\gamma)
    =
    \big\langle(\chainLup{p})^+\chainTwo,\gamma\big\rangle_{C_p}
    =
    \big\langle\phi_{\chainTwo},\flat_p\gamma\big\rangle_{C^p}
    =
    \phi_{\chainTwo}(\gamma).
\end{equation}
Thus, \(\hat{\Rcal}_p(\chainTwo,\gamma)\) is the potential generated by
the boundary current \(\chainTwo\), evaluated on the boundary \(\gamma\).

\cref{lem:ER_chain_boundary_equivalence} allows us to give the corresponding
potential interpretation of the effective resistance bilinear form on
chains. For all \(\chainOne,\chainOne'\in C_p(K)\), \(\Rcal_p(\chainOne,\chainOne')\) is the potential generated by
the boundary current \(\partial_p\chainOne\), evaluated on the boundary
\(\partial_p\chainOne'\).

\subsection{Effective Resistance of Chains and Cochains}
\label{sec:ef-chains-cochains}

We now use the effective resistance bilinear form in simplicial complexes to extend the notion of effective resistance from nodes and edges to general chains and cochains in arbitrary dimension. Note that, for this definition, it suffices to consider the \emph{quadratic form} defined by the bilinear forms above.

\begin{definition}[Effective resistance of chains and cochains]\label[definition]{def:ER_chains_cochains}
Let \( \alpha \in C_p(K) \) be a \( p \)-chain in a simplicial complex \( K \). 
The \emph{effective resistance} of \( \alpha \) is defined as
\[
    r_\alpha \coloneqq \mathcal{R}_p(\alpha, \alpha).
\]
The \emph{effective resistance} of a cochain $f \in C^p(K)$ is defined similarly \(r_f \coloneqq \mathcal{R}^p(f, f) = \Rcal_p\big(\flat_p^{-1} f, \flat_p^{-1} f\big).\)
\end{definition}

\begin{remark}
    Here, we choose to extend the notion of effective resistance to chains and cochains instead of simply to simplices since some of the results that we will derive in what follows require this generality. Some examples include Thomson's principle (\cref{sec:thomson-principle}) or the metric properties of the effective resistance bilinear form (\cref{sec:metric-properties}).  However, we note that previous extensions of effective resistance to simplicial complexes were actually restricted to simplices. Therefore, when we compare with them (\cref{sec:ER-simplicial-matrix-form}) we will need to restrict the quadratic form above to simplices too, but our scope goes beyond their formulation. 
\end{remark}

\begin{remark}\label[remark]{rmk:relative ER}
When \( f = \sigma^\dual \) is the dual cochain corresponding to a \( p \)-simplex \( \sigma \), we have:
\begin{enumerate}
    \item \(
    r_{\sigma^\dual} 
    = \mathcal{R}_p\big(\flat_p^{-1} \sigma^\dual,\, \flat_p^{-1} \sigma^\dual \big) 
    = \mathcal{R}_p(w(\sigma) \sigma, w(\sigma) \sigma) 
    = w(\sigma)^2 \, r_\sigma,
    \)
    \item \(
    r_{\sqrt{w(\sigma)} \, \sigma} =  w(\sigma)\, r_\sigma,
    \)
    \item \(
    r_{\tfrac{1}{\sqrt{w(\sigma)}}{\sigma^\dual}} 
       = w(\sigma)\, r_\sigma.
    \)
\end{enumerate}
In the graph case, letting \( \sigma = e \) be an edge, the quantity \( w(e)\, r_e \) has been termed the \emph{relative effective resistance} \cite{DevriendtLambiotte2022}. 
The above calculations show that the relative effective resistance of an edge \( e \) is equal to the effective resistance of either the normalized chain \( \sqrt{w(e)}\, e \) or the normalized cochain \( \tfrac{1}{\sqrt{w(e)}} e^\dual \).
\end{remark}

To facilitate the discussion of the properties of the effective resistances (e.g. \cref{cor:ER-postive-semi-definite}), we introduce the following linear operator on chains:
\begin{equation}\label{eq:ER operator chains}
    \ERoperator_p \coloneqq \partial_p^\adjoint \big( \chainLup{p-1} \big)^+ \partial_p = \partial_p^\adjoint \big( \partial_p^\adjoint \big)^+.
\end{equation}
By construction, $\Rcal_p(\chainOne, \chainOne') = \langle \ERoperator_p \chainOne, \chainOne' \rangle_{C_p}.$

Because we work in finite-dimensional vector spaces, all linear operators are continuous and hence bounded. Thus any bilinear form of the type $\langle A x, y \rangle$, with $A$ linear, is bounded, since $\langle A x, y \rangle \leq \|A\| \|x\| \|y\|$. By the Riesz representation theorem, bounded bilinear forms correspond one-to-one to linear operators. Therefore, the effective resistance bilinear form $\mathcal{R}_p$ corresponds uniquely to the operator $\ERoperator_p$, which we call the \emph{effective resistance (linear) operator}.

By \cref{eq:ER operator chains}, the effective resistance operator is of the form of $\varphi \varphi^+$. It follows immediately from properties of bounded linear operators and their pseudo-inverses that it satisfies the following properties.
    
\begin{proposition} 
\label[proposition]{prop:properties of R operator}
    The $p$-th effective resistance operator $\ERoperator_p=\partial_p^\adjoint \big( \partial_p^\adjoint \big)^+$ on chains is the orthogonal projection of $C_p(K)$ onto $\big(\kernel ( \partial_p^\adjoint )^+\big) ^\perp
    =\big(\kernel \partial_p \big)^\perp
    =\image \partial_p^\adjoint $. This means that $\ERoperator_p$ is self-adjoint, idempotent and has $\image\partial_p^\adjoint$ as its image and $\ker \partial_p$ as its kernel.
\end{proposition}

\begin{corollary}
\label[corollary]{cor:ER-postive-semi-definite}
    The effective resistance bilinear form $\Rcal_p$ is symmetric and positive semi-definite. Moreover, for any $p$-chain $\alpha\in C_p(K)$ we have
    \begin{equation}\label{eq:ER in norm form}
        r_\alpha 
        = \|\ERoperator_p (\alpha)\|_{C_p}^2
        = \|\partial_p^\adjoint \big( \partial_p^\adjoint \big)^+ (\alpha)\|_{C_p}^2
    \end{equation} 
\end{corollary}

\begin{proof}
    Symmetry follows directly from the self-adjointness of the effective resistance operators.
    Positive semidefiniteness follows from their self-adjointness and idempotence. Indeed, for any $p$-chain $\alpha\in C_p(K)$, we have
    \[
        r_\alpha = \big\langle \ERoperator_p (\alpha), \alpha \big\rangle_{C_p} 
    = \big\langle (\ERoperator_p)^2 (\alpha), \alpha \big\rangle_{C_p} 
    = \langle \ERoperator_p (\alpha), \ERoperator_p (\alpha) \rangle_{C_p} 
    = \|\ERoperator_p (\alpha)\|_{C_p}^2
    = \|\partial_p^\adjoint \big( \partial_p^\adjoint \big)^+ (\alpha)\|_{C_p}^2 \geq 0.
    \]
\end{proof}

\section{Effective Resistance and Thomson's Principle}
\label{sec:thomson-principle}

Thomson's principle provides an alternative definition of effective resistance as the minimal total energy dissipation across all unit flows between two nodes \cite{ThomsonTait1875}. We now provide a formulation of Thomson's principle in our setting. Recall that we say that $\beta\in C_{p-1}(K)$ is a {\em physically valid boundary current} if \cref{eq:circuit-system} in the unknowns $\alpha$, $f$, and $\phi$ admits a solution. 


\begin{definition}[Electric power]
    Given a current $\alpha\in C_p(K)$, its \emph{electric power} is the quadratic form
    \[\mathcal{P}(\alpha)\coloneqq\langle \alpha,\alpha\rangle_{C_p}\]
\end{definition}

The following theorem rephrases Thomson's principle in our language. It says that, among all currents $\alpha\in C_p(K)$ satisfying the boundary condition $\partial_p\alpha=\beta$, the electric power is minimized by those $\alpha$ that are solutions of \cref{eq:circuit-system}, and in that case the electric power is equal to the effective resistance of $\alpha$.  

\begin{theorem}
    If $\alpha\in C_p(K)$ is a solution of \cref{eq:circuit-system} given some boundary current $\beta\in C_{p-1}(K)$, the electric power of $\alpha$ equals the effective resistance $r_\alpha$ of $\alpha$. Moreover, 
    \[r_\alpha=\mathcal{P}(\alpha)=\min_{\substack{\alpha'\in C_p(K)\\ \partial \alpha'=\beta}}\, \mathcal{P}(\alpha') \ .\]
\end{theorem}

\begin{proof}
    Let us start with a physically valid current $\alpha$. Because  its voltage $\flat_p\alpha$ satisfies Kirchhoff's voltage law, and has a  potential $\phi$ that satisfies $\mathcal{L}^{p-1}_{\textrm{up}}\phi=\flat_{p-1}\beta$ by \cref{prop:potential_and_voltage_law},  we have

\begin{equation*}
\begin{array}{ll}
         \mathcal{P}(\alpha)&=\langle \alpha,\alpha\rangle_{C_p}=\langle \flat_p\alpha,\flat_p\alpha\rangle_{C^p}=\langle \delta_{p-1}\phi,\delta_{p-1}\phi\rangle_{C^p}=\langle\mathcal{L}^{p-1}_{\textrm{up}}\phi,\phi\rangle_{C^{p-1}}\\
    & = \langle\mathcal{L}^{p-1}_{\textrm{up}}\left(\mathcal{L}^{p-1}_{\textrm{up}}\right)^+\mathcal{L}^{p-1}_{\textrm{up}}\phi,\phi\rangle_{C^{p-1}}\\
    & =\langle\left(\mathcal{L}^{p-1}_{\textrm{up}}\right)^+\mathcal{L}^{p-1}_{\textrm{up}}\phi,\mathcal{L}^{p-1}_{\textrm{up}}\phi\rangle_{C^{p-1}}\\
    & =\langle\left(\mathcal{L}^{p-1}_{\textrm{up}}\right)^+\flat_{p-1}\beta,\flat_{p-1}\beta\rangle_{C^{p-1}}\\
    & =
\hat{\mathcal R}^{p-1}(\flat_{p-1}\beta,\flat_{p-1}\beta)=\hat{\mathcal R}_{p-1}(\beta,\beta)={\mathcal R}_{p}(\alpha,\alpha)=r_\alpha.        
\end{array}
\end{equation*}

Let us now consider the electric power of some $\alpha'\in C_p(K)$ such that $\partial_p\alpha'=\beta$. We get a $p$-cycle $\gamma\coloneqq\alpha-\alpha'$. By bilinearity of the inner product,
\begin{equation*}
\begin{array}{ll}
         \mathcal{P}(\alpha')&=\mathcal{P}(\alpha-\gamma)=\langle  \alpha-\gamma, \alpha-\gamma \rangle_{C_p}= \mathcal{P}(\alpha)+\mathcal{P}(\gamma)-2\langle  \alpha, \gamma \rangle_{C_p}.
\end{array}
\end{equation*}
By non-negativity of the inner product, $\mathcal{P}(\gamma)\ge 0$. Hence, it is sufficient to prove that $\langle  \alpha, \gamma \rangle_{C_p}=0$ to conclude that $\mathcal{P}(\alpha')\ge \mathcal{P}(\alpha)=r_\alpha$. Recalling that $\alpha$ is a solution of the system of equations in \cref{eq:circuit-system} and that $\gamma$ is a cycle, we have
\[\langle  \alpha, \gamma \rangle_{C_p}=\flat_p(\alpha)(\gamma)=\delta_{p-1}\phi(\gamma)=\phi(\partial_p\gamma)=\phi(0)=0\]\end{proof}

\section{Metric Properties of Effective Resistance}
\label{sec:metric-properties}

In this section, we study the metric structures induced by the effective resistance bilinear form \(\Rcal_p\) and the boundary effective resistance bilinear form \(\hat{\Rcal}_p\) from \cref{def:ER bilinear form in SC} and \cref{def:ER-bilinear-form-boundaries}, respectively. Recall that \(\Rcal_p\) is uniquely determined by the effective resistance operator \(\ERoperator_{p}=\partial_{p}^{\adjoint}(\chainLup{p-1})^+\partial_{p}\) in \cref{eq:ER operator chains}.
Fixing \(p\ge0\), for \(\alpha,\alpha'\in C_p(K)\), we define
\begin{align}
d_p(\alpha,\alpha')
&\coloneqq
\sqrt{\Rcal_p(\alpha-\alpha',\alpha-\alpha')}
=
\|\ERoperator_p(\alpha-\alpha')\|_{C_p},
\label{eq:chain-pseudometric}\\[2mm]
\hat d_p(\alpha,\alpha')
&\coloneqq
\begin{cases}
    \sqrt{\hat{\Rcal}_p(\alpha-\alpha',\alpha-\alpha')}
    =
\|\ERoperator_{p+1}(\beta)\|_{C_{p+1}},
    & \alpha-\alpha'=\partial_{p+1}\beta\in\image(\partial_{p+1}),\\[2mm]
    +\infty,
    & \alpha-\alpha'\notin\image(\partial_{p+1}),
\end{cases}
\label{eq:hat d}\\[2mm]
\ResistanceDistance{p}(\alpha,\alpha')
&\coloneqq
\hat d_p(\alpha,\alpha')^2.
\label{eq:no-square-root-Dp}
\end{align}
where we assume $(+\infty)^2 = +\infty$.
In \cref{eq:hat d}, whenever \(\alpha-\alpha'\) is a boundary,
\(\beta\in C_{p+1}(K)\) is any chain satisfying
\(\partial_{p+1}\beta=\alpha-\alpha'\).
Since the quantity
\(
\ERoperator_{p+1}(\beta)
=
\partial_{p+1}^{\adjoint}(\chainLup{p})^+
\partial_{p+1}(\beta)
=
\partial_{p+1}^{\adjoint}(\chainLup{p})^+
(\alpha-\alpha')
\)
depends only on \(\alpha\) and \(\alpha'\), the finite value
in the definition of $\hat{d}_p$ is independent of the choice of \(\beta\), and thus, well-defined.

In \cref{subsec:square-root metrics}, we show that \(d_p\)
is a pseudometric on \(C_p(K)\), while \(\hat d_p\)
is an extended metric whose restriction to each affine coset
\(\alpha_0+\image(\partial_{p+1})\), with
\(\alpha_0\in C_p(K)\), is a genuine metric
(\cref{prop:ER-metric}).
For a cycle \(\alpha_0\in\kernel\partial_p\), the coset
\(\alpha_0+\image(\partial_{p+1})\) consists precisely of the
cycles representing its homology class.
On connected graphs (\(p=0\)), the restriction of \(\hat d_0\)
to the vertex set satisfies
\(\hat d_0(u,v)=\sqrt{r_{uv}}\), the square root of the classical
effective resistance distance, also studied
in~\cite[Appendix~B.3]{DevriendtLambiotte2022}.

If we want to match the graph convention exactly~\cite{KleinRandic1993, QiuEdwin2007}, we need to focus on \(\ResistanceDistance{0}(u,v)=r_{uv}\), motivating the definition of \(\ResistanceDistance{p}\) in \cref{eq:no-square-root-Dp}.
In \cref{subsec:no-square-root metrics}, we show that, whenever
\(\image(\partial_{p+1})\neq\{0\}\),
\(\ResistanceDistance{p}\) fails the triangle inequality on
every entire affine coset. We give conditions under which its
restriction to a subset is a metric, recovering the classical
effective resistance metric on graph vertices.

\subsection{Distances Induced by the Square Root of Effective Resistance}
\label{subsec:square-root metrics}


\begin{theorem}\label{prop:ER-metric}
Let \(K\) be a finite simplicial complex and \(p \ge 0\).
\begin{enumerate}
    \item The function \(d_p\) in \cref{eq:chain-pseudometric} defines a pseudometric on \(C_p(K)\). It is a metric if and only if \(\kernel(\partial_p)=\{0\}\).
    \item For every \(p\)-chain \(\alpha_0\), the function \(\hat d_p\) in
\cref{eq:hat d} defines a metric on
\(
\alpha_0+\image(\partial_{p+1})
\).
\end{enumerate}
\end{theorem}

\begin{proof}
Non-negativity, symmetry, and the triangle inequality for \(d_p\) follow from the norm representation in \cref{eq:chain-pseudometric} and the linearity of \(\ERoperator_p\). By \cref{prop:properties of R operator}, \(d_p(\alpha,\alpha')=0\) if and only if \(\alpha-\alpha'\in\kernel(\partial_p)\). Thus \(d_p\) is a metric precisely when \(\kernel(\partial_p)=\{0\}\).

For the second statement,
let \(\alpha_0\) be a \(p\)-chain. The non-negativity and symmetry of \(\hat d_p\) are immediate from \eqref{eq:hat d}.
For the triangle inequality, let \(\alpha, \alpha', \alpha'' \in \alpha_0+\image(\partial_{p+1})\) and choose \((p+1)\)-chains \(\beta, \gamma\) such that
\(\partial_{p+1}\beta = \alpha - \alpha'\) and \(\partial_{p+1}\gamma = \alpha' - \alpha''\).
Then \(\partial_{p+1}(\beta + \gamma) = \alpha - \alpha''\), and hence
\[
    \hat d_p(\alpha, \alpha'') 
    = \|\ERoperator_{p+1}(\beta + \gamma)\|_{C_{p+1}}
    \le \|\ERoperator_{p+1}(\beta)\|_{C_{p+1}}
    + \|\ERoperator_{p+1}(\gamma)\|_{C_{p+1}}
    = \hat d_p(\alpha, \alpha') + \hat d_p(\alpha', \alpha''),
\]
by the triangle inequality in the Hilbert space \(C_{p+1}(K)\).
Finally, we verify the identity of indiscernibles. If \(\alpha = \alpha'\), we may take \(\beta = 0\), so \(\hat d_p(\alpha, \alpha) = \|\ERoperator_{p+1}(0)\|_{C_{p+1}} = 0\). Conversely, suppose \(\hat d_p(\alpha, \alpha') = 0\). Then \(\|\ERoperator_{p+1}(\beta)\|_{C_{p+1}} = 0\) for any \(\beta\) satisfying \(\partial_{p+1} \beta = \alpha - \alpha'\). 
This implies \(\beta \in \ker(\ERoperator_{p+1})= \ker(\partial_{p+1})\), so that
\(
\alpha - \alpha' = \partial_{p+1} \beta = 0,
\)
yielding \(\alpha = \alpha'\).
Therefore, \(\hat d_p\) is a metric on
\(\alpha_0+\image(\partial_{p+1})\).
\end{proof}

\begin{remark}[Affine slices of \(C_p(K)\)]
The finite-distance components of \((C_p(K),\hat d_p)\) are the affine
cosets of \(\image(\partial_{p+1})\). We have shown in \cref{prop:ER-metric} that each
component carries a genuine metric. Translation identifies any two
components isometrically, so \((C_p(K),\hat d_p)\) may be viewed as a
collection of parallel, isometric metric slices, as illustrated in
\cref{fig:affine-slices-hat-d}.
\end{remark}

\begin{figure}[ht]
\centering



\begin{tikzpicture}[
    x=1cm,y=1cm,
    plane/.style={
        draw=blue!45!black,
        fill=blue!4,
        line width=0.7pt,
        line join=round
    },
    iso/.style={
        <->,
        draw=blue!45!black,
        line width=0.7pt,
        densely dashed
    }
]
\foreach \i in {0,1,2} {
    \begin{scope}[yshift={1.6*\i cm}]
        \path[plane]
            (0,0) -- (4.4,0) -- (5.4,0.85)
            -- (1,0.85) -- cycle;
        \node at (2.7,0.425) {
            \(\bigl(\alpha_{\i}+\image(\partial_{p+1}),
            \,\hat d_p\bigr)\)
        };
    \end{scope}
}

\draw[iso]
    (5.75,0.85) -- node[right=4pt,text=black]
    {\(\cong\)} (5.75,1.6);
\draw[iso]
    (5.75,2.45) -- node[right=4pt,text=black]
    {\(\cong\)} (5.75,3.2);
\end{tikzpicture}
\caption{The affine cosets of \(\image(\partial_{p+1})\), viewed as the
isometric finite-distance components of \((C_p(K),\hat d_p)\).
Here, $\alpha_i$ are $p$-chains such that $\alpha_i-\alpha_j\notin \image(\partial_{p+1})$ for $i\ne j$.}
\label{fig:affine-slices-hat-d}
\end{figure}

In the graph case (\(p=0\)), let \(K=G=(V,E)\) be connected.
Then all vertices belong to the same homology class.
The restriction of \(\hat d_0\) from \cref{eq:hat d} to \(V\) is therefore a metric.
For \(u,v\in V\), choosing any \(1\)-chain \(\alpha\) with \(\partial_1\alpha=u-v\) yields
\(\hat d_0(u,v)=\sqrt{\Rcal_1(\alpha,\alpha)}=\sqrt{r_{uv}}\).

\begin{remark}[$d_p$ vs. $\hat d_p$]
The relation and distinction between \(d_p\) and \(\hat d_p\) is better explained using the boundary effective resistance bilinear form \(\hat{\Rcal}_p\) defined in \cref{def:ER-bilinear-form-boundaries}. In particular, we have
\begin{itemize}
    \item
    \(d_p(\alpha,\alpha')=\sqrt{\hat{\Rcal}_{p-1}(\partial_{p}\alpha - \partial_{p}\alpha',\,\partial_{p}\alpha - \partial_{p}\alpha')}\), for any \(\alpha,\alpha'\in C_p(K)\);
    \item
    \(\hat d_p(\alpha,\alpha')=\sqrt{\hat{\Rcal}_p(\alpha-\alpha',\alpha-\alpha')}\), for any \(\alpha,\alpha'\in \alpha_0+\image(\partial_{p+1})\).
\end{itemize}

Thus, the \(d_p\)-distance between two \(p\)-chains
depends only on their boundaries and equals the \(\hat d_{p-1}\)-distance
between those boundaries. Consequently, chains with the same boundary
have \(d_p\)-distance zero, and \(\partial_{p+1}\) induces an isometry from
\(C_{p+1}(K)/\kernel(\partial_{p+1})\) onto \(\image(\partial_{p+1})\). In summary, we have the following diagram of extended metric spaces:
\[
\begin{tikzcd}[column sep=large,row sep=large]
\left(
C_{p+1}(K)/\kernel(\partial_{p+1}),d_{p+1}
\right)
\arrow[r,"\overline{\partial}_{p+1}","\cong"']
&
\left(
\image(\partial_{p+1}),\hat d_p
\right)
\arrow[d,"\tau_{\alpha_0}","\cong"']
\\
&
\left(
\alpha_0+\image(\partial_{p+1}),\hat d_p
\right)
\arrow[r,hook]
&
\left(C_p(K),\hat d_p\right),
\end{tikzcd}
\]
where
\(
\overline{\partial}_{p+1}([\alpha])\coloneqq\partial_{p+1}\alpha
\)
,
\(
\tau_{\alpha_0}(\alpha)\coloneqq\alpha_0+\alpha
\) and $\cong$ denotes an isometry.

In the same degree, the two distance functions play complementary roles. Since
\(\image(\partial_{p+1})\subseteq\kernel(\partial_p)\), \(d_p\) vanishes
on each affine coset of \(\image(\partial_{p+1})\), whereas \(\hat d_p\)
provides a genuine metric within each such coset.
\end{remark}

\paragraph{Potential-theoretic interpretation of closeness.}
The characterization of potentials in
\cref{prop:potential_and_voltage_law}
gives a direct interpretation of \(d_p\) and \(\hat d_p\).
For \(p\ge0\) and \(\beta\in\image(\partial_{p+1})\), recall from \cref{eq:general-boundary-potential-pairing-rem} the minimal-norm potential generated by the boundary current \(\beta\):
\begin{equation}\label{eq:potential}
\phi_\beta\coloneqq(\cochainLup{p})^+\flat_p(\beta)\in C^p(K).
\end{equation}

By \cref{lem:ER_chain_boundary_equivalence,eq:general-boundary-potential-pairing-rem},
for \(p\ge1\) and \(\alpha,\alpha'\in C_p(K)\), we have
\[
d_p(\alpha,\alpha')^2=\phi_{\partial_p(\alpha-\alpha')}\bigl(\partial_p(\alpha-\alpha')\bigr).
\]
For \(p\ge0\) and \(\alpha,\alpha'\in C_p(K)\) with
\(\alpha-\alpha'\in\image(\partial_{p+1})\), we have
\[
\hat d_p(\alpha,\alpha')^2=\ResistanceDistance{p}(\alpha,\alpha')=\phi_{\alpha-\alpha'}(\alpha-\alpha').
\]
Thus, for \(p\ge1\), two chains are \(d_p\)-close when the
potential generated by their boundary difference has a small
evaluation on that difference. In particular,
\(d_p(\alpha,\alpha')=0\) precisely when \(\alpha\) and
\(\alpha'\) have the same boundary.
Similarly, two \(p\)-chains in the same affine coset are
\(\hat d_p\)-close when the potential generated by their
difference has a small evaluation on that difference.
For \(p=0\), when \(\alpha\) and \(\alpha'\) are vertices
of a connected graph, the second identity recovers the potential
difference induced by a unit current between them, as in
\cref{coro:difference_potential}.

\subsection{Distances Induced by Effective Resistance}
\label{subsec:no-square-root metrics}

We now study the no-square-root function
\(\ResistanceDistance{p}\) defined in \cref{eq:no-square-root-Dp}.
For \(\alpha,\alpha'\in C_p(K)\) such that
\(\alpha-\alpha'\in\image(\partial_{p+1})\), we have
\begin{equation}\label{eq:Dp-finite}
\ResistanceDistance{p}(\alpha,\alpha')=\hat{\Rcal}_p(\alpha-\alpha',\alpha-\alpha')=\big\langle(\chainLup{p})^+(\alpha-\alpha'),\alpha-\alpha'\big\rangle_{C_p}.
\end{equation}
Otherwise, \(\ResistanceDistance{p}(\alpha,\alpha')=+\infty\).

For a connected graph, \(\ResistanceDistance{0}\) restricted to the set of vertices (0-simplices) coincides with classical effective resistance on the vertex set, which is a well-known metric for graphs. We recall this classical result in \cref{thm:graph-effective-resistance-metric}.

\begin{theorem}[Effective resistance metric on graph vertices \cite{KleinRandic1993, QiuEdwin2007}]
\label[theorem]{thm:graph-effective-resistance-metric}
Let \(G\) be a finite connected weighted graph with unit vertex
weights and positive edge weights. Then the restriction of
\(\ResistanceDistance{0}\) to \(V\) is a metric, with
\(\ResistanceDistance{0}(u,v)=r_{uv}\) for all \(u,v\in V\).
\end{theorem}

The restriction to \(V\) is essential. 
\Cref{ex:no-square-root-triangle-failure} shows that \(\ResistanceDistance{0}\) is not a metric on an entire
homology class, and that the same obstruction occurs in dimension \(p\)
whenever \(\image(\partial_{p+1})\neq\{0\}\).

\begin{example}\label[example]{ex:no-square-root-triangle-failure}
    Let $K$ be a connected simplicial complex.
    Let \(u,v\in V\) be distinct vertices in $K$. Since \(K\) is connected, the \(0\)-chains $u,v$ and $2v-u$ belong to the same homology class. However, by bilinearity,
    \[
        \ResistanceDistance{0}(u,2v-u)
        =4r_{uv}
        >
        2r_{uv}
        =
        \ResistanceDistance{0}(u,v)+\ResistanceDistance{0}(v,2v-u).
    \]
    Thus, \(\ResistanceDistance{0}\) is not a metric on the entire homology class, even though its restriction to \(V\) is a metric.

    For $p\geq 1$, the same obstruction applies. If
    \(0\neq\gamma\in\image(\partial_{p+1})\) and \(\alpha\in C_p(K)\), then
    \(\alpha-\gamma,\alpha,\alpha+\gamma\) belong to the same affine coset \(\alpha+\image(\partial_{p+1})\), but
    \[
        \ResistanceDistance{p}(\alpha-\gamma,\alpha+\gamma)
        =4\hat{\Rcal}_p(\gamma,\gamma)
        >
        2\hat{\Rcal}_p(\gamma,\gamma)
        =
        \ResistanceDistance{p}(\alpha-\gamma,\alpha)+\ResistanceDistance{p}(\alpha,\alpha+\gamma).
    \]
\end{example}

Whenever \(\image(\partial_{p+1})\neq\{0\}\), \(\ResistanceDistance{p}\) is not a metric on
any entire set \(\alpha_0+\image(\partial_{p+1})\). Nevertheless, it may be a
metric after restriction to a subset. \Cref{lem:no-square-root-metric-criterion} gives a necessary and
sufficient condition.

\begin{lemma}
\label[lemma]{lem:no-square-root-metric-criterion}
Let \(\alpha_0\in C_p(K)\) and
\(S\subseteq \alpha_0+\image(\partial_{p+1})\). Then \(\ResistanceDistance{p}\) restricts to a metric on \(S\) if and only if for all \(\alpha,\alpha',\alpha''\in S\),
\begin{equation}\label{eq:sign-condition}
    \hat{\Rcal}_p(\alpha-\alpha',\alpha'-\alpha'')\le0.
\end{equation}
\end{lemma}

\begin{proof}
Symmetry of \(\ResistanceDistance{p}\) is immediate from \cref{eq:no-square-root-Dp}.
Since \(\alpha-\alpha''\in\image(\partial_{p+1})=\image(\chainLup{p})\) and
\((\chainLup{p})^+\) is positive definite on this space, we have
\(\ResistanceDistance{p}(\alpha,\alpha'')\ge0\), with equality if and only if \(\alpha-\alpha''=0\), or equivalently, \(\alpha=\alpha''\).
It therefore remains only to characterize the triangle inequality.

By bilinearity, for every \(\alpha,\alpha',\alpha''\in S\),
\begin{equation*}
    \ResistanceDistance{p}(\alpha,\alpha'')
    =
    \ResistanceDistance{p}(\alpha,\alpha')+\ResistanceDistance{p}(\alpha',\alpha'')
    +2\hat{\Rcal}_p(\alpha-\alpha',\alpha'-\alpha'').
\end{equation*}
Hence,
\(\ResistanceDistance{p}(\alpha,\alpha'')\le \ResistanceDistance{p}(\alpha,\alpha')+\ResistanceDistance{p}(\alpha',\alpha'')\)
if and only if
\(\hat{\Rcal}_p(\alpha-\alpha',\alpha'-\alpha'')\le0\).
\end{proof}

Recall from \cref{para: ER form as potentials} that effective resistance bilinear forms can be expressed in terms of potentials.
For \(S\subseteq \alpha_0+\image(\partial_{p+1})\) as in
\cref{lem:no-square-root-metric-criterion} and \(\alpha,\alpha',\alpha''\in S\),
applying \cref{eq:general-boundary-potential-pairing-rem} 
to \(\alpha'-\alpha\)
and \(\alpha'-\alpha''\) shows that the condition in \cref{eq:sign-condition} is
equivalent to
\begin{equation}\label{eq:sign-condition-potential-form}
\phi_{\alpha-\alpha'}(\alpha')\le\phi_{\alpha-\alpha'}(\alpha''),
\end{equation}
where \(\phi_{\alpha-\alpha'}\coloneqq(\cochainLup{p})^+\flat_p(\alpha-\alpha')\).

In the graph case ($p=0$), let \(\alpha,\alpha'\in V\)
be distinct. If \(\gamma=\alpha-\alpha'\) is taken to be a vertex difference, then $\phi_\gamma$ is the unit-current potential with positive source at $\alpha$ and negative source at $\alpha'$.
Assuming graph connectedness, the \emph{discrete maximum principle} (see, e.g.,~\cite[Ch.~2]{DoyleSnell1984}) states that $\phi_\gamma$ attains its minimum at $\alpha'$ and its maximum at $\alpha$; hence $\phi_\gamma(\alpha'')\ge\phi_\gamma(\alpha')$ for every vertex $\alpha''$, so \cref{eq:sign-condition-potential-form} holds on every triple of vertices and $\ResistanceDistance{0}$ restricted to vertices is a metric.

\paragraph{Metrics from orthogonal boundaries.}
For \(p\ge1\), however, the standard \(p\)-simplex basis does not give a directly analogous restriction, since the difference of two distinct \(p\)-simplices does not belong to \(\image(\partial_{p+1})\). Nevertheless, pairwise \(\hat{\Rcal}_p\)-orthogonal \(p\)-boundaries induce subsets of $p$-chains on which \(\ResistanceDistance{p}\) is a metric.

\begin{proposition}
\label[proposition]{prop:no-square-root-orthogonal-family}
Let \(b_1,\ldots,b_m\) be nonzero, pairwise \(\hat{\Rcal}_p\)-orthogonal elements of \(\image(\partial_{p+1})\). Then, for every \(\alpha_0\in C_p(K)\), the restriction of \(\ResistanceDistance{p}\) to
\(\{\alpha_0,\alpha_0+b_1,\ldots,\alpha_0+b_m\}\)
is a metric.
\end{proposition}

\begin{proof}
Thanks to \cref{lem:no-square-root-metric-criterion}, we only need to check the sign condition.

Set \(b_0\coloneqq0\) and \(r_i\coloneqq\hat{\Rcal}_p(b_i,b_i)\) for \(i=0,\ldots,m\).
Consider three $p$-chains
\(\alpha_0+b_i,\alpha_0+b_j,\alpha_0+b_k\) in \(\{\alpha_0+b_i\}_{i=0}^m\).
If two of the three points coincide, the condition in
\cref{eq:sign-condition} is immediate: one factor vanishes when
\(i=j\) or \(j=k\), while for \(i=k\) the relevant expression is
\(
    -\hat{\Rcal}_p(b_i-b_j,b_i-b_j)\le0
\) by positive semi-definiteness of \(\hat{\Rcal}_p\).

Now suppose \(i,j,k\) are distinct.
Orthogonality of the \(b_i\) implies for \(0\le i,j,k\le m\)
\[
    \hat{\Rcal}_p((\alpha_0+b_i)-(\alpha_0+b_j),(\alpha_0+b_j)-(\alpha_0+b_k))
    =
    \hat{\Rcal}_p(b_i-b_j,b_j-b_k)
    =
    \begin{cases}
        0, & j=0,\\
        -r_j, & j\neq0.
    \end{cases}
\]
Therefore, \cref{lem:no-square-root-metric-criterion} applies, and \(\ResistanceDistance{p}\) restricts to a metric on \(\{\alpha_0+b_i\}_{i=0}^m\).
\end{proof}

A natural example of such a family is obtained from the spectrum of the up Laplacian.

\begin{corollary}
\label[corollary]{cor:no-square-root-spectral-family}
Let \(b_1,\ldots,b_r\) be an orthonormal basis, with respect to
\(\langle\cdot,\cdot\rangle_{C_p}\), of eigenvectors of
\(\chainLup{p}\) corresponding to its positive eigenvalues
\(\lambda_1,\ldots,\lambda_r\). Then, for every \(\alpha_0\in C_p(K)\),
the restriction of \(\ResistanceDistance{p}\) to
\(\{\alpha_0,\alpha_0+b_1,\ldots,\alpha_0+b_r\}\)
is a metric.
\end{corollary}

\begin{proof}
The positive-eigenvalue eigenvectors form a basis of
\(
    \image(\chainLup{p})=\image(\partial_{p+1}).
\)
Moreover, since
\(\chainLup{p}b_i=\lambda_i b_i\) with \(\lambda_i>0\), the definition of the pseudoinverse gives
\[
    (\chainLup{p})^+b_i=\frac{1}{\lambda_i}b_i.
\]
Therefore, orthonormality implies
\[
    \hat{\Rcal}_p(b_i,b_j)
    =\big\langle(\chainLup{p})^+b_i,b_j\big\rangle_{C_p}=\frac{1}{\lambda_i}\langle b_i,b_j\rangle_{C_p}
    =\frac{\delta_{ij}}{\lambda_i}.
\]
Thus, the \(b_i\) are nonzero, pairwise
\(\hat{\Rcal}_p\)-orthogonal elements of
\(\image(\partial_{p+1})\), and the result follows from
\cref{prop:no-square-root-orthogonal-family}.
\end{proof}

We note that, unlike the vertex set in the graph case, the spectral restriction in
\cref{cor:no-square-root-spectral-family} depends on the choice of eigenbasis and is therefore generally noncanonical, particularly when a positive eigenvalue has multiplicity greater than one. However, \cref{cor:no-square-root-spectral-family} opens the way to the definition of other metrics on graphs, among less canonical chain sets than the set of vertices. 

\paragraph{Metrics from inverse Gram matrices.}

The orthogonal families in \cref{prop:no-square-root-orthogonal-family}
provide finite sets of chains on which \(\ResistanceDistance{p}\) is a
metric. They do not, however, recover the vertex set of an arbitrary
connected graph: for a fixed vertex \(v\), the boundaries
\(v_i-v\) need not be pairwise \(\hat{\Rcal}_0\)-orthogonal.
To include the graph case, we replace orthogonality by a condition on
the inverse Gram matrix of a linearly independent family of boundaries.
This condition allows nondiagonal Gram matrices and recovers both the
orthogonal construction and the effective resistance metric on graph
vertices.

\begin{theorem}
\label{thm:no-square-root-graph-laplacian}
Let \(b_1,\ldots,b_m\) be linearly independent elements of
\(\image(\partial_{p+1})\), and set
\[
A\coloneqq\bigl(\hat{\Rcal}_p(b_i,b_j)\bigr)_{i,j=1}^m.
\]
Suppose that
\begin{equation}\label{eq:M condition}
(A^{-1})_{ij}\le0\quad(i\ne j),
\qquad \sum_{j=1}^{m}(A^{-1})_{ij}\ge0\quad(1\le i\le m).
\end{equation}
Then, for every
\(\alpha_0\in C_p(K)\), the restriction of \(\ResistanceDistance{p}\) to
\(\{\alpha_0,\alpha_0+b_1,\ldots,\alpha_0+b_m\}\) is a metric.
\end{theorem}

Following the construction
in \cite[proof of Theorem~2.1]{Fiedler1998}, we associate a connected
weighted graph to the chains
\(\{\alpha_0,\alpha_0+b_1,\ldots,\alpha_0+b_m\}\).
Setting \(b_0=0\), vertex \(i\) in the auxiliary graph corresponds to \(\alpha_0+b_i\), and the
effective resistance distance between vertices \(i\) and \(j\) in the auxiliary graph equals
\(\ResistanceDistance{p}(\alpha_0+b_i,\alpha_0+b_j)\).
The conclusion then follows from
\cref{thm:graph-effective-resistance-metric}.

\cref{thm:no-square-root-graph-laplacian} recovers
\cref{prop:no-square-root-orthogonal-family} when \(A\) is diagonal.
Indeed, nonzero, pairwise \(\hat{\Rcal}_p\)-orthogonal boundaries
are linearly independent and have a positive diagonal Gram matrix.
Its inverse is also positive diagonal and therefore satisfies
\cref{eq:M condition}.
In particular, the eigenvectors in
\cref{cor:no-square-root-spectral-family} form an
\(\hat{\Rcal}_p\)-orthogonal basis of
\(\image(\partial_{p+1})\). Choosing any \(m\) of these
eigenvectors shows that the hypotheses of
\cref{thm:no-square-root-graph-laplacian} can be satisfied
for every \(1\le m\le\dim\image(\partial_{p+1})\).

\smallskip

\Cref{ex:nonorthogonal-boundary-metric} gives a nonorthogonal family of boundaries
satisfying \cref{eq:M condition}. For comparison, we also present
an orthogonal family in the same complex that falls within
\cref{prop:no-square-root-orthogonal-family}.

\begin{example}\label[example]{ex:nonorthogonal-boundary-metric}
Let \(K\) be the simplicial complex generated by the oriented
triangles \(\sigma_1=[0,1,2]\) and \(\sigma_2=[0,2,3]\),
which share the edge \([0,2]\), with unit simplex weights
(\cref{fig:quadrilateral-complex}).
Set \(b_1=\partial_2\sigma_1\) and
\(b_2=\partial_2(\sigma_1+\sigma_2)\), the boundary of the first
triangle and the outer boundary of \(K\), respectively.

\begin{figure}[htbp]
\centering
\begin{tikzpicture}[scale=0.95]
\fill[blue!10] (0,0) -- (2.6,0) -- (2.6,2.2) -- cycle;
\fill[orange!15] (0,0) -- (2.6,2.2) -- (0,2.2) -- cycle;
\draw[thick] (0,0) -- (2.6,0) -- (2.6,2.2) -- (0,2.2) -- cycle;
\draw[thick] (0,0) -- (2.6,2.2);
\fill (0,0) circle (2pt) node[below left] {\(0\)};
\fill (2.6,0) circle (2pt) node[below right] {\(1\)};
\fill (2.6,2.2) circle (2pt) node[above right] {\(2\)};
\fill (0,2.2) circle (2pt) node[above left] {\(3\)};
\node at (1.8,0.65) {\(\sigma_1\)};
\node at (0.8,1.55) {\(\sigma_2\)};
\end{tikzpicture}
\caption{The complex \(K\), consisting of two triangles sharing an edge,
with unit simplex weights.
The chains \(b_1\) and \(b_2\) are the boundary of \(\sigma_1\)
and the outer boundary of \(K\), respectively.}
\label{fig:quadrilateral-complex}
\end{figure}

Since \(\partial_2\) is injective, \cref{ex:injective-boundary} gives
\(
    \hat{\Rcal}_1(\partial_2\alpha,\partial_2\alpha')
    = \langle\alpha,\alpha'\rangle_{C_2}
\)
for all \(\alpha,\alpha'\in C_2(K)\).
Thus, we have
\[
    A\coloneqq\bigl(\hat{\Rcal}_1(b_i,b_j)\bigr)_{i,j=1}^2=\begin{pmatrix}1&1\\1&2\end{pmatrix},
    \qquad
    A^{-1}=\begin{pmatrix}2&-1\\-1&1\end{pmatrix}.
\]
The inverse has negative off-diagonal entries and row sums \(1,0\),
so \cref{thm:no-square-root-graph-laplacian} applies to give a metric on
\(\{0,b_1,b_2\}\).

Alternatively, choosing \(b_1=\partial_2\sigma_1\) and
\(b_2=\partial_2\sigma_2\) gives an \(\hat{\Rcal}_1\)-orthonormal basis
of \(\image(\partial_2)\), since
\(\hat{\Rcal}_1(\partial_2\sigma_i,\partial_2\sigma_j)
=\langle\sigma_i,\sigma_j\rangle_{C_2}=\delta_{ij}\).
The corresponding Gram matrix is the identity, so this choice falls
within \cref{prop:no-square-root-orthogonal-family}.

\end{example}

\begin{remark}
Connected weighted graphs satisfy the hypotheses of
\cref{thm:no-square-root-graph-laplacian} with \(p=0\).
Let \(G=(V,E)\) be a finite connected weighted graph with unit vertex
weights and positive edge weights. Fix \(v\in V\), write
\(S=V\setminus\{v\}=\{v_1,\ldots,v_m\}\), and set \(b_i=v_i-v\).
Then \(b_1,\ldots,b_m\) form a basis of \(\image(\partial_1)\).

Let \(L\) be the graph Laplacian, denote by \(L(S,S)\) its principal
submatrix indexed by \(S\), and set
\(A=(\hat{\Rcal}_0(b_i,b_j))_{i,j=1}^m\).
For each \(j\), let
\(\psi^{(j)}=L^+(\mathbf{1}_{v_j}-\mathbf{1}_v)\), with
\(\psi^{(j)}_u\) denoting its entry at vertex \(u\).
Then \(L\psi^{(j)}=\mathbf{1}_{v_j}-\mathbf{1}_v\) and
\(A_{ij}=\psi^{(j)}_{v_i}-\psi^{(j)}_v\).
Since the rows of \(L\) sum to zero, subtracting the constant
\(\psi^{(j)}_v\) and restricting to \(S\) gives
\(L(S,S)A=\mathrm{Id}\). Hence \(A^{-1}=L(S,S)\).
Using the edge-weight notation in \cref{eq:graph-laplacian-entries}, we obtain
\[
(A^{-1})_{ij}=-w_{v_i v_j}\le0\quad(i\ne j),
\qquad
\sum_{j=1}^{m}(A^{-1})_{ij}=w_{v_i v}\ge0.
\]
Thus the hypotheses of
\cref{thm:no-square-root-graph-laplacian} hold.
Taking \(\alpha_0=v\), the set
\(\{\alpha_0,\alpha_0+b_1,\ldots,\alpha_0+b_m\}\) is precisely \(V\).
\end{remark}

To prove \cref{thm:no-square-root-graph-laplacian}, we first establish
\cref{lem:graph-realization}, which presents the construction in
\cite[proof of Theorem~2.1]{Fiedler1998} in our notation.
We provide a self-contained proof for completeness.

We use the standard basis vectors \(\mathbf{1}_v\) introduced in
\cref{sec:resistance-graph}. For a vector \(x\), write \(x_v\) for
its entry indexed by \(v\) and \(x(S)\) for its restriction to \(S\).
For a matrix \(A\), write \(A(S,T)\) for the submatrix with rows
indexed by \(S\) and columns indexed by \(T\); in particular,
\(A(S,S)\) is a principal submatrix.

\begin{lemma}
\label[lemma]{lem:graph-realization}
Let \(A\in\Rspace^{m\times m}\) be symmetric and positive definite, and
suppose that
\begin{equation}\label{eq:graph-laplacian-condition}
(A^{-1})_{ij}\le0\quad(1\le i\ne j\le m),
\qquad
\sum_{j=1}^{m}(A^{-1})_{ij}\ge0\quad(1\le i\le m).
\end{equation}
Set \(S=\{1,\ldots,m\}\), and let \(\Gamma\) be the undirected
weighted graph on \(\{0\}\cup S\), with unit vertex weights and edge
weights
\[
w_{ij}=-(A^{-1})_{ij}\quad(1\le i<j\le m),
\qquad
w_{i0}=\sum_{j=1}^{m}(A^{-1})_{ij}\quad(1\le i\le m),
\]
with vertices \(i\) and \(j\) adjacent if and only if \(w_{ij}>0\).
Then \(\Gamma\) is connected,
and the effective resistance \(r_{ij}\) between vertices
\(i\) and \(j\) in \(\Gamma\), defined as in \cref{eq:devriendtER}, satisfies
\[
r_{ij}
=\bigl(\mathbf{1}_i(S)-\mathbf{1}_j(S)\bigr)^\top
    A\bigl(\mathbf{1}_i(S)-\mathbf{1}_j(S)\bigr)
\qquad(0\le i,j\le m).
\]
\end{lemma}

\begin{proof}
The edge weights are nonnegative by
\cref{eq:graph-laplacian-condition}.
Let \(L\) be the graph Laplacian of \(\Gamma\).
By \cref{eq:graph-laplacian-entries}, for distinct \(i,j\in S\),
\(L_{ij}=-w_{ij}=(A^{-1})_{ij}\), while for \(i\in S\),
\[
L_{ii}
=w_{i0}+\sum_{\substack{k\in S\\k\ne i}}w_{ik}
=\sum_{k=1}^{m}(A^{-1})_{ik}
    -\sum_{\substack{1\le k\le m\\k\ne i}}(A^{-1})_{ik}
=(A^{-1})_{ii}.
\]
Thus the principal submatrix of \(L\) with rows and columns
indexed by \(S\) satisfies \(L(S,S)=A^{-1}\).

If \(\Gamma\) were disconnected, it would have a connected
component \(C\subseteq S\) not containing \(0\).
Then \(L(C,C)\) would be the graph Laplacian of \(C\), hence singular.
However, as a principal submatrix of the positive-definite matrix
\(L(S,S)=A^{-1}\), it would be positive definite and hence nonsingular.
This contradiction proves that \(\Gamma\) is connected.

Fix \(0\le i,j\le m\), and set
\(\psi=L^+(\mathbf{1}_i-\mathbf{1}_j)\in\Rspace^{m+1}\).
Here \(\psi_k\) denotes the entry of \(\psi\) indexed by vertex \(k\).
By \cref{eq:devriendtER}, the effective resistance between \(i\) and \(j\) in \(\Gamma\) is
\[
r_{ij}
=(\mathbf{1}_i-\mathbf{1}_j)^\top L^+
    (\mathbf{1}_i-\mathbf{1}_j)
=(\mathbf{1}_i-\mathbf{1}_j)^\top\psi
=\psi_i-\psi_j.
\]
Since \(\Gamma\) is connected,
\(\image L=\mathrm{span}(\mathbf{1})^\perp\), where \(\mathbf{1}\)
is the all-ones vector. Thus
\(\mathbf{1}_i-\mathbf{1}_j\in\image L\) and
\(L\psi=LL^+(\mathbf{1}_i-\mathbf{1}_j)=\mathbf{1}_i-\mathbf{1}_j\).
Set \(\widetilde\psi=\psi-\psi_0\mathbf{1}\). Then
\(L\widetilde\psi=L\psi\) and \(\widetilde\psi_0=0\).
Restricting to \(S\) therefore gives
\[
A^{-1}\widetilde\psi(S)
=L(S,S)\widetilde\psi(S)
=(L\psi)(S)
=\mathbf{1}_i(S)-\mathbf{1}_j(S).
\]
Consequently, the effective resistance between \(i\) and \(j\) in \(\Gamma\) satisfies
\[
r_{ij}
=\widetilde\psi_i-\widetilde\psi_j
=\bigl(\mathbf{1}_i(S)-\mathbf{1}_j(S)\bigr)^\top
    \widetilde\psi(S)
=\bigl(\mathbf{1}_i(S)-\mathbf{1}_j(S)\bigr)^\top
    A\bigl(\mathbf{1}_i(S)-\mathbf{1}_j(S)\bigr). \qedhere
\]
\end{proof}

\begin{proof}[Proof of \cref{thm:no-square-root-graph-laplacian}]
The positive definiteness of \(\hat{\Rcal}_p\) on
\(\image(\partial_{p+1})\) and the linear independence of the \(b_i\)
imply that \(A\) is symmetric and positive definite.
Apply \cref{lem:graph-realization} to \(A\), and denote the resulting
graph by \(\Gamma\). Let \(r^\Gamma_{ij}\) be the effective resistance
between vertices \(i\) and \(j\) in \(\Gamma\).
Set \(S=\{1,\ldots,m\}\) and \(b_0=0\).
By the definition of \(A\), for \(0\le i,j\le m\),
\[
\begin{aligned}
r^\Gamma_{ij}
&=\bigl(\mathbf{1}_i(S)-\mathbf{1}_j(S)\bigr)^\top A\bigl(\mathbf{1}_i(S)-\mathbf{1}_j(S)\bigr)\\
&=\hat{\Rcal}_p(b_i-b_j,b_i-b_j)\\
&=\ResistanceDistance{p}(\alpha_0+b_i,\alpha_0+b_j).
\end{aligned}
\]
By \cref{thm:graph-effective-resistance-metric}, effective resistance
defines a metric on the vertex set of \(\Gamma\). Since the \(b_i\) are
linearly independent, \(i\mapsto \alpha_0+b_i\) is a bijection onto the stated
set of chains, so \(\ResistanceDistance{p}\) restricts to a metric on this set.
\end{proof}

\section{Foster's Theorem: From Graphs to Simplicial Complexes}
\label{sec:fosters}

In this section we extend Foster's Theorem, a property relating relative effective resistances in a graph with its 0-dimensional Betti number, to the context of simplicial complexes. We first recall Foster's Theorem for graphs.

\begin{theorem}[Foster's Theorem in graphs,~\cite{DevriendtLambiotte2022}]

\label{thm:foster_graphs}
The sum of all relative effective resistances of edges in a graph $G$ satisfies the following property:
\begin{align}
\label{eq:Foster_graph}
    \sum_{e \in E} w(e)\, r_e = |V| - \beta_0(G), 
\end{align} 
where $\beta_0(G)$ is the number of connected components in the graph. 
\end{theorem}

We are interested in a result similar to \cref{thm:foster_graphs} using the notion of simplicial effective resistance \(r_\sigma\). 
We begin by proving this result in the operator form, \cref{thm:foster}, and then derive the matrix form as \cref{cor:foster}. 
Our generalization offers a deeper interpretation of the classical Foster’s theorem by showing that the \emph{total relative effective resistance} can be viewed as the \emph{sum of the eigenvalues} (i.e., the trace) of the effective resistance operator, which is equal to its rank because the operator is an orthogonal projection.

In what follows, consider a finite (oriented) weighted simplicial complex $K$, and for any degree $p\geq 0$, let $n_p=|K_p|$, and set $n_{-1}=0$. 
Let \(\operatorname{rank}(\cdot)\) and \(\nullity(\cdot)\) denote the rank and nullity of a linear operator (or a matrix), respectively.  
 
\begin{theorem}[Higher-dimensional Foster's theorem: operator form]
\label{thm:foster}
Let $K$ be a finite weighted simplicial complex. For any $p\geq 0$, let $\{\lambda_i(\ERoperator_p)\}$ be the eigenvalues (counted with multiplicities) of the effective resistance operator $\ERoperator_p$ on chain spaces.
Then, 
\begin{equation}
\label{eq:foster}
    \sum_{i} \lambda_i(\ERoperator_p) =\rank(\partial_p^\adjoint)= 
    n_{p-1} - \nullity (\partial_p^\adjoint). 
\end{equation} 
Moreover, given any orthonormal basis $\{\alpha_i\}$ of $C_p(K)$, 
\begin{equation}
\label{eq:foster_basis}
    \sum_i \langle \ERoperator_p(\alpha_i), \alpha_i\rangle =\rank(\partial_p^\adjoint) 
    = n_{p-1} - \nullity (\partial_p^\adjoint). 
\end{equation} 
\end{theorem}

To prove \cref{thm:foster}, we recall the standard linear algebra fact in \cref{lem:trace} (included here for completeness). The \emph{trace} of an operator $\varphi$, denoted by $\trace(\varphi)$, is the sum of its eigenvalues and can also be obtained by taking the trace of its matrix representation on any orthonormal basis.

\begin{lemma}
\label[lemma]{lem:trace}
    For a linear operator $\varphi: X\to X$ on a finite-dimensional Hilbert space $X$ and any orthonormal basis $\{x_i\}$ of $X$,  
    \[\sum_i \langle \varphi x_i, x_i \rangle = \trace (\varphi).\]
\end{lemma}

\begin{proof}
    Let $\{x_i\}$ be any orthonormal basis of $X$, and $A$ the matrix representation of $\varphi$ with respect to $\{x_i\}$.
    In other words, for any $i$, $\varphi (x_i) = \sum_j A_{ji} \, x_j$.
    Since $\{x_i\}$ is orthonormal, it follows that $\langle \varphi x_i, x_j \rangle = A_{ji}$ for any $i$ and $j$.
    Thus, 
    \[\sum_i \langle \varphi x_i, x_i \rangle =
    \sum_i A_{ii} = \trace(A) = \trace (\varphi). \qedhere\]
\end{proof}

\begin{proof}[Proof of \cref{thm:foster}] 
    Recall from \cref{prop:properties of R operator} that the effective resistance operator \(\ERoperator_p = \partial_p^\adjoint (\partial_p^\adjoint)^+\) is the orthogonal projection of \(C_p(K)\) onto \(\image(\partial_p^\adjoint) 
    \).
    As an idempotent operator, $\ERoperator_p$ only has eigenvalues $0$ or $1$; consequently, the sum of its eigenvalues equals the rank of $\ERoperator_p$. From this, we deduce
    \[
    \sum_i \lambda_i(\ERoperator_p) = \rank(\ERoperator_p) = \rank(\partial_p^\adjoint)
    = n_{p-1} - \nullity(\partial_p^\adjoint),
    \]
    establishing \cref{eq:foster}.

    For the second statement, let \(\{\alpha_i\}\) be any orthonormal basis of \(C_p(K)\). 
    By \cref{lem:trace}, the trace of an operator equals the sum of the inner products 
    \(\langle \ERoperator_p \alpha_i, \alpha_i \rangle\) over this basis, which coincides with the sum of its eigenvalues. 
    Combining this with the previous argument yields the desired result.
\end{proof}

\begin{corollary}[Higher-dimensional Foster's theorem: matrix form]
    \label[corollary]{cor:foster}
    Let $K$ be a finite weighted simplicial complex and fix a degree $p\geq 0$.
    Let $B_{p}$ be the $p$-th boundary matrix. 
    For any $p$-simplex $\sigma$, let $w(\sigma)$ and $r_\sigma$ be the weight and effective resistance of $\sigma$, respectively.
    Then, 
    \begin{align*}\label{eq:foster_matrix}
    \sum_{\sigma \in K_p} w(\sigma)\, r_\sigma = \rank(B_p^\top) = n_{p-1} - \operatorname{null}(B_p^\top). 
    \end{align*} 
\end{corollary}

\begin{proof}
    By \cref{rmk:relative ER}, the quantity \(w(\sigma)\, r_\sigma\) equals the effective resistance of the normalized chain \(\sqrt{w(\sigma)}\,\sigma\).  
    Since the set \(\{\sqrt{w(\sigma)}\,\sigma\}_{\sigma \in K_p}\) forms an orthonormal basis of \(C_p(K)\), applying \cref{thm:foster} yields
    \[
    \sum_{\sigma \in K_p} w(\sigma)\, r_\sigma
      = \sum_{\sigma \in K_p} 
        \langle \ERoperator_p(\sqrt{w(\sigma)}\,\sigma), \sqrt{w(\sigma)}\,\sigma \rangle
      = n_{p-1} - \nullity(\partial_p^\adjoint)
      = n_{p-1} - \nullity(B_p^\top)
    \]
    where in the last equality we use that nullity is unchanged by multiplication on either side by invertible matrices.
\end{proof}

The graph version of Foster’s theorem can be equivalently expressed in terms of the rank and nullity of the incidence matrix \(B_1\) via the rank–nullity theorem:
\begin{align*}
    \sum_{e \in E} w(e) \, r_e 
    =\rank(B_1^\top)= |V| - \nullity(B_1^T)
\end{align*}
and using that $\nullity(B_1^\top) = \dim\ker \delta_0 = \beta_0(G)$.
Under this reinterpretation, we observe that the classical Foster’s theorem arises as a special case of \cref{cor:foster} when the simplicial complex is a graph and \(p = 1\).

\begin{remark}
    Kook and Lee~\cite{KookLee2018} also provided a higher-dimensional analogue of Foster’s theorem in the setting of simplicial networks. Since their notion of effective resistance is a special case of ours (see \cref{sec:ER-simplicial-matrix-form}), their result follows as a special case of \cref{cor:foster}. Notably, our version, in contrast to that of~\cite{KookLee2018}, does not require any assumptions on the homology of the simplicial complex in dimension \(d-1\), and it applies to simplices in all dimensions, not just those of maximal dimension.
\end{remark}

\section{Spanning Trees and Effective Resistance}
\label{sec:spanning-trees-and-er}

In this section, we establish a connection between the effective resistance of simplices in a simplicial complex and the probability that such simplices appear in a random simplicial spanning tree (see \cref{thm:rel-eff-resistant-spanning-forest}). This generalizes the well-known result for graphs that the effective resistance of an edge is proportional to the probability that such edge appears in a random spanning tree~\cite{BurtonPemantle1993,Kirchhoff1847}. 

\subsection{Relative Effective Resistance and Random Spanning Trees in Graphs}
\label{sec:er-and-simplicial-spanning-trees-graphs}

Let \(G\) be a connected, finite, undirected graph with positive edge weights \(w: E \to \mathbb{R}_+\). For each edge \(e \in E\), we denote by \(r_e\) the effective resistance of the edge \(e\) in the graph \(G\), and refer to the product \(w(e)\, r_e\) as the \emph{relative effective resistance} of the edge \(e\) as in~\cite{DevriendtLambiotte2022}.
From \cref{rmk:relative ER}, we know that the relative effective resistance coincides with the effective resistance of the normalized edge $\sqrt{w(e)}\, e$.

The relative effective resistance contains information about the \emph{importance} of the edge \(e\) in the graph. Namely, for an edge connecting two vertices in the same group of tightly interconnected nodes, the relative effective resistance can be low; whereas for an edge connecting nodes in two hubs of the graph connected by few edges, the relative effective resistance can be high. We can make this idea precise alluding to random spanning trees.

A \textit{spanning tree} is a connected and acyclic subgraph $T$ of $G$ with same vertex set $V$ and number of edges $\vert V \vert -1$. We denote the set of spanning trees $T=(V,E_T)$ in a graph as $\mathbb{T}$. Here we consider \emph{random spanning trees} \(\mathbf{T}\) as random elements from $\mathbb{T}$ sampled with probability proportional to the product of the weights of their edges:
\begin{equation}
\label{eq:random-tree-graph}
    \Pr[\mathbf{T} = T] \coloneqq \dfrac{\prod_{e \in E_T}w(e)}{\sum_{T' \in \mathbb{T}}\prod_{e' \in E_{T'}}w(e')}. 
\end{equation}
In this setting, the following theorem follows from Kirchhoff's Theorem \cite{Kirchhoff1847}.

\begin{theorem}[\cite{BurtonPemantle1993,Kirchhoff1847}]
\label{thm:rel_eff_res_spanning_trees}
    If we sample a tree \(T \in \mathbb{T}\) with probability proportional to the product of its weights, the probability of finding the edge \(e \in E\) in such random spanning tree is the relative effective resistance 
    \begin{equation*}
        \Pr[e \in \mathbf T] = w(e)\,  r_e.
    \end{equation*}
\end{theorem}

\subsection{Higher-dimensional Simplicial Spanning Trees}
\label{sec:high_dim_trees}

Higher-dimensional notions of spanning trees were first investigated by Bolker~\cite{Bolker1976} in the context of \emph{transportation polytopes}, and later extended to \emph{complete simplicial complexes} by Kalai~\cite{Kalai1983}. Duval, Klivans, and Martin subsequently developed the theory for general simplicial complexes and cell complexes \cite{DuvalKlivansMartin2009,DuvalKlivansMartin2011}; see also their later survey~\cite{DuvalKlivansMartin2016}. Related approaches were explored by Lyons~\cite{Lyons2009} and Petersson~\cite{Petersson2009}.

In what follows, we use $|\cdot|$ to denote either the order of an Abelian group or the cardinality of a set, depending on the context.
We recall that $C_p(K)$, $C^p(K)$, and $\Hgroup_p(K)$ denote the chain, cochain, and homology groups with \emph{real coefficients}.
Analogously, one may define these groups with \emph{integer coefficients} (see~\cite{Hatcher2002} for details). In that case, we write $C_p(K;\Zbb)$, $C^p(K;\Zbb)$, and $\Hgroup_p(K;\Zbb)$ for the corresponding groups, and $\tilde{\Hgroup}_p(K;\Zbb)$ and $\tilde{\beta}_p(K)$ for the \emph{reduced homology groups} with integer coefficients and \emph{reduced Betti numbers}.
Note that the reduced $p$-th Betti number is equal to the torsion-free rank of $\tilde\Hgroup_p(K;\Zbb)$, i.e., the rank of the largest free Abelian subgroup of $\tilde\Hgroup_p(K;\Zbb)$.

\begin{definition}[Simplicial Spanning Forest, Definition 2.5, \cite{DuvalKlivansMartin2016}]
Let $K$ be a $d$-dimensional simplicial complex. Given $1 \leq p\leq d$, a \emph{$p$-simplicial spanning forest} $T$ of $K$ is a $p$-dimensional subcomplex $T \subseteq K$ such that: 
\begin{enumerate}
    \item $T_{\leq p-1} = K_{\leq p-1}$,
    \item $\tilde \beta_{p}(T)=0$ (equivalently, $\tilde \Hgroup_p(T; \Zbb) = 0$). 
\end{enumerate}
\end{definition}

\begin{definition} [Simplicial Spanning Tree, Definition 2.6, \cite{DuvalKlivansMartin2016}]
\label[definition]{def:cellular-spanning-tree}
Let $K$ be a $d$-dimensional simplicial complex. Given $1 \leq p\leq d$, a \emph{$p$-simplicial spanning tree} $T$ of $K$ is a $p$-dimensional subcomplex $T \subseteq K$ such that: 
\begin{enumerate}
    \item $T_{\leq p-1} = K_{\leq p-1}$,
    \item $\tilde \beta_{p}(T)=0$ (equivalently, $\tilde \Hgroup_p(T; \Zbb) = 0$), 
    \item $\tilde \beta_{p-1}(T)=0$ (equivalently, $|\tilde \Hgroup_{p-1}(T; \Zbb)| < \infty$).
\end{enumerate}
We denote by $\mathbb{T}_p(K)$ the set of $p$-simplicial spanning trees of $K$.
\end{definition}

\begin{lemma}
    \label[lemma]{lemma:existence-spanning-trees}
    Let \(K\) be a \(d\)-dimensional simplicial complex and let
    \(1\leq p\leq d\). Then any \(p\)-dimensional spanning forest \(T\) of
    \(K\) satisfies
    \[
        \tilde\beta_{p-1}(T)\geq\tilde\beta_{p-1}(K).
    \]
    In particular, if \(\tilde\beta_{p-1}(K)\neq 0\), then \(K\) admits no
    \(p\)-dimensional spanning tree.
\end{lemma}

\begin{proof}
Write $\tilde\partial_q$ for the augmented boundary: $\tilde\partial_q=\partial_q$ for $q\geq1$ and $\tilde\partial_0=\epsilon$, where $\epsilon$ sends a $0$-chain to the sum of its coefficients. 
Restricting to dimensions $\leq p$, we have 
\begin{equation}
    \begin{tikzcd}
\dots \arrow[r, "\tilde\partial_{p+1}^{K}"] & C_p(K) \arrow[r, "\tilde\partial_{p}^{K}"] & C_{p-1}(K) \arrow[r, "\tilde\partial_{p-1}^{K}"] & C_{p-2}(K) \arrow[r, "\tilde\partial_{p-2}^{K}"] & \dots \\
 & C_p(T) \arrow[r, "\tilde\partial_p^T "] \arrow[u, hook] & C_{p-1}(T)\arrow[u, "="]\arrow[r, "\tilde\partial_{p-1}^{T}"] & C_{p-2}(T)\arrow[u, "="]\arrow[r, "\tilde\partial_{p-2}^{T}"]& \cdots    
\end{tikzcd}
\label{eq:commutative-diagram-spanning-tree}
\end{equation}
    Therefore,
    \[
        \tilde\beta_{p-1}(T)
        =
        \dim\ker\tilde\partial_{p-1}^K
        -
        \dim\image\tilde\partial_p^T,
    \]
    whereas
    \[
        \tilde\beta_{p-1}(K)
        =
        \dim\ker\tilde\partial_{p-1}^K
        -
        \dim\image\tilde\partial_p^K.
    \]
    The diagram in \cref{eq:commutative-diagram-spanning-tree} shows that
    \(\image\tilde\partial_p^T\subseteq\image\tilde\partial_p^K\),
    hence \(\tilde \beta_{p-1}(K) \leq \tilde \beta_{p-1}(T)\).
\end{proof}

The preceding lemma provides a necessary condition for spanning trees to exist. We next give an equivalent cardinality characterization of the two homological conditions in \cref{def:cellular-spanning-tree}. In particular, conditions 2 and 3 in \cref{def:cellular-spanning-tree} can be exchanged for any two out of the three conditions in \cref{rem:2-out-3}.

\begin{lemma}
\label[lemma]{rem:2-out-3}
For any $p$-dimensional subcomplex $T \subseteq K$ with $T_{\leq p-1} = K_{\leq p-1}$, any two of the three following conditions imply the third:
    \begin{enumerate}
        \item $\tilde \beta_{p}(T)=0$, 
        \item $\tilde \beta_{p-1}(T)=0$,
        \item $|T_p| = |K_p| - \tilde\beta_p(K_{\leq p}) + \tilde\beta_{p-1}(K_{\leq p})$.\label{third_condition_CSP}
    \end{enumerate}
\end{lemma}

\begin{proof} 
Restricting \cref{eq:commutative-diagram-spanning-tree} to the chain complex of $K_{\leq p}$, we have 
\begin{equation}
    \begin{tikzcd}
0 \arrow[r, "\tilde\partial_{p+1}^{K_{\leq p}}"] & C_p(K_{\leq p}) \arrow[r, "\tilde\partial_{p}^{K_{\leq p}}"] & C_{p-1}(K_{\leq p}) \arrow[r, "\tilde\partial_{p-1}^{K_{\leq p}}"] & C_{p-2}(K_{\leq p}) \arrow[r, "\tilde\partial_{p-2}^{K_{\leq p}}"] & \dots \\
 & C_p(T) \arrow[r, "\tilde\partial_p^T "] \arrow[u, hook] & C_{p-1}(T)\arrow[u, "="]\arrow[r, "\tilde\partial_{p-1}^{T}"] & C_{p-2}(T)\arrow[u, "="]\arrow[r, "\tilde\partial_{p-2}^{T}"]& \cdots    
\end{tikzcd}
\label{eq:commutative-diagram-spanning-tree-2}
\end{equation}
so we can write
\begin{align}
    |K_p| - \tilde \beta_p(K_{\leq p}) +\tilde \beta_{p-1}(K_{\leq p}) & = |K_p| - \dim \ker \tilde\partial_{p}^{K_{\leq p}} + \dim \ker \tilde\partial_{p-1}^{K_{\leq p}} - \dim \image \tilde\partial_{p}^{K_{\leq p}} \nonumber \\
    & = |K_p| -|K_p| + \dim \ker \tilde\partial_{p-1}^{K_{\leq p}}  = \dim \ker \tilde\partial_{p-1}^{K_{\leq p}} \label{eq:|T_p|}
\end{align}
where in the second to last equality we use $|K_p| = \dim C_p(K_{\leq p}) =  \dim \ker \tilde\partial_{p}^{K_{\leq p}} +\dim \image \tilde\partial_{p}^{K_{\leq p}}  $. 

Similarly, since \(T_{\leq p-1}=K_{\leq p-1}\), we have
\(\tilde\partial_{p-1}^{T}
=\tilde\partial_{p-1}^{K_{\leq p}}\), and hence
\[
|T_p|
=
\tilde\beta_p(T)
+\dim\ker\tilde\partial_{p-1}^{K_{\leq p}}
-\tilde\beta_{p-1}(T).
\]
Thus condition~3 is equivalent to $\tilde\beta_p(T)=\tilde\beta_{p-1}(T)$, and any two of the three conditions imply the third.
\end{proof}

\begin{remark}[Relation with earlier definitions]\label[remark]{rem:cellular-spanning-tree}
The formulation in \cref{rem:2-out-3} makes explicit that the cardinality condition for a \(p\)-dimensional spanning tree depends on the \(p\)-skeleton of the ambient complex. We briefly place this formulation in the context of earlier definitions of higher-dimensional spanning trees.

In \cite[Definition~3.1]{DuvalKlivansMartin2009}, Duval, Klivans, and Martin define a \(p\)-dimensional simplicial spanning tree \(T\) of a simplicial complex \(K\) by requiring
\[
    T_{\leq p-1}=K_{\leq p-1},
    \qquad
    \widetilde H_p(T;\mathbb Z)=0,
    \qquad
    \bigl|\widetilde H_{p-1}(T;\mathbb Z)\bigr|<\infty,
\]
together with the cardinality condition
\begin{equation}
    |T_p|
    =
    |K_p|
    -\widetilde\beta_p(K)
    +\widetilde\beta_{p-1}(K).
    \label{eq:DKM-original-cardinality}
\end{equation}
The cellular formulation in \cite[Definition~2.2]{DuvalKlivansMartin2011} uses an analogous formula. As written, the Betti numbers in \cref{eq:DKM-original-cardinality} are those of the full ambient complex \(K\). When \(p=\dim K\), this agrees exactly with condition~\ref{third_condition_CSP} of \cref{rem:2-out-3}. For \(p<\dim K\), however, \((p+1)\)-dimensional simplices may change \(\widetilde\beta_p(K)\) while leaving the \(p\)-skeleton, and hence the data defining a \(p\)-dimensional spanning tree, unchanged.

A subsequent formulation in \cite[Definition~3.2]{DuvalKlivansMartin2013} makes the skeletal viewpoint explicit: an \(i\)-dimensional spanning tree of \(K\) is defined to be a spanning tree of its \(i\)-dimensional skeleton \(K_{\leq i}\). Under this convention, the corresponding cardinality condition is precisely
\[
    |T_p|
    =
    |K_p|
    -\widetilde\beta_p(K_{\leq p})
    +\widetilde\beta_{p-1}(K_{\leq p}),
\]
which is the formulation appearing in \cref{rem:2-out-3}. The later treatment in \cite[Definitions~2.5--2.6]{DuvalKlivansMartin2016} adopts the homological notions of spanning forests and spanning trees used in \cref{def:cellular-spanning-tree}; the cardinality identity above then follows from the two-out-of-three property established in \cref{rem:2-out-3}.

The distinction between \(K\) and \(K_{\leq p}\) is immaterial in top dimension but becomes important for spanning trees in positive codimension. For example, let \(K\) be the \(2\)-dimensional simplicial complex consisting of two filled triangles sharing an edge, as in \cref{fig:SC_two_triangles}. The complex \(K\) is contractible, so
\[
    \widetilde\beta_1(K)
    =
    \widetilde\beta_0(K)
    =
    0.
\]
Its \(1\)-skeleton, however, has five edges and two independent cycles, and therefore
\[
    \widetilde\beta_1(K_{\leq 1})=2,
    \qquad
    \widetilde\beta_0(K_{\leq 1})=0.
\]
Consequently, a \(1\)-dimensional spanning tree \(T\) has
\[
    |T_1|
    =
    |K_1|
    -\widetilde\beta_1(K_{\leq 1})
    +\widetilde\beta_0(K_{\leq 1})
    =
    5-2+0
    =
    3
\]
edges. By contrast, substituting the Betti numbers of the full complex into the cardinality expression in \cref{eq:DKM-original-cardinality} gives
\[
    |K_1|
    -\widetilde\beta_1(K)
    +\widetilde\beta_0(K)
    =
    5.
\]
This example illustrates why the skeletal formulation is necessary when the dimension of the spanning tree is smaller than that of the ambient complex.

\begin{figure}[H]
\centering
\begin{tikzpicture}
\coordinate (A) at (0, 0);
\coordinate (B) at (2, 0);
\coordinate (C) at (1, 1.73); 
\coordinate (D) at (3, 1.73); 

\draw[thick] (A) -- (B) -- (C) -- cycle; 
\draw[thick] (B) -- (D) -- (C) -- cycle; 

\draw[red, ultra thick] (A) -- (B); 
\draw[red, ultra thick] (B) -- (C); 
\draw[red, ultra thick] (C) -- (D); 

\fill[gray!50, opacity=0.6] (A) -- (B) -- (C) -- cycle; 
\fill[gray!50, opacity=0.6] (B) -- (D) -- (C) -- cycle; 

\fill[black] (A) circle (2pt);
\fill[black] (B) circle (2pt);
\fill[black] (C) circle (2pt);
\fill[black] (D) circle (2pt);

\node[below left] at (A) {0};
\node[below right] at (B) {1};
\node[above] at (C) {2};
\node[above right] at (D) {3};
\end{tikzpicture}
\caption{A \(2\)-dimensional simplicial complex \(K\) consisting of two filled triangles sharing an edge. The \(1\)-dimensional spanning tree
\(T=\{(0,1),(1,2),(2,3)\}\) is shown in red. Although \(K\) is contractible, its \(1\)-skeleton satisfies
\(\widetilde\beta_1(K_{\leq 1})=2\), and hence a spanning tree of \(K_{\leq 1}\) contains three edges.}
\label{fig:SC_two_triangles}
\end{figure}
\end{remark}

A related formulation appears in subsequent work by Duval, Kook, Lee, and Martin~\cite{DuvalKookLee2023}, where the authors consider \(p\)-dimensional subcomplexes \(T\subseteq K\) satisfying
\[
    T_{\leq p-1}=K_{\leq p-1},
    \qquad
    \tilde\beta_p(T)=0,
    \qquad
    \tilde\beta_{p-1}(T)=\tilde\beta_{p-1}(K).
\]
As we show next, these conditions characterize \emph{maximal} \(p\)-dimensional spanning forests, that is, spanning forests containing the maximal possible number of \(p\)-simplices. Unlike spanning trees, which by \cref{lemma:existence-spanning-trees} can exist only when \(\tilde\beta_{p-1}(K)=0\), maximal spanning forests always exist. When \(\tilde\beta_{p-1}(K)=0\), the two notions coincide. Related characterizations of maximal spanning forests are stated in \cite[Proposition~2.13(f)--(g)]{DuvalKlivansMartin2016}.

\begin{proposition}
\label[proposition]{prop:characterization-maximal-forests}
    Let $K$ be a $d$-dimensional complex and $T\subseteq K$ a $p$-dimensional subcomplex, $0< p \leq d$, with $T_{\leq p-1} = K_{\leq p-1}$. The following conditions are equivalent:
    \begin{enumerate}
        \item $T$ is a maximal spanning forest of $K$ (that is, a spanning forest with maximal number of $p$-simplices),
        \item $\tilde \beta_p(T) =0$ and $\image \tilde\partial_p^T = \image \tilde\partial_p^K$,
        \item $\tilde \beta_p(T) =0$ and $\tilde\beta_{p-1}(T) = \tilde \beta_{p-1}(K)$.
    \end{enumerate}
\end{proposition}

\begin{proof}
    Suppose first that \(T\) is a \(p\)-dimensional spanning forest. Since
    \(T\) is \(p\)-dimensional, we have
    \[
        \tilde\beta_p(T)
        =
        \dim\ker\tilde\partial_p^T
        =
        0.
    \]
    Hence \(\tilde\partial_p^T\) is injective, and therefore
    \[
        |T_p|
        =
        \dim C_p(T)
        =
        \dim\image\tilde\partial_p^T.
    \]
    Thus, among \(p\)-dimensional spanning forests, maximizing the number
    of \(p\)-simplices is equivalent to maximizing
    \(\dim\image\tilde\partial_p^T\).

    Since \(T\subseteq K\), we have
    \[
        \image\tilde\partial_p^T
        \subseteq
        \image\tilde\partial_p^K.
    \]
    Moreover, a basis of \(\image\tilde\partial_p^K\) can be chosen among the boundaries of the \(p\)-simplices of \(K\). The corresponding \(p\)-simplices, together with \(K_{\leq p-1}\), therefore form a \(p\)-dimensional spanning forest whose boundary image is \(\image\tilde\partial_p^K\). 
    Consequently, the maximal possible value of \(\dim\image\tilde\partial_p^T\) is \(\dim\image\tilde\partial_p^K\), and it is attained if and only if
    \[
        \image\tilde\partial_p^T
        =
        \image\tilde\partial_p^K.
    \]
    This proves the equivalence between conditions \(1\) and \(2\).

    By \cref{lemma:existence-spanning-trees}, we know that \(\tilde \beta_{p-1}(K) \leq \tilde \beta_{p-1}(T)\). Hence,
    \[
        \image\tilde\partial_p^T
        =
        \image\tilde\partial_p^K
        \quad\Longleftrightarrow\quad
        \tilde\beta_{p-1}(T)
        =
        \tilde\beta_{p-1}(K).
    \]
    This proves the equivalence between conditions \(2\) and \(3\), and
    concludes the proof.
\end{proof}

\begin{corollary}
    Let $K$ be a $d$-dimensional simplicial complex with $\tilde \beta_{p-1}(K)=0$ and let $T\subseteq K$ be a $p$-dimensional subcomplex, $0< p \leq d$, with $T_{\leq p-1} = K_{\leq p-1}$. Then, the following are equivalent:
    \begin{enumerate}
        \item $T$ is a $p$-dimensional maximal spanning forest,
        \item $T$ is a $p$-dimensional spanning tree. 
    \end{enumerate}
\end{corollary}

\begin{proof}
By \cref{prop:characterization-maximal-forests}, a maximal spanning forest satisfies \(\tilde\beta_{p-1}(T)=\tilde\beta_{p-1}(K)=0\), and hence is a spanning tree by \cref{def:cellular-spanning-tree}. Analogously, a $p$-spanning tree $T$ has $\tilde \beta_p(T) =0$ and $\tilde\beta_{p-1}(T) = 0 = \tilde\beta_{p-1}(K)$ by assumption, thus by \cref{prop:characterization-maximal-forests}, it is a maximal spanning forest.
\end{proof}

We close the section with the following lemma relating the rank of the Laplacian of a maximal simplicial spanning forest $T$ and that of the complex $K$ in which it is contained.

\begin{lemma}
\label[lemma]{prop:range_laplacian_spanning_forest}
Let \(K\) be a \(d\)-dimensional simplicial complex, $d\geq1$, and let \(T\) be a $d$-dimensional maximal simplicial spanning forest. Then
\begin{equation}
\label{eq:rank_laplacian_spanning_forest}
\rank \bigl(\chainLup{d-1}(K)\bigr)
  = \rank \bigl(\chainLup{d-1}(T)\bigr)
  = |T_d|
  = |K_d| - \tilde\beta_d(K).
\end{equation}
\end{lemma}
\begin{proof}
Notice that, since $d>0$, $\partial_d^K = \tilde \partial_d^K$, and the same applies to $T$. By definition, $\chainLup{d-1}(K) = \partial_d^K \circ (\partial_d^K)^\adjoint$, so by a standard property of adjoint operators (see \cref{eq:adjoint-image}), we have $\image(\chainLup{d-1}(K)) = \image(\partial_d^K)$. 
Therefore,
\[
\rank \chainLup{d-1}(K)
  = \dim \image(\partial_d^K)
  = \dim C_d(K) - \dim \ker \partial_d^K
  = |K_d| - \tilde\beta_d(K),
\]
where the second equality follows from rank--nullity, 
and the last from the fact that \(K\) is \(d\)-dimensional,
hence \(\dim \ker \partial_d^K = \dim \ker\tilde \partial_d^K = \tilde\beta_d(K)\).

Applying the same argument to \(T\) gives
\[\rank \chainLup{d-1}(T) = |T_d| - \tilde\beta_d(T) = |T_d|\]
since by definition a spanning forest has
\(\tilde\beta_d(T) = 0\).

Lastly, since the forest is maximal, we get $ \image \tilde\partial_d^T = \image \tilde\partial_d^K$, and thus $\image(\chainLup{d-1}(T)) = \image(\chainLup{d-1}(K))$, connecting the two equalities above and concluding the proof.
\end{proof}

\subsection{Tree Numbers}
\label{sec:tree-numbers}

We now present \emph{simplicial tree numbers}, which extend the formulas counting spanning trees in graphs to simplicial complexes. The definitions here are taken from those in~\cite[Section 4]{DuvalKlivansMartin2009} for the unweighted case, and in~\cite[Section 2.6]{DuvalKlivansMartin2011} for the weighted version.
Recall that $\mathbb{T}_p(K)$ denotes the set of all $p$-spanning trees of $K$.

\begin{definition}[Tree numbers]
\label[definition]{def:tree_number}
    Given a $d$-dimensional simplicial complex $K$, its $p$-dimensional \emph{tree number} is defined as
    \begin{equation}
    \label{eq:tree_number}
        \tau_p (K) \coloneqq \sum_{T \in \mathbb{T}_p(K)} \abs{\tilde\Hgroup_{p-1}(T; \mathbb{Z})}^2, \quad 1\leq p \leq d.
    \end{equation}
Note that the sum above is finite by condition 3 of \cref{def:cellular-spanning-tree}. This justifies the restriction to spanning trees, and not forest or maximal forests.
    
If $K$ is weighted by $w : K \to \Rspace_{> 0}$, we define the \emph{weighted tree number}, where each tree is weighted by the product of the weights of the top-dimensional simplices appearing in the tree, as
\begin{equation}
\label{eq:wtree_number}
    \tau_p (K, w) \coloneqq \sum_{T \in \mathbb{T}_p(K)} \abs{\tilde\Hgroup_{p-1}(T; \mathbb{Z})}^2 \prod_{\sigma \in T_p} w(\sigma), \quad 1\leq p \leq d.
\end{equation}
If all weights are equal to $1$, that is, our simplicial complex is unweighted, we recover $\tau_p(K)$.
In degree $0$, we take $\mathbb{T}_0(K)$ to be the set of single-vertex subcomplexes and set $\tau_0(K,w)\coloneqq\sum_{v\in K_0}w(v)$ and $\tau_0(K)\coloneqq|K_0|$. For nonempty $K$, we use $\tilde H_{-1}(K;\mathbb Z)=0$, so $|\tilde H_{-1}(K;\mathbb Z)|=1$.
\end{definition}

\subsection{Exterior Powers and Pseudo-Determinants}

We collect here the exterior-algebra identities used in the proofs of the weighted matrix-tree theorem (\cref{thm:matrix-tree-thm-sc}) and the spanning-tree interpretation of effective resistance (\cref{thm:rel-eff-resistant-spanning-forest}). Exterior powers allow us to express determinants and squared volumes intrinsically, without base choices; while the Hilbert--Schmidt norm will be used to work with these quantities as sums over orthonormal simplices. Unless otherwise stated, all vector spaces are finite-dimensional and real. We begin with some definitions.

Recall that a linear map $\varphi: X \to Y$ between finite-dimensional vector spaces $X$ and $Y$ induces a map $ \Lambda^r \varphi: \Lambda^r X \to \Lambda^r Y$ over their exterior algebras, satisfying
\[\Lambda^r \varphi\, (v_1\wedge \dots \wedge v_r) = \varphi(v_1) \wedge \dots \wedge \varphi(v_r), \quad \forall v_1, \dots, v_r \in X.\]
For a set of indices $I=(i_1< i_2 < \dots < i_r)$, we denote \[v_I \coloneqq v_{i_1} \wedge v_{i_2} \wedge \dots \wedge v_{i_r}.\] 
Notice that for $n= \dim(X)$, $\Lambda^n X $ is one dimensional.
\begin{remark}[Exterior powers of free abelian groups]
\label[remark]{remark:exterior-algebra-free-abelian-groups}
    Exterior powers also apply to free abelian groups of finite rank. Indeed, viewing such a group $F$ as a subgroup of $F_{\mathbb R}\coloneqq F\otimes_{\mathbb Z}\mathbb R$, we define $\Lambda_{\mathbb Z}^kF$ as the subgroup of $\Lambda^kF_{\mathbb R}$ generated by wedges of elements of $F$. If $f_1,\dots,f_r$ is an integral basis of $F$, multilinearity and alternation show that the wedges $f_{i_1}\wedge\cdots\wedge f_{i_k}$, with $i_1<\cdots<i_k$, form an integral basis of $\Lambda_{\mathbb Z}^kF$. In particular, $\Lambda_{\mathbb Z}^rF$ is free abelian of rank one, generated by $f_1\wedge\cdots\wedge f_r$, uniquely up to sign. Any homomorphism $F\to G$ extends to a real linear map and hence induces a homomorphism on these exterior powers by the same wedge formula as above.
\end{remark}

\begin{definition}
\label{def:determinant}
    The \emph{determinant} of an endomorphism $\varphi: X \to X$ of an $n$-dimensional vector space is defined as the number $\det(\varphi)$ such that
$$\Lambda^n\varphi\, (\omega) = \det(\varphi)\, \omega, \quad \forall \omega\in \Lambda^n X, \ \omega \neq 0. $$
\end{definition}

An inner product $\langle\cdot , \cdot\rangle$ on $ X$ induces an inner product on $\Lambda^rX$ as the determinant of the Gram matrix: 
\[\langle v_1\wedge\cdots\wedge v_r,w_1\wedge\cdots\wedge w_r\rangle \coloneqq \det(\langle v_i,w_j\rangle_{i,j})\]
for any pair of tuples $v_1, \dots, v_r \in X$ and $w_1, \dots w_r \in X$.
From this definition, we observe that if $\{e_1, \dots, e_n\}$ is an orthonormal basis of $X$, then the vectors $e_{i_1} \wedge \dots \wedge e_{i_r}$ with $i_1 < \dots < i_r$, constitute an orthonormal basis of $\Lambda^rX$, and that the norm of a vector $v_1 \wedge \dots \wedge v_r \in \Lambda^r X$ is given by 
\begin{equation}
    \label{eq:gram-equation-exterior-algebra}
    \norm{v_1 \wedge \dots \wedge v_r}^2 = \det \left(\langle v_i, v_j\rangle_{i,j =1}^r\right).
\end{equation}

\begin{definition}[Hilbert--Schmidt norm]
\label{def:hilbert-schmidt-norm}
    For a linear map $\varphi:X\to Y$ between inner product spaces, the \emph{Hilbert--Schmidt norm} is defined as
    \[
    \|\varphi\|_{\mathrm{HS}}^2\coloneqq\sum_j\|\varphi(e_j)\|^2,
    \]
    where $\{e_j: j \in J\}$ is any orthonormal basis of $X$
\end{definition}
\begin{remark}
    The Hilbert--Schmidt norm of an operator is independent of the chosen orthonormal basis.
\end{remark}

\begin{definition}[Pseudo-determinant]
\label{def:pseudodeterminant}
    For an endomorphism $\varphi: X \to X$, its \emph{pseudo-determinant} $\lambda(\varphi)$ is the product of its nonzero eigenvalues, counted with algebraic multiplicity.
\end{definition}
\begin{remark}
    We take an empty product to be $1$. We also use the conventions $\Lambda^0X=\mathbb R$ and $\Lambda^0 \varphi=\operatorname{id}_{\mathbb R}$.
\end{remark}

The following identity converts the pseudo-determinant of $\varphi \varphi^\adjoint$ into a sum of squared exterior coefficients. This will be the starting point for the weighted matrix-tree theorem (\cref{thm:matrix-tree-thm-sc}) proof.

\begin{lemma}[Pseudo-determinants and exterior powers]
\label[lemma]{lem:pseudodeterminant-exterior-power}
Let $\varphi:X\to Y$ be a linear map between finite-dimensional inner product spaces, and let $r=\rank \varphi$. Then
\begin{equation}
\label{eq:pseudodeterminant-exterior-power}
    \lambda(\varphi \varphi^*) = \norm{\Lambda^r\varphi}_{\mathrm{HS}}^2.
\end{equation}

Equivalently, if $\{e_I\}_{|I|=r}$ and $\{f_J\}_{|J|=r}$ are orthonormal bases of $\Lambda^rX$ and $\Lambda^rY$, then
\begin{equation}
\label{eq:pseudodeterminant-exterior-coordinate-expansion}
    \lambda(\varphi\varphi^*) =
    \sum_{\substack{|I|=r\\|J|=r}}
    \abs{
      \left\langle
      \Lambda^r\varphi(e_I),f_J
      \right\rangle
    }^2.
\end{equation}
\end{lemma}

\begin{proof}
Let $s_1,\dots,s_r>0$ be the nonzero singular values of $\varphi$. The nonzero eigenvalues of $\varphi \varphi^*$ are $s_1^2,\dots,s_r^2$, and therefore
\[
    \lambda(\varphi\varphi^*)=\prod_{i=1}^r s_i^2.
\]
On the other hand, $\Lambda^r \varphi$ has rank one and its unique nonzero singular value is $\prod_{i=1}^r s_i$. Hence
\[
    \norm{\Lambda^r\varphi}_{\mathrm{HS}}^2
    =
    \prod_{i=1}^r s_i^2
    =
    \lambda(\varphi \varphi^*).
\]
The second formula is the standard expansion of the Hilbert--Schmidt norm in orthonormal bases.
\end{proof}

We next record two identities for compositions and the pseudo-determinant.

\begin{lemma}[Reversibility and multiplicativity of the pseudo-determinant]
\label[lemma]{lem:lambda_multiplicative}
Let $\varphi:X\to Y$ and $\phi:Y\to Z$ be linear maps between
finite-dimensional vector spaces.
\begin{itemize}
    \item If $Z=X$, then
    \[
        \lambda(\phi\circ\varphi)
        =
        \lambda(\varphi\circ\phi).
    \]
    \item Suppose $X=Y=Z$ is an inner product space, and that $\varphi$ and $\phi$ are
    self-adjoint positive semidefinite endomorphisms of the same rank.
    If
    \[
        \image(\varphi)\subseteq\kernel(\phi)^\perp,
    \]
    then
    \[
        \lambda(\phi\circ\varphi)
        =
        \lambda(\phi)\lambda(\varphi).
    \]
\end{itemize}
\end{lemma}

\begin{proof}
The first assertion follows from the fact that
$\phi\circ\varphi$ and $\varphi\circ\phi$ have the same nonzero
eigenvalues, counted with algebraic multiplicity.

For the second assertion, self-adjointness gives
\[
    \kernel(\varphi)^\perp=\image(\varphi),
    \qquad
    \kernel(\phi)^\perp=\image(\phi).
\]
Since the two operators have the same rank and
$\image(\varphi)\subseteq\kernel(\phi)^\perp$, we have
\[
    \image(\varphi)=\kernel(\phi)^\perp.
\]
Thus both restrictions
\[
    \varphi|_{\kernel(\varphi)^\perp}
    \quad\text{and}\quad
    \phi|_{\kernel(\phi)^\perp}
\]
are positive definite endomorphisms of spaces of the same dimension, and
$\phi\circ\varphi$ has kernel $\kernel(\varphi)$. Therefore, the product of the non-zero eigenvalues is given by the determinant of the operators restricted to the orthogonal complement of their kernels, and we have:
\begin{align*}
\lambda(\phi\circ\varphi)
&=
\det\!\left(
(\phi\circ\varphi)|_{\kernel(\varphi)^\perp}
\right)\\
&=
\det\!\left(
\phi|_{\kernel(\phi)^\perp}
\right)
\det\!\left(
\varphi|_{\kernel(\varphi)^\perp}
\right)\\
&=
\lambda(\phi)\lambda(\varphi).
\end{align*}
\end{proof}

Finally, we present a volume identity that separates one vector from the remaining factors of a wedge product.

\begin{lemma}[Orthogonal decompositions in exterior powers]
\label[lemma]{lem:exterior-norm-projection}
Let $X$ be a finite-dimensional inner product space, let $\gamma\in X\setminus\{0\}$, and let $v_1,\dots,v_k\in X$, where $0\leq k\leq\dim X-1$. Then
\[
\left\|\gamma\wedge v_1\wedge\cdots\wedge v_k\right\|^2
=
\|\gamma\|^2
\left\|
P_{\gamma^\perp}v_1\wedge\cdots\wedge
P_{\gamma^\perp}v_k
\right\|^2.
\]
For $k=0$, the empty wedge is understood to be $1\in\Lambda^0X$, with norm $1$.
\end{lemma}

\begin{proof}
The case $k=0$ is immediate. Otherwise, write $w_i=P_{\gamma^\perp}v_i$. Each difference $v_i-w_i$ is a multiple of $\gamma$, so multilinearity and the alternating property of the wedge product give
\[
\gamma\wedge v_1\wedge\cdots\wedge v_k
=
\gamma\wedge w_1\wedge\cdots\wedge w_k.
\]
Since $\langle\gamma,w_i\rangle=0$ for every $i$, the definition of the induced inner product on exterior powers yields
\begin{align*}
\left\|\gamma\wedge w_1\wedge\cdots\wedge w_k\right\|^2
&=
\det
\begin{pmatrix}
\|\gamma\|^2 & 0\\
0 & \bigl(\langle w_i,w_j\rangle\bigr)_{i,j=1}^k
\end{pmatrix}\\
&=
\|\gamma\|^2
\det\bigl(\langle w_i,w_j\rangle\bigr)_{i,j=1}^k\\
&=
\|\gamma\|^2
\left\|w_1\wedge\cdots\wedge w_k\right\|^2,
\end{align*}
which proves the identity.
\end{proof}

\subsection{Simplicial Weighted Matrix-Tree Theorem}

The purpose of this section is to prove our simplicial weighted matrix-tree theorem (\cref{thm:matrix-tree-thm-sc}).
Matrix-tree theorems arise in graph theory, relating the determinant of a submatrix of the graph Laplacian with the number of spanning trees in the graph, and they are first attributed to Kirchhoff in his seminal work about circuit theory \cite{Kirchhoff1847}.
The extension of these ideas to simplicial complexes was first studied by Kalai \cite{Kalai1983} for \textit{complete simplicial complexes}. Duval, Klivans and Martin \cite{DuvalKlivansMartin2009} later extended matrix-tree theorems to \textit{shifted simplicial complexes}, introducing a weighted version. Other unweighted generalizations are due to Adin \cite{Adin1992} for \textit{complete colorful complexes} and to Duval, Klivans and Martin \cite{DuvalKlivansMartin2011} for \textit{hypercubes}. Finally, a weighted matrix-tree theorem for cell complexes has also been recently introduced by Aalipour et al.~\cite{AalipourDuvalKook2018}, applied to \textit{colorful} and \textit{cubical complexes}. Our version is very similar in formulation to this last one, but follows an original, different proof strategy that makes it basis-free.

The first two results (\cref{prop:kernel_relative_boundary} and \cref{prop:determinat_kernel_relative_boundary}) of this section are largely inspired by \cite[Section~4]{DuvalKlivansMartin2009} and \cite[Section 2.5]{DuvalKlivansMartin2011}, although our approach here differs in two key aspects. First, our results are basis-independent and apply to the up Laplacian operator of simplicial complexes, whereas previous results were formulated for specific matrix representations of this Laplacian. We also work under slightly different homological assumptions. In particular, Duval, Klivans, and Martin~\cite{DuvalKlivansMartin2009,DuvalKlivansMartin2011} require simplicial complexes to be acyclic in positive codimension (APC), i.e., for a $d$-dimensional complex $K$, they assume $\tilde{\beta}_p(K)=0$ for $0\leq p< d$. In contrast, for the results below it is sufficient to assume \(\tilde\beta_{d-1}(K)=\tilde\beta_{d-2}(K)=0.\)

Given a subcomplex $A \subset K$ we denote by $C_p(K,A;\Zbb)$ the quotient group $C_p(K; \Zbb)/C_p(A; \Zbb)$. For $p\geq 0$, the boundary operator of free abelian groups 
\[\partial_p^{K}: C_{p}(K; \Zbb) \to C_{p-1}(K; \Zbb)\] 
induces a \emph{relative} boundary operator \[\partial_{p}^{K, A}: C_{p}(K,\,A;\Zbb) \to C_{p-1}(K,A;\Zbb)\] 
from which we define the \textit{relative homology groups} 
\[\Hgroup_p(K,A;\Zbb) \coloneqq \ker \partial_{p}^{K,A}/ \mathrm{im}\, \partial_{p+1}^{K,A}.\]

For the remainder of the section, let $d\geq 1$ and $K$ be a $d$-dimensional simplicial complex with $$\tilde \beta_{d-1}(K) = \tilde \beta_{d-2}(K) = 0.$$
Let $\TopSubset \subset K_d$ be such that $|\TopSubset| = |K_d| - \tilde \beta_d(K)$, and $\FaceSubset \subset K_{d-1}$ such that $|\FaceSubset| = |\TopSubset|$.

Define
$$\TopTree \coloneqq \TopSubset \cup K_{\leq d-1}, \quad \overline{\FaceSubset} \coloneqq K_{d-1} \setminus \FaceSubset, \quad \text{and} \quad \FaceTree \coloneqq \overline{\FaceSubset} \cup K_{\leq d-2}$$
and note that \(\FaceTree \subseteq K_{\leq d-1} \subseteq \TopTree\).

\begin{proposition}
\label[proposition]{prop:kernel_relative_boundary}
    Let $K$ be a $d$-dimensional simplicial complex with $\tilde \beta_{d-1}(K) = \tilde \beta_{d-2}(K) = 0$, and let $\TopSubset \subset K_d$ and $\FaceSubset \subset K_{d-1}$ be as above.
    The relative boundary operator 
    \[
    \partial_{d}^{\TopTree,\, \FaceTree}: C_{d}(\TopTree, \FaceTree;\Zbb)\to C_{d-1}(\TopTree, \FaceTree;\Zbb) 
    \] 
    of the pair $\FaceTree \subset \TopTree$  has $\ker \partial_{d}^{\TopTree,\, \FaceTree}= 0 $ if and only if $\TopTree \in \mathbb{T}_d(K)$ and $\FaceTree \in \mathbb{T}_{d-1}(K)$.
\end{proposition}

\begin{proof}
    For notational simplicity, we drop the coefficient ring $\Zbb$ in the chain groups and homology groups throughout this proof.

    The relative boundary operator $\partial_d^{\TopTree, \FaceTree}$ has trivial kernel if and only if $\tilde H_d(\TopTree, \FaceTree) = 0$, since $d$ is the top dimension of $K$ and therefore no nontrivial $d$-boundaries exist.

    By the standard theory of relative homology groups (see \cite[Theorem 2.16]{Hatcher2002}) we know that the relative homology group fits into a long exact sequence:
    \begin{equation}
    \label{eq:les_relative_homology}
        0 \to \tilde \Hgroup_d(\FaceTree) = 0 \to \tilde \Hgroup_d(\TopTree) \to \tilde \Hgroup_{d}(\TopTree, \FaceTree) \to \tilde \Hgroup_{d-1}(\FaceTree) \to \tilde \Hgroup_{d-1}(\TopTree) \to \dots.
    \end{equation}

    \begin{itemize}
        \item If $\tilde \Hgroup_d(\TopTree, \FaceTree) \neq 0$ then we must have $\tilde \Hgroup_d(\TopTree)\neq 0$ or $\tilde \Hgroup_{d-1}(\FaceTree) \neq 0$, meaning that either $\TopTree \notin \mathbb{T}_d(K)$ or $\FaceTree \notin \mathbb{T}_{d-1}(K)$. This proves the if part of the proposition.

        \item Assume that $\tilde \Hgroup_d(\TopTree, \FaceTree) = 0$.
        We also know that $\tilde \Hgroup_d(\FaceTree) = 0$ because $\dim(\FaceTree) = d-1$. From the exactness of \cref{eq:les_relative_homology} we conclude that $\tilde \Hgroup_d(\TopTree) = 0$ and by definition we know
        \[|(\TopTree)_d| = |\TopSubset| = |K_d| - \tilde\beta_d(K) = |K_d| - \tilde{\beta}_d(K_{\leq d}) +\tilde{\beta}_{d-1}(K_{\leq d})\]
        where we also use that $K= K_{\leq d}$ because $K$ is $d$-dimensional and that $\tilde{\beta}_{d-1}(K) = 0$. This proves that $\TopTree$ is a spanning $d$-tree, i.e., $\TopTree \in \mathbb{T}_d(K)$.

        We now want to show that $\FaceTree \in \mathbb{T}_{d-1}(K)$. By Condition 3 of the definition of spanning tree, $\tilde \Hgroup_{d-1}(\TopTree)$ must be finite. The exactness of \cref{eq:les_relative_homology} implies that the map $\tilde \Hgroup_{d-1}(\FaceTree) \to \tilde \Hgroup_{d-1}(\TopTree)$ is injective, which in turn means that $\tilde \Hgroup_{d-1}(\FaceTree)$ is finite. As the top dimensional homology group, it is also torsion-free, which allows us to conclude $\tilde \Hgroup_{d-1}(\FaceTree)=0$. 
        Finally, we know that $|\FaceSubset|=|\TopSubset| = |K_d| - \tilde\beta_d(K)$  and $\tilde\beta_{d-1}(K)=\tilde\beta_{d-2}(K)=0$. Consequently,
        \[|(\FaceTree)_{d-1}|=|\bar \FaceSubset|=|K_{d-1}|-|\FaceSubset|=|K_{d-1}|- |K_d| + \tilde\beta_d(K) \] 
        We observe that 
        $$\tilde\beta_d(K)=\dim \tilde \Hgroup_d(K;\mathbb{R})=\dim \ker \tilde \partial ^{K;\mathbb{R}}_{d}=|K_d|-\dim \mathrm{im}\, \tilde\partial ^{K;\mathbb{R}}_{d}=|K_d|-\dim \ker\tilde \partial ^{K;\mathbb{R}}_{d-1},$$ 
        where in the last equality we use $\tilde\beta_{d-1}(K)=0$. Hence,  we obtain 
        \begin{align}
            |(K_{\bar \FaceSubset})_{d-1}| &=|K_{d-1}|- \dim \ker\tilde \partial ^{K;\mathbb{R}}_{d-1}\nonumber\\
            \, & = |K_{d-1}|- \tilde{\beta}_{d-1}(K_{\leq d-1})\nonumber \\
            \, & = |K_{d-1}|- \tilde{\beta}_{d-1}(K_{\leq d-1}) + \tilde{\beta}_{d-2}(K_{\leq d-1}) \label{eq:clue-modification-def-spanning-tree}
        \end{align}
    since $\dim \ker\tilde \partial ^{K;\mathbb{R}}_{d-1} = \tilde{\beta}_{d-1}(K_{\leq d-1})$ and $\tilde{\beta}_{d-2}(K_{\leq d-1}) = \tilde{\beta}_{d-2}(K) = 0$. With this, we conclude that $\FaceTree \in \mathbb{T}_{d-1}(K)$. \qedhere
    \end{itemize}
\end{proof}

\begin{remark}
\label[remark]{rem:reason for definition modification}
The final part of the proof of \cref{prop:kernel_relative_boundary} makes explicit the role of the \((d-1)\)-skeleton in the cardinality count for \(\FaceTree\). This complements the arguments of \cite[Proposition~4.1]{DuvalKlivansMartin2009} and \cite[Proposition~2.6]{DuvalKlivansMartin2011}, where the corresponding cardinality condition is used in establishing that \(\FaceTree\) is a \((d-1)\)-dimensional simplicial spanning tree.

To arrive to this conclusion, the skeletal formulation of spanning trees counts discussed in \cref{rem:cellular-spanning-tree} is key. If we consider the full ambient complex \(K\), one does not in general have
\[
    \dim\ker\tilde{\partial}^{K;\mathbb{R}}_{d-1}
    =
    \tilde{\beta}_{d-1}(K),
\]
since \(d\)-dimensional simplices contribute to \(\image\tilde{\partial}^{K;\mathbb{R}}_d\).
\end{remark}

\begin{proposition}
\label[proposition]{prop:determinat_kernel_relative_boundary} 
Under the preceding assumptions, if $\ker\partial_d^{\TopTree,\FaceTree}=0$, then
\[
\left|\mathrm{det}_{\mathbb Z}
\partial_d^{\TopTree,\FaceTree}\right|
=
\frac{
|\tilde\Hgroup_{d-1}(\TopTree;\mathbb Z)|
\,|\tilde\Hgroup_{d-2}(\FaceTree;\mathbb Z)|
}{
|\tilde\Hgroup_{d-2}(\TopTree;\mathbb Z)|
}.
\]
Here $|\det_{\mathbb Z}|$ denotes the absolute determinant in integral bases of the source and target (recall that $\partial_d^{\TopTree,\FaceTree}$ denotes the boundary operator between relative integral chain groups).
\end{proposition}

\begin{proof}
To simplify notations, we write
\[
F=C_d(\TopTree,\FaceTree;\mathbb Z),\qquad
G=C_{d-1}(\TopTree,\FaceTree;\mathbb Z),\qquad
\delta=\partial_d^{\TopTree,\FaceTree}.
\]
Both $F$ and $G$ are free abelian groups of rank $r=|\TopSubset|$.

From \cref{remark:exterior-algebra-free-abelian-groups}, we know that the top exterior powers $\Lambda_{\mathbb Z}^rF$ and $\Lambda_{\mathbb Z}^rG$ are free abelian groups of rank one, so we can choose generators $\omega_F$ and $\omega_G$.
The induced map satisfies
\[
(\Lambda_{\mathbb Z}^r\delta)(\omega_F)
=m\,\omega_G
\]
for some integer $m$, the determinant of $\delta$ relative to these generators.
Changing either generator changes $m$ at most by a sign, hence
\[
|\mathrm{det}_{\mathbb Z}\delta|=|m|
\]
is independent of these choices.

Since $\delta$ is injective, the Smith normal form provides integral bases $a_1,\dots,a_r$ of $F$ and $b_1,\dots,b_r$ of $G$ such that
\[
\delta(a_i)=s_i b_i,\qquad s_i>0.
\]
Consequently,
\[
(\Lambda_{\mathbb Z}^r\delta)
(a_1\wedge\cdots\wedge a_r)
=
\left(\prod_{i=1}^r s_i\right)
b_1\wedge\cdots\wedge b_r,
\]
whereas
\[
\operatorname{coker}\delta
\cong
\bigoplus_{i=1}^r\mathbb Z/s_i\mathbb Z.
\]
Therefore,
\begin{equation}
\label{eq:relative-determinant-cokernel}
\left|\mathrm{det}_{\mathbb Z}\delta\right|
=
\prod_{i=1}^r s_i
=
\left|\operatorname{coker}\delta\right|.
\end{equation}

The relative chain groups vanish outside degrees $d$ and $d-1$, so
\[
\tilde\Hgroup_{d-1}(\TopTree, \FaceTree;\mathbb Z)
\cong\operatorname{coker}\delta.
\]
By \cref{prop:kernel_relative_boundary}, $\TopTree$ is a spanning $d$-tree and $\FaceTree$ is a spanning $(d-1)$-tree. In particular,
\[
\tilde\Hgroup_{d-1}(\FaceTree;\mathbb Z)=0.
\]
Also, $\tilde\Hgroup_{d-2}(\TopTree,\FaceTree;\mathbb Z)=0$ because the relative chain group in that degree vanishes. The long exact sequence of the pair therefore gives
\[
0\longrightarrow
\tilde\Hgroup_{d-1}(\TopTree;\mathbb Z)
\longrightarrow
\tilde\Hgroup_{d-1}(\TopTree,\FaceTree;\mathbb Z)
\longrightarrow
\tilde\Hgroup_{d-2}(\FaceTree;\mathbb Z)
\longrightarrow
\tilde\Hgroup_{d-2}(\TopTree;\mathbb Z)
\longrightarrow0.
\]
All four groups are finite: the first and third by the spanning-tree conditions, the second by \cref{eq:relative-determinant-cokernel}, and the fourth because
\[
\tilde\Hgroup_{d-2}(\TopTree;\mathbb Z)
\cong\tilde\Hgroup_{d-2}(K;\mathbb Z),
\qquad
\tilde\beta_{d-2}(K)=0.
\]

For an exact sequence of finite groups $0\to E\to F\to G\to H\to0$, multiplicativity of orders in the two associated short exact sequences gives $|F|\,|H|=|E|\,|G|$. Applying this identity yields
\[
|\tilde\Hgroup_{d-1}(\TopTree,\FaceTree;\mathbb Z)|
=
\frac{
|\tilde\Hgroup_{d-1}(\TopTree;\mathbb Z)|
\,|\tilde\Hgroup_{d-2}(\FaceTree;\mathbb Z)|
}{
|\tilde\Hgroup_{d-2}(\TopTree;\mathbb Z)|
}.
\]
Combining this with \cref{eq:relative-determinant-cokernel} proves the result.
\end{proof}

We are now in a position to prove our weighted simplicial matrix-tree theorem. 
Similar results can be found in~\cite[Theorem~1.4]{DuvalKlivansMartin2009},~\cite[Theorem~2.12]{DuvalKlivansMartin2011}, and~\cite[Theorem~1.1]{AalipourDuvalKook2018}; here we present a different proof strategy. A key distinction is that these works rely heavily on modified versions of the boundary operators, viewed as matrices after choosing specific basis, to incorporate the weights in the simplicial complex in the analysis. For us, however, the weights appear naturally through the inner product, and our result holds independently of the chosen basis.

\begin{theorem}[Simplicial Weighted Matrix-Tree Theorem]
\label{thm:matrix-tree-thm-sc}
    Let $d\geq 1$ and $K$ be a $d$-dimensional simplicial complex weighted by $w:K \to \Rspace_{>0}$ with $\tilde \beta_{d-1}(K) = \tilde\beta_{d-2}(K) = 0$. Then
\begin{equation}
\label{eq:matrix-tree-thm}
    \lambda\!\left(\chainLup{d-1}(K)\right)
    =
    \frac{
        \tau_d(K,w)\,\tau_{d-1}(K,w)
    }{
        \abs{\tilde\Hgroup_{d-2}(K;\mathbb Z)}^2
        \displaystyle\prod_{\rho\in K_{d-1}}w(\rho)
    },
\end{equation}
where $\lambda(\cdot)$ denotes the pseudo-determinant (\cref{def:pseudodeterminant}).
\end{theorem}

\begin{remark}
The expression in \cref{eq:matrix-tree-thm} is well-defined because $\tilde\beta_{d-2}(K)=0$, so $\tilde H_{d-2}(K;\mathbb Z)$ is finite. Tree numbers are finite by definition of simplicial spanning tree.
\end{remark}

\begin{proof}
Recall that, since $d\geq 1$, we have \(\tilde \partial_d =\partial_d:C_d(K;\mathbb R)\longrightarrow C_{d-1}(K;\mathbb R)\), where both chain spaces carry the weighted inner products introduced earlier in the paper and \(\chainLup{d-1}(K)=\partial_d\partial_d^\adjoint.\) Since $K$ has dimension $d$,
    \[
    r\coloneqq\rank \partial_d
    =
    |K_d|-\tilde\beta_d(K).
    \]
By \cref{lem:pseudodeterminant-exterior-power},
\begin{equation}
\label{eq:matrix-tree-exterior-start}
    \lambda\!\left(\chainLup{d-1}(K)\right)
    =
    \norm{\Lambda^r\partial_d}_{\mathrm{HS}}^2.
\end{equation}

For each simplex $\sigma$, write
\[
    e_\sigma\coloneqq \sqrt{w(\sigma)}\,\sigma.
\]
Recall that, since the standard simplices are orthogonal and satisfy $\langle\sigma,\sigma\rangle=w(\sigma)^{-1}$, the vectors $e_\sigma$ form an orthonormal basis of the corresponding weighted chain space. Hence, for every $r$-element subset $\TopSubset\subseteq K_d$, the vectors
\[
    e_{\TopSubset}
    \coloneqq
    \bigwedge_{\sigma\in\TopSubset} e_\sigma
\]
form an orthonormal basis of $\Lambda^rC_d(K;\mathbb R)$, and similarly the vectors
\[
    e_{\FaceSubset}
    \coloneqq
    \bigwedge_{\rho\in\FaceSubset}e_\rho,
    \qquad
    |\FaceSubset|=r,
\]
form an orthonormal basis of $\Lambda^rC_{d-1}(K;\mathbb R)$.

Therefore,
\begin{equation}
\label{eq:matrix-tree-exterior-sum}
\lambda\!\left(\chainLup{d-1}(K)\right)
=
\sum_{\substack{\TopSubset\subseteq K_d\\|\TopSubset|=r}}
\sum_{\substack{\FaceSubset\subseteq K_{d-1}\\|\FaceSubset|=r}}
\abs{
\left\langle
\Lambda^r \partial_d (e_{\TopSubset}),
e_{\FaceSubset}
\right\rangle
}^2.
\end{equation}

To identify each term, define
\[
    \TopTree\coloneqq
    \TopSubset\cup K_{\leq d-1},
    \qquad
    \overline{\FaceSubset}
    \coloneqq
    K_{d-1}\setminus\FaceSubset,
    \qquad
    \FaceTree\coloneqq
    \overline{\FaceSubset}\cup K_{\leq d-2}.
\]
The quotient $C_d(\TopTree,\FaceTree;\mathbb Z)$ has its standard basis indexed by $\TopSubset$, while $C_{d-1}(\TopTree,\FaceTree;\mathbb Z)$ has its standard basis indexed by $\FaceSubset$. Let
\[
    \mathrm{det}_{\mathbb Z}
    \partial_d^{\TopTree,\FaceTree}
\]
denote the determinant of the relative boundary with respect to these standard integral bases. The coefficient in \cref{eq:matrix-tree-exterior-sum} is then
\begin{equation}
\label{eq:weighted-exterior-coefficient}
\abs{
\left\langle
\Lambda^r \partial_d(e_{\TopSubset}),
e_{\FaceSubset}
\right\rangle
}^2
=
\frac{
    \displaystyle\prod_{\sigma\in\TopSubset}w(\sigma)
}{
    \displaystyle\prod_{\rho\in\FaceSubset}w(\rho)
}
\abs{
    \mathrm{det}_{\mathbb Z}
    \partial_d^{\TopTree,\FaceTree}
}^2.
\end{equation}

By \cref{prop:kernel_relative_boundary}, $\partial_d^{\TopTree,\FaceTree}$ is nonsingular precisely when
\[
    \TopTree\in\mathbb T_d(K)
    \qquad\text{and}\qquad
    \FaceTree\in\mathbb T_{d-1}(K).
\]
Hence only such pairs contribute to \cref{eq:matrix-tree-exterior-sum}. For these pairs, \cref{prop:determinat_kernel_relative_boundary} gives 
\begin{equation} 
\label{eq:relative-boundary-torsion-determinant}
\abs{
\mathrm{det}_{\mathbb Z}\partial_d^{\TopTree,\FaceTree}
}
=
\frac{
    \abs{\tilde\Hgroup_{d-1}(\TopTree;\mathbb Z)}
    \abs{\tilde\Hgroup_{d-2}(\FaceTree;\mathbb Z)}
}{
    \abs{\tilde\Hgroup_{d-2}(\TopTree;\mathbb Z)}
}.
\end{equation}
Because $\TopTree$ contains the full $(d-1)$-skeleton of $K$,
\[
    \tilde\Hgroup_{d-2}(\TopTree;\mathbb Z)
    =
    \tilde\Hgroup_{d-2}(K;\mathbb Z).
\]
Substituting \cref{eq:weighted-exterior-coefficient} and \cref{eq:relative-boundary-torsion-determinant} into \cref{eq:matrix-tree-exterior-sum}, we obtain
\[
\lambda\!\left(\chainLup{d-1}(K)\right)=
\frac{1}{
    \abs{\tilde\Hgroup_{d-2}(K;\mathbb Z)}^2
}
\sum_{\TopTree\in\mathbb T_d(K)}
\abs{\tilde\Hgroup_{d-1}(\TopTree;\mathbb Z)}^2
\prod_{\sigma\in \TopSubset}w(\sigma)
\sum_{\FaceTree\in\mathbb T_{d-1}(K)}
\abs{\tilde\Hgroup_{d-2}(\FaceTree;\mathbb Z)}^2
\frac{1}{
    \displaystyle\prod_{\rho\in\FaceSubset}w(\rho)
}.
\]

Finally,
\[
    \FaceSubset
    =
    K_{d-1}\setminus\overline{\FaceSubset},
\]
so
\[
    \frac{1}{\prod_{\rho\in\FaceSubset}w(\rho)}
    =
    \frac{
        \displaystyle\prod_{\rho\in\overline{\FaceSubset}}w(\rho)
    }{
        \displaystyle\prod_{\rho\in K_{d-1}}w(\rho)
    }.
\]
The two sums therefore become the weighted tree numbers $\tau_d(K,w)$ and $\tau_{d-1}(K,w)$, and \cref{eq:matrix-tree-thm} follows.
\end{proof}

\begin{remark}
The proof separates three ingredients that are sometimes combined at the matrix level. The identity
\[
    \lambda(\partial_d\partial_d^*)=\norm{\Lambda^r\partial_d}_{\mathrm{HS}}^2
\]
is purely due to the inner product spaces structure, and the weight factors arise from comparing the canonical simplicial lattice with the orthonormal simplex vectors induced by the weighted inner products. The torsion factors arise from the determinant of the relative boundary over the integral chain groups. In consequence, there is no need to change the boundary operator itself by the weighting, the weights appear naturally thanks to the inner product that we have defined.
\end{remark}

\subsection{Effective Resistance and Tree Numbers}
\label{sec:er-and-asimplicial-spanning-trees}
We now present an \emph{enumeration theorem} for simplicial complexes, relating the relative effective resistance of a simplex and the simplicial tree numbers defined in \cref{sec:tree-numbers}. Subsequently, we provide a probabilistic interpretation of this theorem that generalizes the graph-theoretic case presented in \cref{thm:rel_eff_res_spanning_trees}, where the relative effective resistance of a simplex indicates its probability of appearance in a random spanning tree under an appropriate probability measure. 

Our proof is inspired by the matrix-based argument of \cite[Theorem 13.5.1]{spielman2025spectral} in the graph setting, but again reformulated, in terms of exterior algebra and independent of bases. Related connections between effective resistance and simplicial spanning-tree or spanning-forest enumeration appear in \cite[Theorem 5.2]{KookLee2018} and \cite[Theorem 6]{kookHighdimensionalNetworksSpanning2019b}; see also \cite{DuvalKookLee2023}.
Here we express the weighted spanning-tree inclusion probability in terms of relative effective resistance. Our argument combines the weighted matrix-tree theorem (\cref{thm:matrix-tree-thm-sc}) with the orthogonal-projection structure of the effective-resistance operator and an exterior-algebra form of the Cauchy--Binet identity.

\begin{theorem}
\label{thm:rel-eff-resistant-spanning-forest}
    Let \(K\) be a \(d\)-dimensional simplicial complex, $d\geq 1$, with 
    $\tilde \beta_{d-1}(K) = \tilde \beta_{d-2}(K) = 0$. Consider a $d$-dimensional simplex $\sigma \in K$ and call 
    \[\tau_d(K,w)_\sigma\coloneqq \sum_{\stackrel{T \in \mathbb{T}_d(K),}{\sigma \in T_d}} \left|\tilde H_{d-1}(T;\Zbb) \right|^2 \prod_{\rho \in T_d} w(\rho) = \sum_{\stackrel{T \in \mathbb{T}_d(K),}{\sigma \in T_d}} \tau_d(T, w)\]
     to the weighted tree number restricted to the trees $T \in \mathbb{T}_d(K)$ that contain $\sigma$. Then, we have
    \begin{equation}
    \label{eq:er_quotient_tree_numbers}
         \dfrac{\tau_d(K,w)_\sigma}{\tau_d(K, w)} = w(\sigma) \, r_\sigma.
    \end{equation}
\end{theorem}

\begin{remark}
    The quotient in \cref{eq:er_quotient_tree_numbers} is well defined because, if $\tilde\beta_{d-1}(K)=0$, then maximal spanning forests are spanning trees. Moreover, such forests exist by \cref{prop:characterization-maximal-forests}. Hence $\tau_d(K,w)>0$.
\end{remark}

\begin{proof}
Throughout the proof, to ease the notations, we use
\[
\mathcal L_T = \mathcal L_{d-1}^{\mathrm{up}}(T) = \partial_d^T\circ (\partial_d^T)^\adjoint,\quad \mathcal L_K = \mathcal L_{d-1}^{\mathrm{up}}(K) =\partial_d^K \circ (\partial_d^K)^\adjoint.
\]
and 
\[
r=\rank\partial_d^K=\rank\mathcal L_K,
\qquad
P=\mathcal T_d
=(\partial_d^K)^\adjoint\mathcal L_K^+\partial_d^K.
\]
By \cref{prop:properties of R operator}, $P$ is an
orthogonal projection of rank $r$, with
$\kernel P=\kernel\partial_d^K$.

For a $d$-dimensional tree $T \in \mathbb{T}_d(K)$, since $T$ and $K$ have the same $(d-1)$-skeleton, the lower-dimensional factors in \cref{thm:matrix-tree-thm-sc} cancel, giving
\begin{equation}
\label{eq:tree-inclusion-pseudodeterminants}
\frac{\tau_d(K,w)_\sigma}{\tau_d(K,w)}
=
\sum_{\substack{T\in\mathbb T_d(K)\\ \sigma\in T_d}}
\frac{\lambda(\mathcal L_T)}{\lambda(\mathcal L_K)}.
\end{equation}
Moreover, for every such tree, by \cref{prop:range_laplacian_spanning_forest} we have
\[
|T_d|=r,\qquad
\image \mathcal L_T=\image \mathcal L_K=\kernel(\mathcal L_K^+)^\perp.
\]
Thus, by the multiplicativity and composition-reversal identities in \cref{lem:lambda_multiplicative}, in addition to the eigenvalues of the pseudoinverse being the inverses of the ones from the original matrix, we can rewrite
\begin{align*}
\frac{\lambda(\mathcal L_T)}{\lambda(\mathcal L_K)}
&=\lambda(\mathcal L_T)\lambda(\mathcal L^+_K)\\
&=\lambda(\mathcal L_T \mathcal L_K^+)\\
&=\lambda((\partial_d^T)^\adjoint \mathcal L_K^+ \partial_d^T)\\
&=\lambda(\iota_T^\adjoint P\iota_T),
\end{align*}
where $\iota_T:C_d(T)\hookrightarrow C_d(K)$ is the inclusion, so that $\partial_d^T=\partial_d^K\iota_T$.

The operator $\iota_T^\adjoint P\iota_T$ is positive definite on the $r$-dimensional space $C_d(T)$. Indeed, $\partial_d^T$ is injective, and $\mathcal L^+_K$ is positive definite on $\image \partial_d^K=\image \mathcal L_K$, so for every nonzero $v\in C_d(T)$,
\[
\langle v,\iota_T^\adjoint P\iota_Tv\rangle_{C_d}
=
\langle \partial_d^Tv,\mathcal L_K^+ \partial_d^T v\rangle_{C_{d-1}}>0.
\]
Its pseudo-determinant therefore equals its determinant. Consequently, \cref{eq:tree-inclusion-pseudodeterminants}
becomes
\begin{equation}
\label{eq:tree-inclusion-compressions}
\frac{\tau_d(K,w)_\sigma}{\tau_d(K,w)}
=
\sum_{\substack{T\in\mathbb T_d(K)\\ \sigma\in T_d}}
\det(\iota_T^\adjoint P\iota_T).
\end{equation}

We next extend this sum to all $r$-element subsets of $K_d$. Given such a subset $T_d$, let $T$ be the subcomplex with these top-dimensional simplices and the full $(d-1)$-skeleton of $K$.
If $\partial_d^T$ were injective, its $r$ columns would form a basis of $\image \partial_d^K$. Hence
\[
\widetilde H_d(T;\mathbb R)=0,
\qquad
\widetilde H_{d-1}(T;\mathbb R)
=
\widetilde H_{d-1}(K;\mathbb R)=0,
\]
so $T$ would be a spanning $d$-tree. Here, when $d=1$, homology is computed using the augmented chain complex. Thus, if $T$ is not a spanning tree, there is a nonzero $v\in\kernel \partial_d^T$, and
\[
(\iota_T^\adjoint P\iota_T)v
=(\partial_d^T)^\adjoint \mathcal L_K^+\partial_d^Tv=0.
\]
The corresponding determinant vanishes. Adding these zero terms to \cref{eq:tree-inclusion-compressions} yields
\begin{equation}
\label{eq:tree-inclusion-all-compressions}
\frac{\tau_d(K,w)_\sigma}{\tau_d(K,w)}
=
\sum_{\substack{T_d\subseteq K_d\\
                 |T_d|=r,\ \sigma\in T_d}}
\det(\iota_T^\adjoint P\iota_T).
\end{equation}

Denote the orthonormal simplex vectors as
\[
e_\rho=\sqrt{w(\rho)}\,\rho,\qquad \rho\in K_d.
\]
Since $P^\adjoint P=P$, for \(\tau, \rho \in T\) we have that
\[
\langle \iota_T^\adjoint P\iota_T\, e_\tau, e_\rho\rangle_{C_d} =
\langle P\iota_T\, e_\tau, P\iota_T\, e_\rho\rangle_{C_d}
= 
\langle P e_\tau, Pe_\rho\rangle_{C_d}
\]
and
\begin{equation}
\label{eq:compression-exterior-norm}
\left\|\bigwedge_{\rho\in T_d}Pe_\rho\right\|^2 = \det\bigl(\langle P e_\tau, Pe_\rho\rangle_{C_d}\bigr)_{\tau, \rho\,\in\, T_d} = \det(\iota_T^\adjoint P\iota_T).
\end{equation}
The squared norm is independent of the ordering used in the wedge.

Set
\(\gamma=Pe_\sigma.\)
Because $\kernel P=\kernel \partial_d^K$ and $\partial_d^K \sigma\ne0$, we have
$\gamma\ne0$. Let
\[
W=\image P\cap\gamma^\perp,
\qquad
Q=P_{\gamma^\perp}P.
\]
Since $\gamma\in\image P$, the operator $Q$ is the orthogonal projection onto $W$. In particular,
\[
\dim W=r-1,\qquad Qe_\sigma=0.
\]

Let $T_d=\{\sigma\}\cup S$. Then, applying \cref{lem:exterior-norm-projection} with $\gamma=Pe_\sigma$ and using $P_{\gamma^\perp}Pe_\rho=Qe_\rho$, we obtain
\[
\left\|
\gamma\wedge\bigwedge_{\rho\in S}Pe_\rho
\right\|^2
=
\|\gamma\|^2
\left\|\bigwedge_{\rho\in S}Qe_\rho\right\|^2.
\]
Together with \cref{eq:tree-inclusion-all-compressions} and \cref{eq:compression-exterior-norm}, this gives
\begin{equation}
\label{eq:tree-inclusion-exterior-sum}
\frac{\tau_d(K,w)_\sigma}{\tau_d(K,w)}
=
\|\gamma\|^2
\sum_{\substack{S\subseteq K_d\setminus\{\sigma\}\\
                 |S|=r-1}}
\left\|\bigwedge_{\rho\in S}Qe_\rho\right\|^2.
\end{equation}

Since $Qe_\sigma=0$, any term indexed by a set containing $\sigma$ vanishes. We may therefore sum over all $(r-1)$-element subsets of $K_d$. The simplex wedges form an orthonormal basis of $\Lambda^{r-1}C_d(K)$, and hence
\[
\sum_{\substack{S\subseteq K_d\\ |S|=r-1}}
\left\|\bigwedge_{\rho\in S}Qe_\rho\right\|^2
=
\|\Lambda^{r-1}Q\|_{\mathrm{HS}}^2.
\]
For $r>1$, applying \cref{lem:pseudodeterminant-exterior-power} to $Q$ gives
\[
\|\Lambda^{r-1}Q\|_{\mathrm{HS}}^2
=
\lambda(QQ^\adjoint)
=
\lambda(Q)
=
1,
\]
because $Q$ is an orthogonal projection of rank $r-1$.
For $r=1$, the same equality follows from
$\Lambda^0Q=\operatorname{id}_{\mathbb R}$.

The sum in \cref{eq:tree-inclusion-exterior-sum} is therefore equal to $1$. Recalling that $\gamma=Pe_\sigma=\sqrt{w(\sigma)}\,\mathcal T_d(\sigma)$, and using \cref{eq:ER in norm form}, we conclude
\[
\frac{\tau_d(K,w)_\sigma}{\tau_d(K,w)}
=
\|\gamma\|^2
=
w(\sigma)\,\|\mathcal T_d(\sigma)\|^2
=
w(\sigma)\,r_\sigma,
\]
as claimed.
\end{proof}

Recall that, for graphs, we had that a \textit{random spanning tree} was a random element sampled from $\mathbb{T}$ with probability proportional to the products of the edge weights (see \cref{eq:random-tree-graph}). We can extend this notion to simplicial complexes by considering a \textit{random simplicial spanning tree} $\mathbf{T}$ as a random element of $\mathbb{T}_d(K)$ sampled from a prescribed probability law proportional to the weights of the $d$-simplices, as we see next.

\begin{corollary}[Reinterpretation in terms of probability of \cref{thm:rel-eff-resistant-spanning-forest}]
\label[corollary]{coro:probability-finding-simplex-spanning-tree}
Let \(K\) be a \(d\)-dimensional simplicial complex, $d\geq 1$, with $\tilde \beta_{d-1}(K) = \tilde \beta_{d-2}(K) = 0$. If we sample a tree $T\in \mathbb{T}_d(K)$ with probability 
\begin{equation}
\label{eq:probability-spanning-tree}
    \mathrm{Pr}[\mathbf{T}=T]\coloneqq\dfrac{\left|\tilde H_{d-1}(T;\Zbb) \right|^2 \prod_{\rho \in T_d} w(\rho)}{\sum_{T' \in \mathbb{T}_d(K)} \left|\tilde H_{d-1}(T';\Zbb) \right|^2 \prod_{\rho \in T_d'} w(\rho)}
\end{equation}
then, the probability of finding the simplex $\sigma \in K_d$ in such random simplicial spanning tree is 
\[\mathrm{Pr}[\sigma \in \mathbf{T}] = w(\sigma) \, r_\sigma.\]
\end{corollary}

\begin{proof}
    Using the law of total probability, the probability that a $d$-simplex $\sigma\in K$ belongs to a random simplicial spanning $d$-tree $\mathbf{T}$ sampled from the law of the statement is given by
    \begin{align*}
    \Pr[\sigma\in \mathbf{T}]&=\sum_{\substack{T \in \mathbb{T}_d(K),\\ \sigma \in T}}\Pr[\sigma \in T\, |\, \mathbf{T} = T]\,\Pr[{\mathbf T}=T]\\
    & =\sum_{\substack{T \in \mathbb{T}_d(K),\\ \sigma \in T}}\frac{\left|\tilde H_{d-1}(T;\Zbb) \right|^2 \prod_{\rho \in T_d} w(\rho)}{\sum_{T' \in \mathbb{T}_d(K)} \left|\tilde H_{d-1}(T';\Zbb) \right|^2 \prod_{\rho \in T_d'} w(\rho)}\\
     &=\sum_{\substack{T \in \mathbb{T}_d(K),\\ \sigma \in T}}\dfrac{\tau_d(T,w)}{\tau_d(K, w)}=\dfrac{\tau_d(K,w)_\sigma}{\tau_d(K, w)}.
     \end{align*}

     By \cref{thm:rel-eff-resistant-spanning-forest}, we conclude that $\Pr[\sigma \in \mathbf{T}] = w(\sigma) \, r_\sigma$.
\end{proof}

\begin{remark}
    \cref{coro:probability-finding-simplex-spanning-tree} is a generalization of \cref{thm:rel_eff_res_spanning_trees} in higher dimensions in the following sense. In the graph case, i.e., when $d=1$ in the previous statement, the assumptions on the Betti numbers simply reduce to imposing that the graph is connected. In that case, $|\tilde H_0(T; \Zbb)| =1$ for all (connected) trees in the graph $T$ and thus, the probability in \cref{eq:probability-spanning-tree} coincides with that of \cref{thm:rel_eff_res_spanning_trees}. Note also that the tree number $\tau_1(K)= \sum_{T \in \mathbb{T}(G)} |\tilde H_0(T; \Zbb)|$ counts the number of trees present in the graph $G$. 
\end{remark}

\section{Connections to Previous Effective Resistance Matrices} 
\label{sec:ER-simplicial-matrix-form}

We show that the notions of effective resistance of simplices from \cite{BlackMaxwell2021, KookLee2018, OstingPalandeWang2020} can be recovered from our effective resistance bilinear form by selecting appropriate bases.

\subsection{Connection to the Definition of Osting, Palande, and Wang} 
\label{sec:connection_Osting}

We review the notion of effective resistance of simplices introduced by Osting et al.~\cite{OstingPalandeWang2020} and recover their notion using our framework in \cref{prop:matrix of caR_p}. 

According to~\cref{eq:L_up}, \([\cochainLup{p}]_{\caB^p}^{\caB^p}\) is not symmetric despite originating from a self-adjoint operator, since $\Bcal^p$ is not an orthonormal basis. 
Osting et al.~\cite{OstingPalandeWang2020} worked in this basis but avoided using $    [\cochainLup{p}]_{\caB^p}^{\caB^p}$ due to this lack of symmetry. Instead, they considered the following 
normalized Laplacian matrix
\begin{align}
\label{eqn:relative-er}
    \Lupmatrix{p-1}{sym} \coloneqq W_{p-1}\Lupmatrix{p-1}{} = B_p  W_{p} B_p^\top;  
\end{align}
from which they defined the \textit{effective resistance matrix} $R_{p} \in \Rspace^{n_p \times n_p}$, 
\begin{align}\label{eq:R_p_matrix_original}
    R_{p} \coloneqq B_p^\top (\Lupmatrix{p-1}{sym})^{+} B_p  = B_p^\top \left(
    B_p  W_{p} B_p^\top \right)^{+
    } B_p .  
\end{align}
The effective resistance of a $p$-simplex $\sigma$ is defined as the diagonal entry of $R_{p}$ corresponding to $\sigma$, which reduces to the edge-based effective resistance in graphs (\cref{eq:graph_resistance}) for $p=1$.

\begin{remark}
    The matrix defined in \cref{eqn:relative-er} is in fact the matrix representation of the up Laplacian operator with respect to the standard basis $\Bcal_{p-1}$ of $C_{p-1}(K)$ and the basis $\hat\Bcal_{p-1}\coloneqq ({w_{\sigma_1}} \sigma_1 , \dots , {w_{\sigma_{n_{p-1}}}} \sigma_{n_{p-1}})$ of $C_{p-1}(K)$. Indeed, by $[\partial_{p}]_{\Bcal_p}^{\Bcal_{p-1}}=B_p$ and \cref{eq:adjoint of boundary matrix representation},
    \begin{align*}
        [\partial_{p} \partial_{p}^\adjoint]^{\Bcal_{p-1}}_{\hat\Bcal_{p-1}} & = [\partial_{p} \partial_{p}^\adjoint]^{\Bcal_{p-1}}_{\Bcal_{p-1}} [\mathrm{Id}]_{\hat\Bcal_{p-1}}^{\Bcal_{p-1}}\\
        & = [\partial_{p}]_{\Bcal_{p}}^{\Bcal_{p-1}} [\partial_p^\adjoint]_{\Bcal_{p-1}}^{\Bcal_p} [\mathrm{Id}]_{\hat\Bcal_{p-1}}^{\Bcal_{p-1}}\\
        & = B_p W_p B_p^\top W_{p-1}^{-1} W_{p-1}\\
        & = \,B_p  W_{p}B_p^\top=\Lupmatrix{p-1}{sym}.
    \end{align*}
\end{remark}

We relate the matrix representation of the effective resistance bilinear form \( \Rcal_p \) with the matrix $R_p$ in \cref{eq:R_p_matrix_original}. 

\begin{proposition} \label[proposition]{prop:matrix of caR_p}
Let \(W_{p-1}\) and \(W_p\) be the weight matrices on the \((p-1)\)- and \(p\)-simplices, respectively.
With respect to the standard bases, the effective-resistance bilinear form \(\Rcal_p\) has a matrix representation
\begin{equation}
\label{eq:caRp-direct}
[\Rcal_p]_{\caB_p,\caB_p}
= B_p^\top W_{p-1}^{-1/2} \left( W_{p-1}^{-1/2} B_p W_p B_p^\top W_{p-1}^{-1/2}\right)^+ W_{p-1}^{-1/2} B_p,
\end{equation}
which reduces to \(B_p^\top (B_p W_p B_p^\top)^{+} B_p= R_p\) when the \((p-1)\)-simplices have unit weights (\(W_{p-1}=\mathrm{Id}\)).
\end{proposition}

\begin{proof}
Using the standard bases \( \caB_{p-1} \) and \( \caB_p \) (see \cref{def:standard basis}), putting together the matrix expressions for the adjoint of the boundary (\cref{eq:adjoint of boundary matrix representation}), the up Laplacian  (\cref{eq:L_up_chain}) and the pseudoinverse (\cref{eq:pseudoinverse-matrix-expresion}),
the effective-resistance operator
\(\ERoperator_p = \partial_p^\adjoint (\chainLup{p-1})^+ \partial_p\)
has matrix representation
    \begin{align*}
        [\ERoperator_p]_{\caB_{p}}^{\caB_p} & = \big[\partial_p^\adjoint \big]_{\caB_{p-1}}^{\caB_p}\, 
        \big[(\chainLup{p-1})^+\big]_{\caB_{p-1}}^{\caB_{p-1}}\,[\partial_p ]_{\caB_p}^{\caB_{p-1}} \\
        &= W_p B_p^\top W_{p-1}^{-1} \, W_{p-1}^{1/2}\left( W_{p-1}^{-1/2} \left(B_p W_p B_p^\top W_{p-1}^{-1}\right) W_{p-1}^{1/2}\right)^+ W_{p-1}^{-1/2} B_p
    \end{align*}
Since
\([\Rcal_p]_{\caB_p,\caB_p}
=  \big([\ERoperator_p]_{\caB_p}^{\caB_p} \big)^\top \, [\langle\cdot,\cdot\rangle_{C_p}]_{\caB_p,\caB_p}
= \big([\ERoperator_p]_{\caB_p}^{\caB_p} \big)^\top \,  W_p^{-1},\)
the desired result follows after simplification.
\end{proof}

\begin{remark}
With respect to the orthonormal bases \(\tilde\caB_{p-1}\) and \(\tilde\caB_p\),
the Gram matrix $[\langle\cdot,\cdot\rangle_{C_p}]_{\tilde \caB_p, \tilde \caB_p}$ is the identity, so the matrix representation of the effective resistance bilinear form $\Rcal_p$ coincides with that of the effective resistance operator $\ERoperator_p$. 
In particular, we have
\[
  [\Rcal_p]_{\tilde \caB_p, \tilde \caB_p}
  = \big( [\ERoperator_p]_{\tilde \caB_{p}}^{\tilde\caB_p}  \big)^\top 
  = [\ERoperator_p]_{\tilde \caB_{p}}^{\tilde\caB_p}
  = W_p^{-1/2}  [\ERoperator_p]_{\caB_{p}}^{\caB_p}  W_{p}^{1/2}
    = W_p^{1/2} [\Rcal_p]_{\caB_p, \caB_p} W_{p}^{1/2},
\]
which is equal to $W_p^{1/2} B_p^\top \left(  B_p W_p B_p^\top\right)^+ B_p W_{p}^{1/2}$ when the \((p-1)\)-simplices have unit weights.
\end{remark}

\subsection{Connection to the Definition of Kook and Lee}
\label{sec:er_by_kook_lee}

Kook and Lee~\cite{KookLee2018} defined a notion of effective resistance for simplicial networks \((K, W_d)\), that is, \(d\)-dimensional simplicial complexes \(K\), for \(d\ge1\), with weights only on the \(d\)-simplices, gathered in the diagonal matrix \(W_d\).
In \cref{prop:equivalence-kook-and-lee}, we show that their definition is recovered from our effective resistance bilinear form; and in \cref{prop:R_d=R'_d}, that their matrix expression is a special case of that of \cite{OstingPalandeWang2020}, which represents our bilinear form under suitable bases (\cref{sec:connection_Osting}).

Kook and Lee based their definition of effective resistance in a current generator $\sigma$. If our $d$-dimensional simplicial complex $K$ has $n$ vertices, a current generator is a subset $\sigma \subset\{1, \dots, n\}$ with $|\sigma| = d+1$ such that \(\partial_d \sigma = - \partial_d^K (c)\), for some $c \in C_d(K)$. 
This is used to define the following complex:
\begin{equation}
\label{eq:Z-complex}
    Z:\quad H_d(K) \xrightarrow{\partial_{d+1}^Z} 
    C_d(K \sqcup \{\sigma\}) \cong C_d(K) \oplus \mathrm{span}(\sigma)
    \xrightarrow{\partial_d^Z} C_{d-1}(K) \to \cdots,
\end{equation}
where 
\begin{itemize}
    \item $\image \partial_{d+1}^Z = \ker \partial_d^K \subset C_d(K) \subset  C_d(K \sqcup \{\sigma\})$,
    \item $\partial_d^Z|_{C_d(K)} = \partial_d^K$, 
    \item and $\partial_d^Z(\sigma) = - \partial_d^K(c)$, $c$ being the chain in the definition of the current generator. 
\end{itemize}

We endow \(H_d(K) = \ker \partial^K_d\) with the inner product induced from \(C_d(K)\). 
For \(C_d(K \sqcup \{\sigma\})\), we extend the inner product on \(C_d(K)\) by declaring the new basis element \(\sigma\) to be orthogonal to the standard basis of $C_d(K)$ and have weight one. 
In other words, for \(\alpha, \beta \in C_d(K \sqcup \{\sigma\})\) written as \(\alpha = \alpha_K + \alpha_\sigma \sigma\) and \(\beta = \beta_K + \beta_\sigma \sigma\) with \(\alpha_K, \beta_K \in C_d(K)\) and \(\alpha_\sigma, \beta_\sigma \in \Rspace\), we set 
\[\langle \alpha, \beta \rangle_{C_d(K \sqcup \{\sigma\})}
= \langle \alpha_K, \beta_K \rangle_{C_d(K)} + \alpha_\sigma\, \beta_\sigma\, ,
\qquad
\langle \sigma, \sigma \rangle_{C_d(K \sqcup \{\sigma\})} = 1.\]
The \(d\)th Hodge Laplacian of the complex in \cref{eq:Z-complex} is then
\[\Lcal^Z_d \coloneqq \Lcal^{\mathrm{Hodge}}_d(Z)
= (\partial_d^Z)^\adjoint \partial_d^Z + \partial_{d+1}^Z (\partial_{d+1}^Z)^\adjoint,\]
where adjoints are taken with respect to these inner products.

Notice that now, the homology of \(Z\) at level $d$ is 1-dimensional:
\[H_d(Z) = \dfrac{\ker \partial_d^Z}{\image \partial_{d+1}^Z} = \dfrac{\ker \partial_d^K \oplus \mathrm{span}\, (\sigma + c)}{\ker \partial_d^K} \cong \mathrm{span}\, (\sigma +c). \]
Hence, the kernel of the $d$th Hodge Laplacian of $Z$ is 1-dimensional.

The definition of the effective resistance by Kook and Lee \cite{KookLee2018} requires the  following three steps.

\begin{enumerate}[leftmargin=*]
    \item \textbf{Current construction:} 
    First, fixing a value $i_\sigma \in \Rspace\smallsetminus \{0\}$, Kook and Lee~\cite{KookLee2018} defined the chain $I_\sigma\in C_d(K \sqcup \{\sigma\})$ as the unique vector in the kernel of the $d$th Hodge Laplacian of $Z$ whose projection onto $\sigma$ is precisely $i_\sigma$, that is, 
    \begin{equation}
        \label{eq:current-kook-lee}
        I_\sigma \in \ker(\Lcal^Z_d) \cong \ker\left(\partial_d^Z\right) \cap \ker\left((\partial_{d+1}^Z)^\adjoint\right), \text{ s.t. }I_\sigma \coloneqq I + i_\sigma \sigma,
    \end{equation}
    with $I \in C_d(K)$, also uniquely defined since $\dim \ker(\Lcal^Z_d)=1$. 
    
    Both $I_\sigma$ and $I$ satisfy Kirchhoff's current law for their respective complexes. For $I_\sigma$, this follows from it being a cycle, $\partial_d^Z I_\sigma = 0$. Similarly, for $I$ we have that
    $0 = \partial^Z_d I_\sigma = \partial_d^K I + i_\sigma \partial_d^Z(\sigma)$, and hence $\partial_d^K I = - i_\sigma \partial_d^Z(\sigma) = i_\sigma \partial_d^K(c)$ with $c\in C_d(K)$, so that $\beta = i_\sigma \partial_d^K(c)\in C_{d-1}(K) $ is the boundary current in our terminology, where $i_\sigma$ amp\`{e}res are being injected/extracted in each simplex of the boundary. This explains the intuition behind the value $i_\sigma$ set at the beginning.

    \item \textbf{Voltage construction:} 
    The authors then defined a cochain $V_\sigma \in C^d(K \sqcup \{\sigma\})$ by applying Ohm’s law to the current \(I\) constructed on \(K\), which in  our language means setting $V=\flat_d^K I$, and defining
    \begin{equation}
        \label{eq:voltage-kook-lee}
        V_\sigma \coloneqq V + v_\sigma \sigma^\dual, \text{ s.t. } V_\sigma (I_\sigma) = 0,
    \end{equation}
    or, in other words, setting $v_\sigma$ to be the value that satisfies
    \begin{equation}
    \label{eq:equation_vsigma_isigma}
     V_\sigma (I_\sigma) = \langle I, I\rangle_{C_d(K)} + i_\sigma v_\sigma = 0.
    \end{equation}
    Note that, from \cref{eq:equation_vsigma_isigma} we deduce that $i_\sigma v_\sigma < 0$, i.e., they must have different signs.

    In this setting, \(V\) satisfies Ohm’s law in \(K\) by definition with respect to \(I\), and \(I\) has boundary current \(\beta = i_\sigma \partial_d^K(c)\). By \cref{coro:voltage}, Kirchhoff’s voltage law holds for \(V\) in \(K\); that is, there exists \(\phi \in C^{d-1}(K)\) such that \(\delta_{d-1}^K \phi = V\). Hence, \(V\) and \(I\) are indeed a voltage and a current according to our definition: all three electric laws hold on them.

    The case of \(V_\sigma\) is more subtle. Proving that $V_\sigma$ and $I_\sigma$ are related through Ohm's law would entail
    \[\flat_d^{Z} I_\sigma 
    = \langle I + i_\sigma \sigma, \cdot \rangle_{C_d(K \sqcup \{\sigma\})} 
    = \flat_d^K I + i_\sigma \langle \sigma, \cdot \rangle_{C_d(K \sqcup \{\sigma\})},\]
    being equal to
    \[V_\sigma = \flat_d^K I + v_\sigma \sigma^\dual,
    \quad \text{where} \quad \sigma^\dual = \frac{\langle \sigma, \cdot \rangle_{C_d(K \sqcup \{\sigma\})}}{\langle \sigma, \sigma \rangle_{C_d(K \sqcup \{\sigma\})}}.\]
    However, with our definition of the inner product on \(C_d(K \sqcup \{\sigma\})\), this equality fails in general since \(i_\sigma \neq v_\sigma\). Modifying the inner product to mirror the weighting in \(K\), i.e. setting \(\langle \sigma, \sigma \rangle_{C_d(K \sqcup \{\sigma\})} = w(\sigma)^{-1}\) with $\sigma$ having weight $w(\sigma)$ in the complex $K\sqcup \{\sigma\}$, would impose 
    \[w(\sigma) = \frac{i_\sigma}{v_\sigma}.\]
    This is impossible because \(i_\sigma v_\sigma < 0\), implying \(w(\sigma) < 0\), whereas the inner product requires \(\langle \sigma, \sigma \rangle_{C_d(K \sqcup \{\sigma\})} = w(\sigma)^{-1} \geq 0\). Thus, Ohm’s law cannot hold for \(V_\sigma\). Nevertheless, we will show in \cref{lemma:potential_V_sigma} that \(V_\sigma\) satisfies Kirchhoff’s voltage law independently of Ohm's law, so two of the three electric laws remain valid for the pair $I_\sigma$ and $V_\sigma$.

\item \textbf{Effective resistance definition:} With these definitions, Kook and Lee defined the effective resistance of the current generator $\sigma$ \cite[Definition 3.1]{KookLee2018} precisely as
    \begin{equation}
        \label{eq:er-kook-lee}
        r'_\sigma \coloneqq -\dfrac{v_\sigma}{i_\sigma}.
    \end{equation}
\end{enumerate}

We prove the following two lemmas establishing some properties of the voltage obtained for $V$ from Kirchhoff's voltage law. Similar properties are already stated with brief proofs in \cite[Section 4.1]{KookLee2018}, here we provide the statements and proofs adapted to our setting.

\begin{lemma}
\label[lemma]{lemma:potential_V_sigma}
    A potential $\phi \in C^{d-1}(K)$ for $V$, i.e.~$V= \delta_{d-1}^K \phi$, is also a potential for $V_\sigma$, that is 
    \[\delta_{d-1}^Z \phi = V_\sigma,\]
    with $\delta_{d-1}^Z$ the dual map of $\partial_{d}^Z$ in \cref{eq:Z-complex}.
\end{lemma}

\begin{proof}
    We check that the restriction of $\delta_{d-1}^Z \phi$ to $K$ agrees with $V$ and that $\delta_{d-1}^Z\phi (I_\sigma) = 0$. Indeed, for $\alpha \in C_d(K)$, we have $$\delta_{d-1}^Z \phi (\alpha) = \phi(\partial_d^Z \alpha) = \phi(\partial_d^K \alpha) = \delta_{d-1}^K \phi (\alpha) = V(\alpha).$$
    In addition, by definition (\cref{eq:current-kook-lee}), we have that $I_\sigma \in \ker \partial_d^Z$, which implies
    $$\delta_{d-1}^Z \phi(I_\sigma) = \phi (\partial_d^Z I_\sigma) = 0,$$
    proving that $\delta_{d-1}^Z\phi = V_\sigma$ as it satisfies its two defining conditions.
\end{proof}

\begin{lemma}
\label[lemma]{lemma:v_sigma}
    Given a potential $\phi \in C^{d-1}(K)$ for $V$, i.e.~$V= \delta_{d-1}^K \phi$, we have
    \begin{equation}
    \label{eq:v_sigma}
        v_\sigma = \phi(\partial_d\sigma)
    \end{equation}
    where $v_\sigma$ is the component accompanying $\sigma^\dual$ in \cref{eq:voltage-kook-lee}.
\end{lemma}

\begin{proof}
    From \cref{lemma:potential_V_sigma}, we know that $\delta_{d-1}^Z \phi = V_\sigma$. 
    Let $\alpha \in C_d(K\sqcup \{\sigma\})$, and write $\alpha = \alpha_K + \alpha_\sigma \sigma$. Then, we have
    $$\delta_{d-1}^Z \phi (\alpha) = \phi \left( \partial_{d}^Z (\alpha)\right) = \phi \left( \partial_d^K(\alpha_K)\right) + \alpha_\sigma\, \phi \left( \partial_d(\sigma)\right) = V (\alpha_K) + \alpha_\sigma\, \phi \left( \partial_d(\sigma)\right), $$
    but also,
    $$V_\sigma (\alpha) = V(\alpha_K) + v_\sigma \, \alpha_\sigma \,\sigma^\dual(\sigma) =V(\alpha_K) + v_\sigma \, \alpha_\sigma.$$
    Since $\delta_{d-1}^Z \phi = V_\sigma$, and given that $\alpha_\sigma$ is arbitrary, we conclude:
    \[\phi(\partial_d\sigma) = v_\sigma\]
    as desired.
\end{proof}

Finally, we are ready to prove that the definition of effective resistance by Kook and Lee coincides with \cref{def:ER_chains_cochains} when $\sigma \in C_d(K)$.

\begin{proposition}
\label[proposition]{prop:equivalence-kook-and-lee}
    Let $\sigma \in C_d(K)$ be a current generator, and define the current $I \in C_d(K)$ and $i_\sigma \in \Rspace$ (resp.~the voltage $V\in C^d(K)$ and $v_\sigma \in \Rspace$) as in \cref{eq:current-kook-lee} (resp.~\cref{eq:voltage-kook-lee}). Let $\phi \in C^{d-1}(K)$ be a potential for $V$, i.e.~$\delta_{d-1}^K \phi = V$. Then, the following equality holds:
    \[\mathcal{R}_d(\sigma, \sigma) = -\dfrac{v_\sigma}{i_\sigma}\]
    or equivalently, $r_\sigma = r'_\sigma$.
\end{proposition}

\begin{proof}
    Denoting the boundary current $\beta = \partial_d^K I$, from \cref{prop:potential_and_voltage_law} we know that the potential $\phi$ must satisfy $\cochainLup{d-1} \phi = \flat^K_{d-1}\beta$.
    Using the pseudo inverse, we obtain $\phi = (\cochainLup{d-1} )^+ \flat_{d-1}^K \beta$. Applying $(\flat_{d-1}^K)^{-1}$ to both sides, and \cref{lem:L pseudoinverse chains cochains}, $$(\flat_{d-1}^K)^{-1} \phi = (\Lcal^{\text{up}}_{d-1})^+ \, \beta \in C_{d-1}(K). $$
    Considering the inner product with $\partial_d \sigma$ in $C_{d-1}(K)$, we obtain the following sequence of equivalences:
    \begin{align}
        \langle (\Lcal^{\text{up}}_{d-1})^+ \, \beta, \partial_d \sigma\rangle_{C_{d-1}(K)} & =  \langle(\flat_{d-1}^K)^{-1} \phi , \partial_d\sigma \rangle_{C_{d-1}(K)} & \nonumber\\
        & = \phi (\partial_d \sigma) & (\text{definition of }\flat_{d-1}^{-1})\nonumber\\
        & = v_\sigma. & (\text{\cref{lemma:v_sigma}}) \label{eq:inner_product_v_sigma} 
    \end{align}
     
    We know that $\beta = \partial_d^K I = -i_\sigma \partial_d(\sigma)$ since $0 = \partial_d^Z I_\sigma = \partial_d^K I + i_\sigma \partial_d(\sigma)$.
    Substituting this in \cref{eq:inner_product_v_sigma}, we conclude
    \begin{align*}
        v_\sigma = - i_\sigma\langle (\Lcal^{\text{up}}_{d-1})^+ \partial_d \sigma, \partial_d \sigma\rangle_{C_{d-1}(K)} = - i_\sigma\langle \partial_d^\adjoint (\Lcal^{\text{up}}_{d-1})^+ \partial_d \sigma, \sigma\rangle_{C_{d}(K)} = -i_\sigma \mathcal{R}_d(\sigma, \sigma).
    \end{align*}
    Hence, the result follows.    
\end{proof}

In~\cite[Theorem 4.2]{KookLee2018}, the authors provided a matrix characterization of their notion of effective resistance (\cref{eq:er-kook-lee}) in terms of the \emph{Green's function}, i.e., the inverse of the Hodge Laplacian matrix for $K$  in the standard bases, that is \(L_{d-1} = L_{d-1, \mathrm{up}} + L_{d-1, \mathrm{down}}\), where $L_{d-1, \mathrm{up}}$ is given in \cref{eq:L_up} and $L_{d-1,\mathrm{down}}$ can be defined similarly. Assuming that the reduced homology in dimension \(d-1\) is trivial (i.e., \(\widetilde{\mathrm{H}}_{d-1}(K, \mathbb{R}) = \{0\}\)), the Hodge theorem guarantees that \(L_{d-1}\) is invertible. Under this assumption, Kook and Lee \cite{KookLee2018} defined 
\begin{equation}\label{eq:R'_p}
    R'_d \coloneqq B_d^\top \, L_{d-1}^{-1} \, B_d.
\end{equation}
For a given \(d\)-simplex \(\tau \in K_d\), the effective resistance defined in \cref{eq:er-kook-lee} is equivalent to the diagonal entry of $R'_d$, that is, \(r'_\tau = R'_d(\tau, \tau)\).

We now show that the matrix \(R'_d\) defined by~\cite{KookLee2018} coincides with the matrix $R_d$ (see \cref{eq:R_p_matrix_original}) introduced by Osting et al.~\cite{OstingPalandeWang2020}. Notice that the latter is formulated in a more general setting (with no assumptions on the invertibility of the Laplacian) and has already been shown to be the matrix representation of our effective resistance bilinear form in the standard bases in \cref{sec:connection_Osting}. 

\begin{proposition}
\label[proposition]{prop:R_d=R'_d}
    When the Hodge Laplacian \(L_{d-1}\) is invertible, and only the $d$-dimensional simplices are weighted, the two resistance matrices agree:
    \(
    R_d = R'_d,
    \)
    or, equivalently,
    \[B_d^\top \left(
    B_d  W_{d} B_d^\top \right)^{+
    } B_d  = B_d^\top(B_d W_d B_d^\top  + B_{d-1}^\top B_{d-1} )^{-1}B_d .  \]    
\end{proposition}

To prove this proposition we first need the following result from \cite{FillFishkind2000}. For a subspace \(\Omega\), let \(P_\Omega\) denote the matrix of the orthogonal projection onto \(\Omega\).

\begin{lemma}[Theorem 3, \cite{FillFishkind2000}] \label[lemma]{thm:presudoinverse-sum}
    Let \(A, C \in \mathbb{R}^{n \times n}\) with \(\mathrm{rank}(A + C) = \mathrm{rank}(A)+ \mathrm{rank}(C)\). Then
    \begin{equation}\label{eq:sum_pseudoinverse}
        (A + C)^{+} = (I - S) A^{+} (I-T) + S C^{+} T,
    \end{equation}
    where $I$ is the identity matrix, \(S\coloneqq \left( P_{\mathrm{col}(C^\top)} P_{\mathrm{col}(A^\top)^{\perp}}\right)^{+}\) and \(T\coloneqq \left( P_{\mathrm{col}(A)^{\perp}} P_{\mathrm{col}(C)}\right)^{+}\).
\end{lemma}

\begin{proof}[Proof of Proposition \ref{prop:R_d=R'_d}]
Let \(A = B_{d} W_d B_{d}^\top \) and \(C = B_{d-1}^\top  B_{d-1} \). \(L_{d-1} = A+C\) is invertible, so \((A+C)^{-1} = (A+C)^{+}\). On the other hand, invertibility implies \(\ker (A+C) = \{0\}\), and from the Hodge decomposition for simplicial complexes we get
\begin{equation}\label{eq:Hodge_decomposition_L_invertible}
    C_{d-1}(K,\mathbb{R}) \simeq \mathrm{col}(A) \oplus \mathrm{col}(C),
\end{equation}
ensuring that \(\mathrm{rank}(A+C) = \mathrm{rank}(A) + \mathrm{rank}(C)\). This means that we can apply \cref{thm:presudoinverse-sum}.

We now show that \(\mathrm{col}(A)^\perp = \mathrm{col}(C)\). Note that, from the vanishing of the composition of consecutive coboundary maps, we have
\(
C A = B_{d-1}^\top B_{d-1} B_d W_d B_d^\top = 0
\).
Let \(x = C x' \in \mathrm{col}(C)\) and \(y = A y' \in \mathrm{col}(A)\). Then, using the symmetry of \(A\) and \(C\), and the inner product on $C_{d-1}(K)$, which is the standard $\ell^2$ product since $W_{d-1}$ is the identity matrix, we compute:
\[
\langle x, y \rangle_{C_{d-1}} = x^\top  y  = (x')^\top C A y' = 0.
\]
Hence, \(\mathrm{col}(C) \subseteq \mathrm{col}(A)^\perp\). 
\cref{eq:Hodge_decomposition_L_invertible} implies that both $\mathrm{col}(C)$ and $\mathrm{col}(A)^\perp$ have the same dimension, and thus \(\mathrm{col}(C) = \mathrm{col}(A)^\perp\).

Therefore, we have
\begin{align*}
    \left( P_{\mathrm{col}(C^\top)} P_{\mathrm{col}(A^\top)^{\perp}}\right)^{+} 
      & =  \left( P_{\mathrm{col}(C)} P_{\mathrm{col}(A)^{\perp}}\right)^{+} & (C = C^\top, \ A = A^\top) \\
      & =  \left( P_{\mathrm{col}(C)}^2\right)^{+} & (\mathrm{col}(C) = \mathrm{col}(A)^\perp)\\
      & =  \left( P_{\mathrm{col}(C)}\right)^{+}  & \text{(projection matrix is idempotent)} \\
      & = \left( C C^+\right)^{+} & \text{(properties of pseudoinverses)}\\
      & =  C C^+ & \text{(properties of pseudoinverses)}
\end{align*}
and similarly, \( \big( P_{\mathrm{col}(A)^{\perp}} P_{\mathrm{col}(C)}\big)^{+} = C C^+\). Applying \cref{eq:sum_pseudoinverse}, we obtain:
\begin{align}
    (A + C)^{+}
    &= (I - C C^+) A^{+} (I - C C^+) + C C^+ C^{+} C C^+ \notag \\
    &= A^+ - A^+ C C^+ - C C^+ A^+ + C C^+ A^+ C C^+ + C C^+ C^+. \label{eq:pseudoinverse_laplacian}
\end{align}

Next, we observe that \(B_d^\top C = B_d^\top B_{d-1}^\top B_{d-1} = 0\) due to the composition of consecutive coboundary maps vanishing. In addition, from the properties of the pseudoinverse, we get:
\[
C^+ B_d = (C^\top C)^+ C^\top B_d = (C^\top C)^+ B_{d-1}^\top B_{d-1} B_d = 0.
\]

Therefore, when multiplying both sides of~\cref{eq:pseudoinverse_laplacian} on the left by \(B_d^\top\) and on the right by \(B_d\), all terms vanish except the first, and we conclude the desired result:
\[
R_d' = B_d^\top (A + C)^{-1} B_d = B_d^\top (A + C)^{+} B_d = B_d^\top A^{+} B_d = R_d.\qedhere
\]
\end{proof}

\subsection{Connection to the Definition of Black and Maxwell} 
\label{sec:connection_black_maxwell}

Black and Maxwell~\cite{BlackMaxwell2021} worked with weighted simplicial complexes, but in their definition of effective resistance they assigned weights only to the $p$-simplices, assuming unit weights on the $(p-1)$-simplices. This assumption is consistent with the settings in~\cite{KookLee2018,OstingPalandeWang2020}, but it is more restrictive than our formulation, which accommodates weights on both $p$-simplices and $(p-1)$-simplices.
Authors in \cite{BlackMaxwell2021} used the up Laplacian matrix $\Lupmatrix{p-1}{sym} = B_pW_pB_p^\top$ (see \cref{eqn:relative-er}) and defined the effective resistance of any $(p-1)$-cycle $\gamma$ to be 
\[
    r''_{\gamma}\coloneqq 
    \begin{cases}
    [\gamma]_{\caB_{p-1}}^\top (\Lupmatrix{p-1}{sym})^+ [\gamma]_{\caB_{p-1}}, & \text{if } \gamma \text{ is a boundary}; \\
    \infty, & \mathrm{otherwise};
    \end{cases}
\]
where $[\gamma]_{\caB_{p-1}}$ is the vector representation of $\gamma$ in the standard basis \(\caB_{p-1}\) of $(p-1)$-chains.

We now relate their definition to that of Osting et al.~\cite{OstingPalandeWang2020}, and hence to ours. If \(\gamma\) is a boundary, then \(\gamma = \partial_p \beta\) for some \(p\)-chain \(\beta\). In this case,
\begin{align*}
    r''_{\gamma} 
    & = [\partial_p \beta]_{\caB_{p-1}}^\top (\Lupmatrix{p-1}{sym})^+ [\partial_p \beta]_{\caB_{p-1}} & \text{(definition of } r''_{\gamma}\text{)}\\
    & = [\beta]_{\caB_{p}}^\top B_p^\top (\Lupmatrix{p-1}{sym})^+ B_p [\beta]_{\caB_{p}} & \text{(matrix form of $\partial_p$)}\\
    & = [\beta]_{\caB_{p}}^\top R_p [\beta]_{\caB_{p}} = r_\beta. & \text{(\cref{eq:R_p_matrix_original})} 
\end{align*}

Thus, Black and Maxwell’s effective resistance agrees with ours on boundaries: what they call the effective resistance of the boundary \(\gamma = \partial_p \beta\) is exactly what we call the effective resistance of the \(p\)-chain \(\beta\) (or, the boundary effective resistance of $\gamma$).

\section{Conclusion}
\label{sec:conclusion}

In this paper, we introduce a basis-independent definition of effective resistance on the space of (co)chains of a simplicial complex via an effective-resistance bilinear form, and we show that previous matrix-based definitions of effective resistance on simplicial complexes arise as special cases of this framework under appropriate choices of basis. This formulation satisfies an additional desirable property: taking square roots yields a pseudometric on the space of chains and a metric within each homology class of cycles, while effective resistance itself defines a metric on suitable finite sets of chains, recovering the classical resistance metric on graph vertices.

We further extend Thomson’s principle to our definition of effective resistance, and establish a generalized Foster’s theorem, relating the eigenvalues of the effective-resistance operator to the rank of the adjoint boundary operator. In addition, we prove a weighted matrix–tree theorem, and use it to establish a relationship between the relative effective resistance of a simplex and its probability of appearing in a random spanning tree under an appropriate sampling distribution.

Looking ahead, we expect that this basis-independent formulation of effective resistance will support future work on analyzing the structural organization of simplicial complexes and on developing machine-learning methods for simplicial domains, where resistance may inform notions of similarity, diffusion, or label propagation across higher-order structures; as it is already used in graphs.

\bibliography{refs-er.bib}

\appendix
\section{List of Notations}
\label{sec:notations}

\begin{adjustbox}{max width=\textwidth}
\begin{tabular}{ll}
\toprule
\rowcolor{gray!10} Adjoint and dual of linear operators& \\
$U, V$ & vector spaces over a field $\Bbbk$ (typically, $\Bbbk=\Rspace$)\\
$U^\dual$ & dual vector space\\ 
$f \colon U \to V$ & a linear operator\\
$f^\adjoint \colon V \to U$ & adjoint of a linear operator \\
$f^\dual \colon V^\dual \to U^\dual$ & dual of a linear operator\\
\midrule
$\Bcal_U, \Bcal_V$ & Bases of $U$ and $V$ respectively\\
$[f]^{\Bcal_V}_{\Bcal_U}$ & matrix representation of $f$ w.r.t. bases $\Bcal_U$ and $\Bcal_V$\\
\midrule
$\varphi: U \times V \to \Bbbk$ & a bilinear form\\
$[\varphi]_{\Bcal_U, \Bcal_V}$ & matrix representation of $\varphi$ w.r.t. bases $\Bcal_U$, $\Bcal_V$ and $\{1\}$\\
\midrule
\rowcolor{gray!10} Simplicial homology& \\
$K$ & a $d$-dimensional simplicial complex \\
$K_p$ & $p$-simplices in $K$\\
$n_p$ & number of $p$-simplices in $K$\\
$w_p \colon K_p \to \mathbb{R}_{>0}$& weight function on the $p$-simplices of $K$\\
$W_p$ & diagonal weight matrix\\
$C_p(K; \Bbbk)$ & $p$-chain group of $K$ with coefficients in $\Bbbk$\\
$C^p(K; \Bbbk)$ & $p$-cochain group of $K$ with coefficients in $\Bbbk$\\
$\partial_p \colon C_p(K, \Bbbk) \to C_{p-1}(K, \Bbbk)$ & boundary operators\\
$\delta_p \colon C^p(K, \Bbbk) \to C^{p+1}(K, \Bbbk)$ & coboundary operators\\
\midrule
$\mathcal{B}_p$ & standard basis in the space of chains\\
$\mathcal{B}^p$ & standard basis in the space of cochains\\
$\tilde{\mathcal{B}}_p$ & orthonormal basis in the space of chains\\
$\tilde{\mathcal{B}}^p$ & orthonormal basis in the space of  cochains\\
\midrule
$\langle \cdot, \cdot\rangle_{C_p}$ & inner product endowed on the space of chains\\
$\langle \cdot, \cdot\rangle_{C^p}$ & inner product endowed on the space of cochains\\
$\flat_p \colon C_p(K) \to C^p(K)$ & flat isomorphism\\
$\sharp_p \colon C^p(K) \to C_p(K)$ & sharp isomorphism\\ 
\midrule
\rowcolor{gray!10} Laplacian operators & \\
$\cochainLup{p} \colon C^p(K) \to C^p(K)$ & up Laplacian operator on cochain spaces\\
$\cochainLdown{p} \colon C^p(K) \to C^p(K)$ & down Laplacian operator on cochain spaces\\
$\chainLup{p} \colon C_p(K) \to C_p(K)$ & up Laplacian operator on chain spaces\\
$\chainLdown{p} \colon C_p(K) \to C_p(K)$ & down Laplacian operator on chain spaces\\
\midrule
\rowcolor{gray!10} Effective resistance bilinear forms& \\
$\Rcal_p: C_p(K) \times C_p(K) \to \Rspace$ & effective resistance bilinear form on chain spaces\\
$\Rcal^p: C^p(K) \times C^p(K) \to \Rspace$ & effective resistance bilinear form on cochain spaces\\
\bottomrule
\end{tabular}
\end{adjustbox}

\end{document}